\documentclass[twoside,11pt]{article}
\usepackage{amssymb,amsmath,amsthm,amsfonts, epsfig, esint}           
\usepackage{mathtools} 
\usepackage{hyperref}
\usepackage{eucal}
\usepackage{a4wide}
\usepackage[svgnames]{xcolor}
\usepackage{mathptmx} 
\usepackage{graphicx}
\usepackage{tikz-cd}
 \usepackage{fancyhdr}
 \usepackage{MnSymbol}
 \usepackage{subcaption}

\title{Determination of stable branches of relative equilibria of the $N$-vortex problem on the sphere}

\author{Constantineau,~K.\footnote{Department of Mathematics and Statistics, McGill University, 805 Sherbrooke West, Montreal, QC H3A 0B9, Canada. kevin.constantineau@mail.mcgill.ca},\, Garc\'{\i}a-Azpeitia,~C.\footnote{Departamento de Matem\'aticas y Mec\'anica, 
IIMAS-UNAM.
Apdo. Postal 20-126, Col. San \'Angel,
Mexico City, 01000,  Mexico. cga@mym.iimas.unam.mx. ORCID: 0000-0002-6327-1444.}, \, Garc\'{\i}a-Naranjo,~L.~C\footnote{Dipartimento di Matematica ``Tullio Levi-Civita", Universit\`a di Padova, Via Trieste 63, 35121 Padova, 
Italy. luis.garcianaranjo@math.unipd.it. ORCID: 0000-0002-3589-6068.} \, and  J.-P.~Lessard\footnote{Department of Mathematics and Statistics, McGill University, 805 Sherbrooke West, Montreal, QC H3A 0B9, Canada. jp.lessard@mcgill.ca}}

\numberwithin{equation}{section}
\numberwithin{table}{section}
\numberwithin{figure}{section}

\hypersetup{colorlinks=true,linkcolor=violet,citecolor=purple}

\newtheorem{theorem}{Theorem}[section]
\newtheorem{lemma}[theorem]{Lemma}
\newtheorem{proposition}[theorem]{Proposition}
\newtheorem{corollary}[theorem]{Corollary}

\theoremstyle{definition}
\newtheorem{definition}[theorem]{Definition}
\newtheorem{remark}[theorem]{Remark}
\newtheorem*{remarks*}{Remarks}

\newcommand{\bydef}{:=}

\newtheorem{innercustomgeneric}{\customgenericname}
\providecommand{\customgenericname}{}
\newcommand{\newcustomtheorem}[2]{%
  \newenvironment{#1}[1]
  {%
   \renewcommand\customgenericname{#2}%
   \renewcommand\theinnercustomgeneric{##1}%
   \innercustomgeneric
  }
  {\endinnercustomgeneric}
}

\newcustomtheorem{customhyp}{}

\newcommand{\defn}[1]{{\bfseries\itshape{#1}}}

\def\headcolour{\color{Grey}}
\def\restr#1{\,\vrule height1.2ex width.4pt
  depth0.8ex\lower0.4ex\hbox{\scriptsize $\,#1$}}
\newcommand{\C}{\mathbb{C}}
\newcommand{\R}{\mathbb{R}}   
\newcommand{\Z}{\mathbb{Z}}   

\newcommand{\I}{\mathbb{I}}

\newcommand{\so}{\mathfrak{so}}

\newcommand{\D}{\mathcal{D}}

\begin{document}

\maketitle

\begin{abstract}
We consider the $N$-vortex problem on the sphere assuming that all  vorticities
have equal strength. We investigate 
relative equilibria (RE) consisting of $n$ latitudinal  rings which are  uniformly rotating 
about the vertical axis with angular velocity $\omega$. Each such ring contains $m$ vortices placed  at the vertices of a concentric 
regular polygon and  we allow the presence of additional  vortices  at the poles. We develop a framework to prove existence
and orbital stability of branches of RE of this type parametrised by $\omega$. Such framework is  implemented 
to  rigorously determine and prove stability of segments of  branches using computer-assisted proofs.
This approach circumvents the analytical complexities that arise when the number of rings  $n\geq 2$ and allows us to give several new rigorous results. 
We  exemplify our method providing new contributions consisting in the determination of enclosures  
and proofs of stability of several equilibria and RE
 for $5\leq N\leq 12$.
\end{abstract}

\vspace{0.5cm}
\noindent
{\em Keywords:} $N$-vortex problem, relative equilibrium, stability, interval arithmetic, computer-assisted proofs. \\
{\em 2020 MSC}:  70K42,  76M60, 65G30, 65G20, 47H10, 37C25.

\section{Introduction}
In the last  decades  the 
 $N$-vortex problem on the sphere  has received considerable attention.
 A non-exhaustive list of references contains \cite{Borisov1,Borisov2,Kidambi1,Marsden-Pekarski,Borisov4, SouliereTokieda, Borisov3,Meleshko-Newton,Vankerschaver, Renato, Carlos, Ohsawa} and other papers that are mentioned below. The equations of motion  go back  to Gromeka \cite{Gromeka} and Bogomolov \cite{Bogomolov}. 
 The importance of the system is usually associated with geophysical fluid dynamics since it provides a simple model for the dynamics
 of cyclones and hurricanes in planetary atmospheres. Moreover, 
 recent numerical work  \cite{ModinViviani2} suggests  that the $N$-vortex  problem on the sphere plays a crucial role in 
 the long term behaviour of 
the incompressible two-dimensional Euler equations on the same domain. We also mention that the $N$-vortex problem on the 
sphere equipped with a general Riemannian metric was recently studied in \cite{Wang} and, as noted in \cite{KoillerCastilhoRodrigues,Rodrigues-Koiller},
  this is a convenient approach to treat the $N$-vortex problem
on more general closed surfaces of genus $0$. We refer the reader to the 
book \cite{NewtonBook}  and the  papers \cite{Aref1,Aref2} for  an overview  and extensive bibliography on  vortex dynamics.

The equations of motion  of the problem define a Hamiltonian system on the $2N$-dimensional phase
space $M$ which is obtained as the cartesian product of $N$ copies of the 2-sphere $S^2$ minus the
collision set. The  Hamiltonian function  $H:M\to \R$ accounts for the pairwise interaction between
the vortices, and the  symplectic form $\Omega$ on $M$ is a weighted sum by the vortex
strengths of  the area form on each copy of $S^2$.  A fundamental aspect of the problem is that both $H$ and
$\Omega$ are  invariant under the   action of $SO(3)$ on $M$ that simultaneously rotates all vortices. 
As a consequence, the equations of motion are  $SO(3)$-equivariant and, in accordance with
Noether's Theorem,  there exists a momentum map $\Phi : M\to \R^3$, whose
components are first integrals. 
It is well-known (see e.g. \cite{ModinViviani1}) that the problem is integrable if $N\leq 3$, 
and on the zero level set of $\Phi$ if $N= 4$, and 
appears to be non-integrable otherwise.

 The simplest solutions to the problem are the equilibrium 
configurations which correspond to the critical points of $H$.
 If all the vortex strengths  have the same sign, then $H$ is bounded from below and it assumes a global minimum on $M$. The 
minimising configurations are called \defn{ ground states} and 
form a very  important kind of stable equilibria. 
Unfortunately, the  rigorous determination of the ground states  is an extremely   complicated
problem for which very little is known. If all vortices  have equal strengths, which is the case that we treat in most  of this work,
after a  suitable normalisation, the Hamiltonian $H:M\to \R$ becomes  
\begin{equation*}
H(v)=-\sum_{i<j} \ln \left ( \left \| v_i -v_j  \right \|^2 \right ), \qquad v=(v_1,\dots, v_N)\in M,
\end{equation*}
where, for all $i=1,\dots, N$,  $v_i$ is a unit vector on $\R^3$ which specifies  the position of the $i^{th}$ vortex  on the unit sphere $S^2$ (note that  $v_i\neq v_j$ for all $i$, $j$, since collisions have been removed).
Using the above expression for $H$, it is easy to see that the ground state
configurations are geometrically characterised as those which maximise the
product of the pairwise euclidean distances between the vortices. These configurations
are sometimes called   \defn{Fekete points}, and their determination corresponds to Smale's problem $\#7$ which has only been
solved for a few values of $N$ (2, 3, 4, 5, 6 and 12). We refer the reader to  \cite{Armentano11},  the nice review
 \cite{Beltran20}  and the references therein for more information on Smale's problem $\#7$.
In Table \ref{ground-states} below we list the known Fekete points
together with a conjecture\footnote{%
In section \ref{ss:CAP-minimisers} we show that these are at least local
non-degenerate minima of $H$ for the values of $N=8,9,10,11$. We also show in
section \ref{ss:CAP-RE-groundstates} that the configuration for $N=7$, which
is widely conjectured to be the minimiser \cite{Beltran20}, is degenerate in the sense of Bott \cite{Bott54}.} of their position for the values $N=7,\dots,11$ (the 
detailed
description of the configurations in the table is given in section \ref{ss:CAP-minimisers}). The configurations of the table 
are illustrated in Figures \ref{Fig:GS-Known} and \ref{Fig:GS-Conjecture}. One can analytically
show that these  are critical points of $H$ except for $N=11$ whose treatment,
as explained in section \ref{ss:CAP-minimisers}, 
 required the use of a computer-assisted proof.

\begin{table}[h]
\begin{center}
\begin{tabular}{|l|l|l|}
\hline
$N$ & Polyhedron & $\mathbb{Z}_{m}$ -symmetries \\ \hline
2 & antipodal points & $\mathbb{Z}_{2}$ \\ \hline
3 & equilateral triangle & $\mathbb{Z}_{3}$ \\ \hline
4 & tetrahedron & $\mathbb{Z}_{2}$ and $\mathbb{Z}_{3}$ \\ \hline
5 & triangular bipyramid & $\mathbb{Z}_{2}$ and $\mathbb{Z}_{3}$ \\ \hline
6 & octahedron & $\mathbb{Z}_{2}$, $\mathbb{Z}_{3}$ and $\mathbb{Z}_{4}$ \\ 
\hline
7 & pentagonal bipyramid & $\mathbb{Z}_{2}$ and $\mathbb{Z}_{5}$ \\ \hline
8 & square antiprism & $\mathbb{Z}_{2}$ and $\mathbb{Z}_{4}$ \\ \hline
9 & triaugmented triangular prism & $\mathbb{Z}_{2}$ and $\mathbb{Z}_{3}$ \\ 
\hline
10 & gyroelongated square bipyramid & $\mathbb{Z}_{2}$ and $\mathbb{Z}_{4}$
\\ \hline
11 &-- $\quad$ (see section  \ref{ss:CAP-minimisers} for details) & $\mathbb{Z}_{2}$ \\ \hline
12 & icosahedron & $\mathbb{Z}_{2}$, $\mathbb{Z}_{3}$ and $\mathbb{Z}_{5}$
\\ \hline
\end{tabular}
\end{center}
\caption{Known ($N=2,3,4,5, 6, 12$) and conjectured ($N=7,8,9,10,11$) ground states of the $N$ vortex
problem on the sphere assuming identical vortex strengths. See section  \ref{ss:CAP-minimisers} for details.}
\label{ground-states}
\end{table}

Another type of fundamental solutions, whose study is the main topic of this work,
 is given by  relative equilibria (RE). For our specific problem, these are
periodic solutions
in which all vortices  rotate uniformly about a fixed axis at constant angular speed $\omega\in \R$. 
These solutions correspond to critical points of $H$ restricted to the level sets of $\Phi$.
Without loss of generality, the rotation axis is chosen as   the vertical $z$-axis in $\R^3$ throughout this work. 


 An important class of RE is that in which the vortices
are arranged  in $n$ latitudinal rings, each of which  consists  of a regular polygon of $m$ identical vortices. We allow for the possibility of having   
 $p$ vortices at the poles ($p=0,1$ or $2$) so the total
number of vortices  $N=mn+p$. Such RE possess a discrete $\Z_m$-symmetry and
 will be thus called \defn{ $\Z_m$-symmetric RE}.  Several  previous works
have focused on the study of this type of RE \cite{PolvDritschel,Borisov4,MR1811389,LaurentPolz1,CabralMeyerSchmidt,BoattoCabral,Mo11}. In particular, the seminal work of Lim, Montaldi and Roberts
\cite{MR1811389} 
 gives several existence results based on symmetry considerations. 
However, explicit
analytic expressions and stability results for these RE are mainly known 
for the case of $n=1$ ring (see \cite{Mo11} for a comprehensive list of results in this case). When the number of rings $n\geq 2$, the computational
complexity enormously  increases and  one has to settle with numerical investigations as in \cite{Mo11}.  This is somewhat unsettling, 
considering that the $\Z_m$-symmetric RE consisting   of only one ring  lose stability as the number of vortices   
$m$ in each  ring grows (see e.g. \cite{Mo11} and our discussion in section \ref{ss:RE-totalcollision} for precise statements). 

In this work we circumvent the difficulty
to obtain rigorous existence and stability results for $\Z_m$-symmetric RE having  $n\geq 2$ rings  by relying on computer-assisted proofs (CAPs). 
This approach  falls in the category of CAPs in dynamics, which is by now a
well-developed field, with some famous early pioneering works being the
proof of the universality of the Feigenbaum constant \cite{MR648529} and the
proof of existence of the strange attractor in the Lorenz system \cite%
{MR1870856}. We refer the interested reader to the survey papers \cite%
{MR1420838, MR1849323,MR2652784,MR3444942,MR3990999,MR4283203}, as well as
the books \cite{MR2807595,MR3822720,MR3971222}.

\subsection{Contributions}
Our work focuses on the case of equal vortex strengths and develops a theoretical framework for
existence and stability of branches of  $\Z_m$-symmetric\footnote{The 
value of $m$ can be chosen as $1$ corresponding to general asymmetric RE.}   RE parametrised by the angular speed 
$\omega$. This framework
is  implemented in an   INTLab code (available in \cite{KevinCode}) which establishes existence
and provides rigorous bounds for (segments of) the branch using   CAPs. Moreover, the code
 also performs a (nonlinear) stability test which may be validated with a CAP. 
 
 The input for the code
 is a numerical approximation of a non-degenerate RE (configuration and angular speed) with  a
 prescribed $\Z_m$-symmetry. Provided that the given RE approximation has sufficient
 precision and  is far  from bifurcations, 
 the code returns an enclosure of a segment of the branch containing the approximated RE
 and possessing  the prescribed  symmetry. The
 stability procedure involves validation of the positivity of the spectrum of
 a certain matrix and  limitations  
 arise in the presence of eigenvalues which are either too close to zero or which cluster.
 
The fundamental aspects of our theoretical framework and its CAP-implementation are summarised below. 
We divide the presentation of our results on existence, stability and CAPs.

\subsubsection*{Existence.}

%

The first step in our construction is to perform a discrete reduction of the system by $\Z_m$ leading to a  reduced Hamiltonian
system with an $SO(2)$-symmetry. This is the content 
of  Theorem  \ref{th-main-symmetry} whose formulation  is inspired by the previous work \cite{Luis} of the second and third
author. This discrete reduction relies on the permutation symmetry of the vortices and 
corresponds to the restriction of the dynamics to the invariant  submanifold of $M$ formed by $\Z_m$-symmetric configurations\footnote{These are 
generic configurations (not necessarily RE) in which 
 the vortices are organised  in regular $m$-gons on  $n$ latitudinal rings
and possibly in the  presence of $p$ vortices at the poles ($p=0,1,2$) so $N=mn+p$. The $\Z_m$-symmetries of each of the equilibria 
 in Table \ref{ground-states} is indicated in the last column and is 
  illustrated in  Figures \ref{Fig:GS-Known} and \ref{Fig:GS-Conjecture}. In each case,  it is easy to determine the corresponding values of $n$ and $p$ from the figure.}. Theorem  \ref{th-main-symmetry} shows that the dynamics in this invariant manifold is conjugate to the
dynamics of a   symplectic Hamiltonian system
 on a  $2n$-dimensional {\em reduced phase space}, that we denote $M_n$, and whose {\em reduced Hamiltonian function} $h:M_n\to \R$ is given explicitly by \eqref{eq:h(u)}. 
 Moreover, the theorem also shows that this reduced Hamiltonian system possesses an $SO(2)$-symmetry
 corresponding to the simultaneous rotation of the rings. 

The  discrete reduction allows us to cast 
 the problem  of determination of  branches of $\Z_m$-symmetric RE  parametrised by $\omega$ in terms of the 
 determination of branches of critical points of the {\em augmented Hamiltonian} of the reduced system, which is
the function $h_\omega: M_n\to \R$ depending parametrically on $\omega$ given by $h_\omega=h-\omega \phi$, where $\phi:M_n\to \R$ 
is the momentum map of the $SO(2)$-action on $M_n$. After introducing Lagrange multipliers to 
embed $M_n$ in $\R^{3n}$, and introducing an  {\em unfolding parameter} as  in  \cite{MunozFreireGalanetal,Ize} to
deal with  an $SO(2)$-degeneracy coming from the symmetries, we  reduce  the problem of determining critical points of $h_\omega$ 
 to finding zeros of a suitable map 
 \begin{equation}
 \label{eq:F-intro}
F: \R^{d+1}\to \R^{d}, \qquad (x,\omega)\mapsto F(x,\omega).
\end{equation}
The map $F$ involves the derivatives of $h_\omega$ and its explicit form, and the dimension $d$, may be read from \eqref{eq:augmap}. 
 Under the hypothesis that  a non-degenerate zero of $F$ is known (corresponding to a  $\Z_m$-symmetric RE of our problem which is non-degenerate in a sense that
 we make precise in the text) the existence 
 of a local branch of $\Z_m$-symmetric RE parametrised
 by $\omega$ is then guaranteed by the implicit function theorem.  The existence results  
 that we have just outlined  are formalised in Theorem \ref{thm:continuation} and Corollary \ref{cor:existenceREsymmetric}. In particular, the corollary implies the existence of
 local branches of $\Z_m$-symmetric RE emanating from each $\Z_m$-symmetry of the 
 equilibria illustrated in  Figures \ref{Fig:GS-Known} and \ref{Fig:GS-Conjecture} having  $N\neq 7$.\footnote{Due to
 the degeneracy  of the $N=7$ configuration (mentioned in a previous footnote) the existence of branches of RE cannot be concluded from the corollary if $N=7$. Despite this, we 
 analytically determine a $\Z_5$-symmetric branch of RE emanating from the symmetry depicted in Figure \ref{Fig:GSN7Z5} (containing only 
 $1$ ring) and determined a region of  nonlinear stability for it in section \ref{ss:CAP-RE-groundstates}.}  These RE are obtained via  a suitable vertical displacement of the rings, which results in their  uniform rotation with small angular frequency 
 about the vertical axis. 
 
 Finally we  indicate that   this paper also contains an original 
  general  existence result which is not used in our framework but provides a 
  groundwork for our investigation. It is valid for arbitrary vortex strengths and shows existence of two local branches of 
  RE parametrised by $\omega$ emanating
from any non-degenerate equilibrium (Theorem \ref{th:existenceRE}). 
\begin{figure}[ht]
\centering
\begin{subfigure}{.22\textwidth}
  \centering
  \includegraphics[width=.9\linewidth]{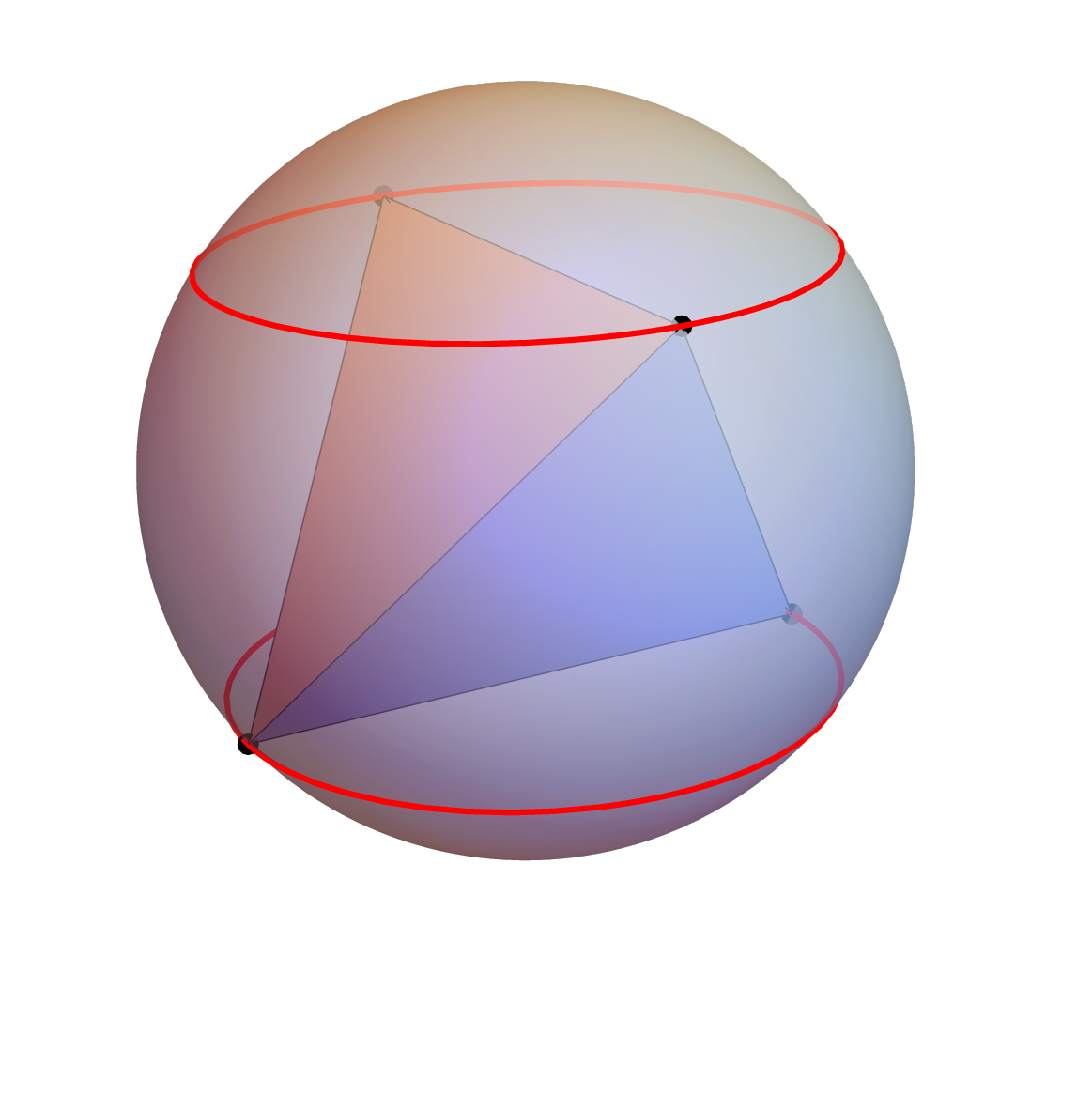}  
  \caption{$N=4$, $\mathbb{Z}_2$}
  \label{Fig:GSN4Z2}
\end{subfigure} 
\begin{subfigure}{.22\textwidth}
  \centering
  \includegraphics[width=.9\linewidth]{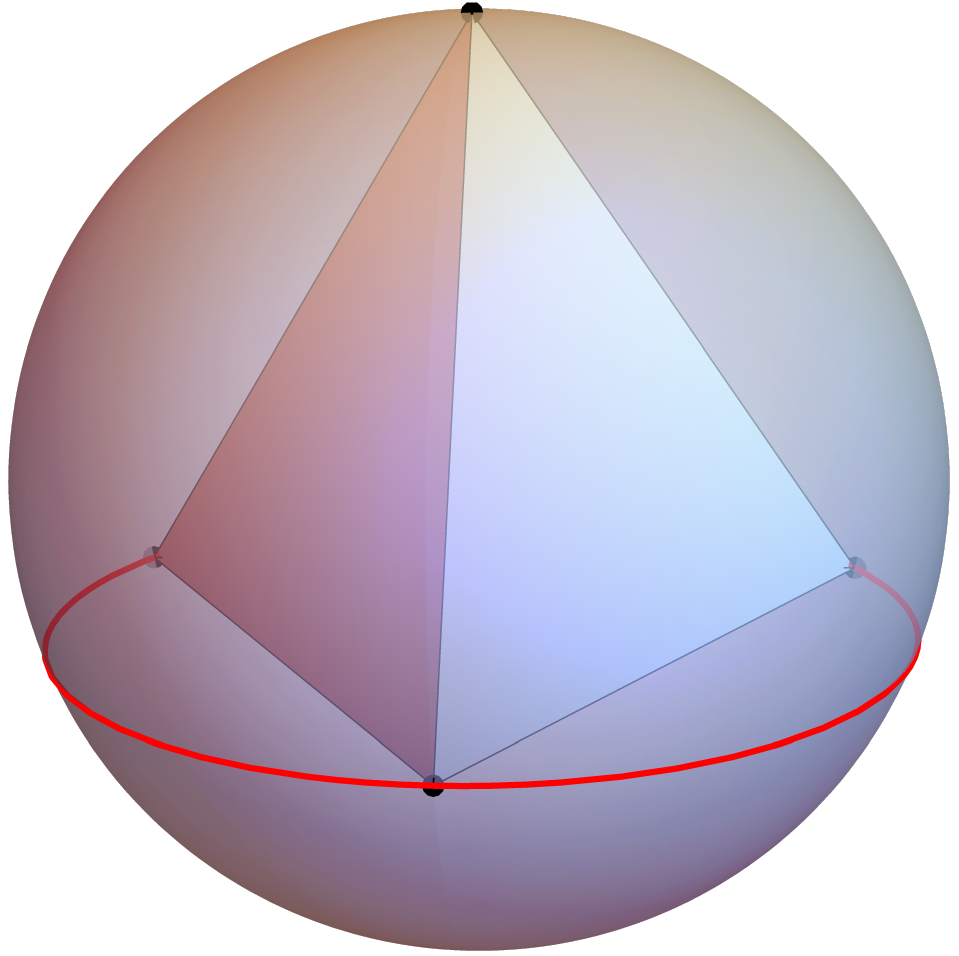}
  \caption{$N=4$, $\mathbb{Z}_3$}
  \label{Fig:GSN4Z3}
\end{subfigure} \qquad
\begin{subfigure}{.22\textwidth}
  \centering
  \includegraphics[width=.9\linewidth]{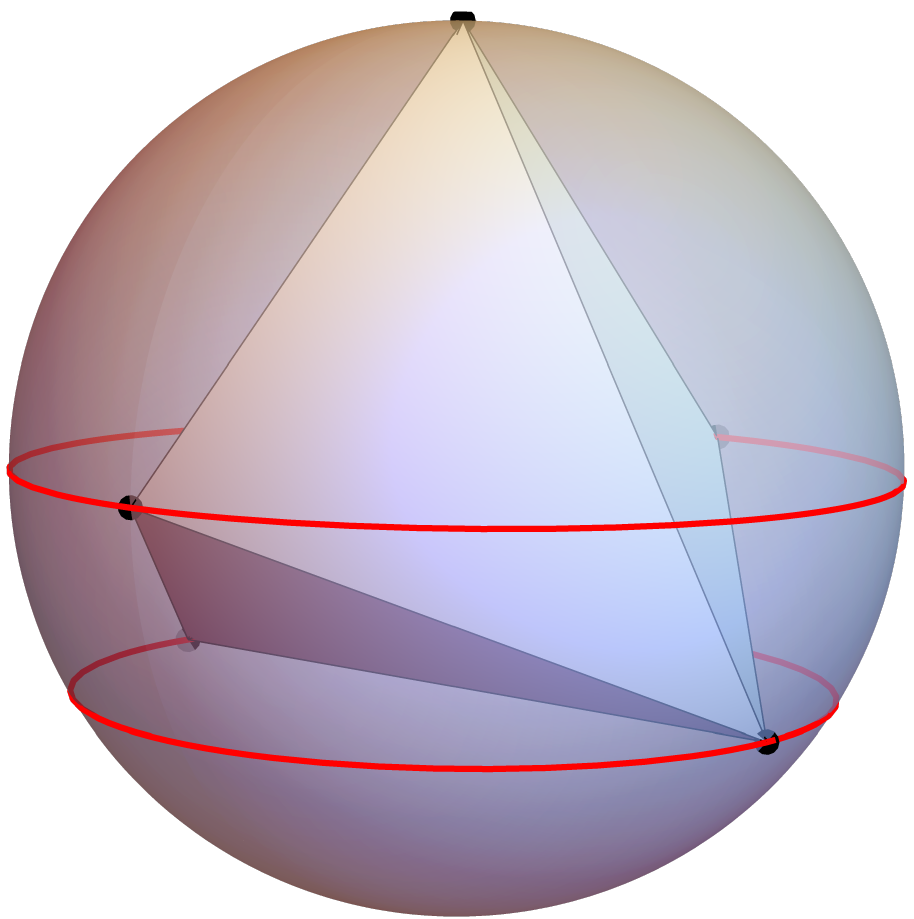}  
  \caption{$N=5$, $\mathbb{Z}_2$}
  \label{Fig:GSN5Z2}
\end{subfigure} 
\begin{subfigure}{.22\textwidth}
  \centering
  \includegraphics[width=.9\linewidth]{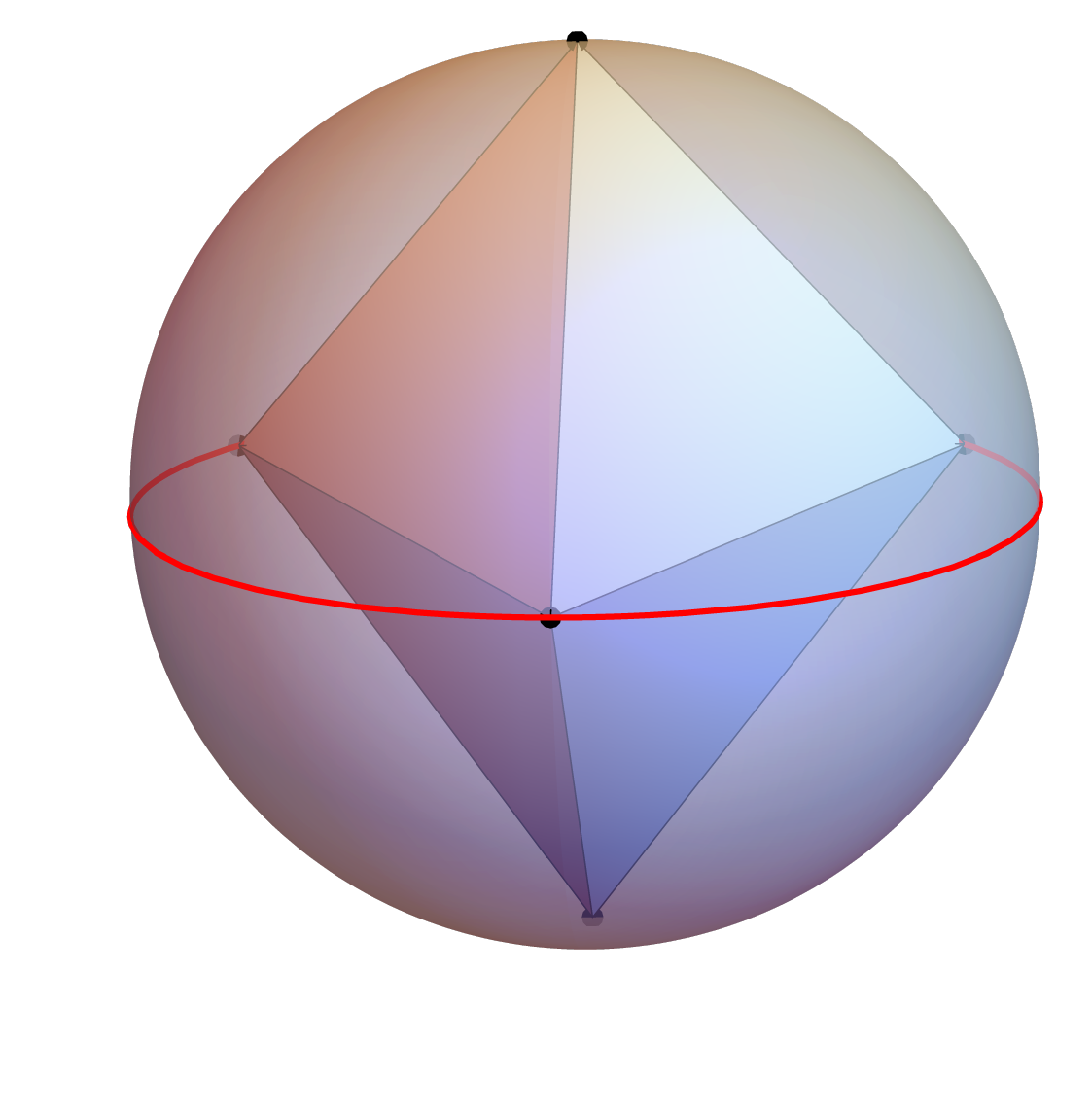}
  \caption{$N=5$, $\mathbb{Z}_3$}
  \label{Fig:GSN5Z3}
\end{subfigure}  \\
\begin{subfigure}{.22\textwidth}
  \centering
  \includegraphics[width=.9\linewidth]{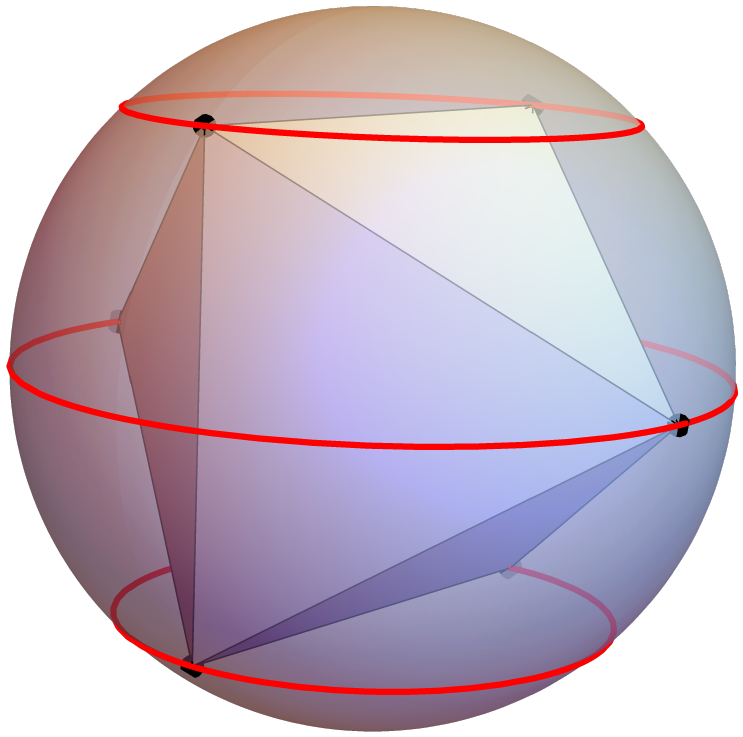}  
  \caption{$N=6$, $\mathbb{Z}_2$}
  \label{Fig:GSN6Z2}
\end{subfigure}
\quad
\begin{subfigure}{.22\textwidth}
  \centering
  \includegraphics[width=.9\linewidth]{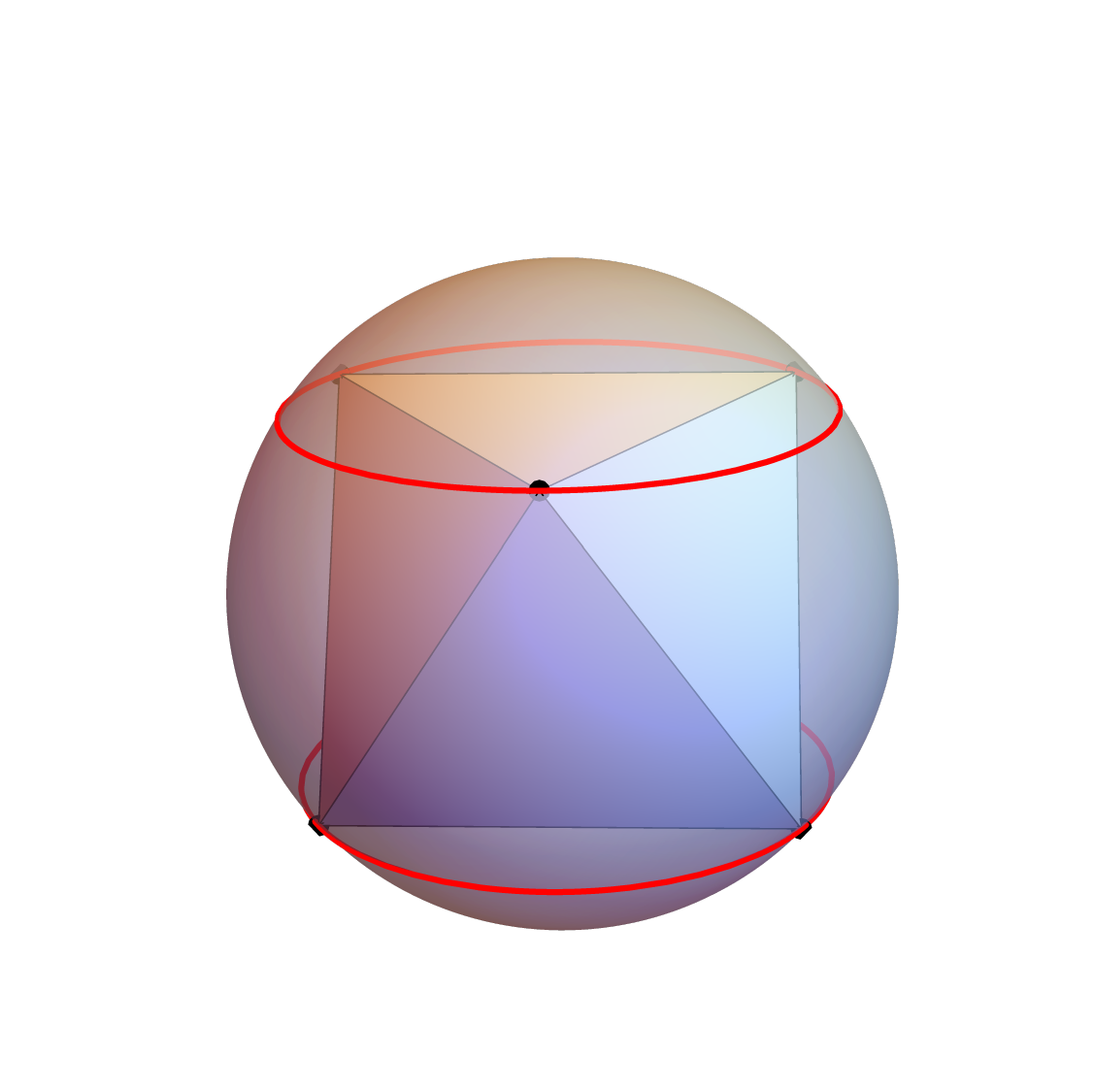}  
  \caption{$N=6$, $\mathbb{Z}_3$}
  \label{Fig:GSN6Z3}
\end{subfigure}
\quad
\begin{subfigure}{.22\textwidth}
  \centering
  \includegraphics[width=.9\linewidth]{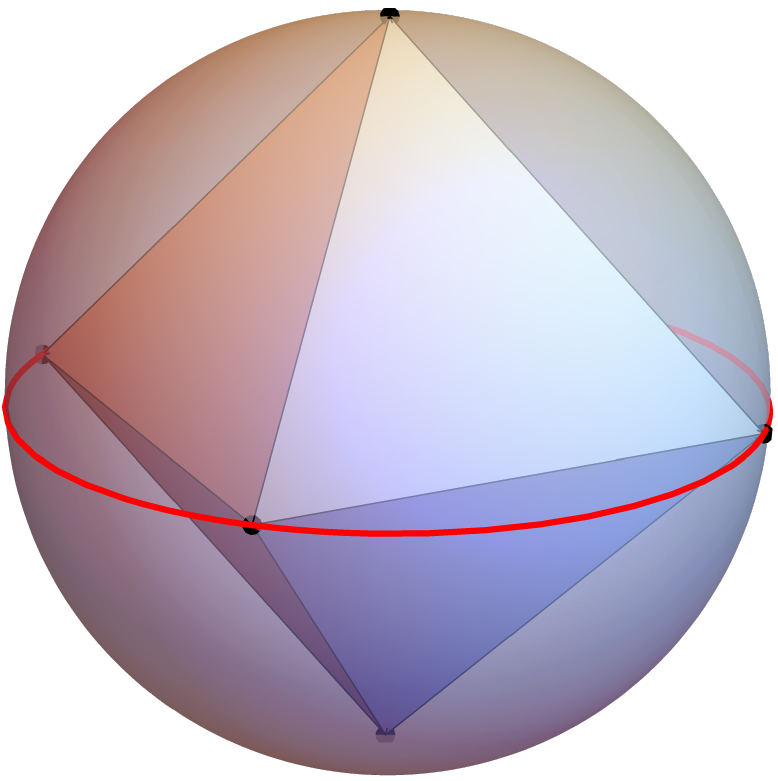}  
  \caption{$N=6$, $\mathbb{Z}_4$}
  \label{Fig:GSN6Z4}
\end{subfigure} \\
\begin{subfigure}{.22\textwidth}
  \centering
  \includegraphics[width=.9\linewidth]{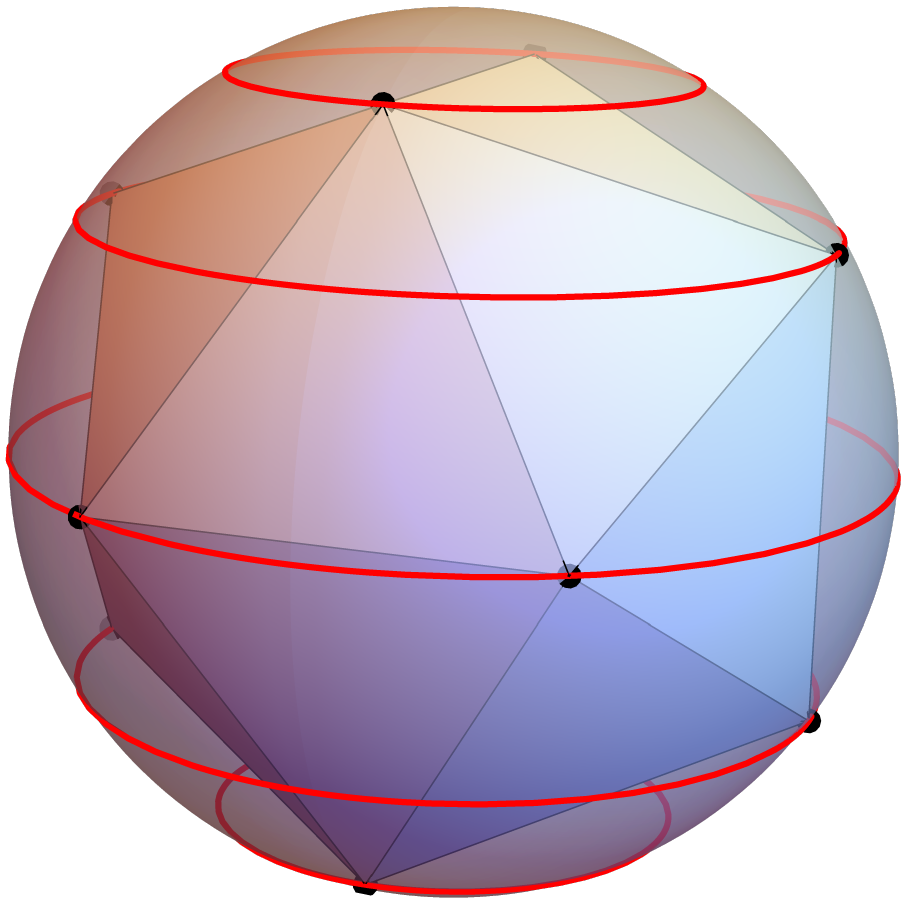}  
  \caption{$N=12$, $\mathbb{Z}_2$}
  \label{Fig:GSN12Z2}
\end{subfigure}
\quad
\begin{subfigure}{.22\textwidth}
  \centering
  \includegraphics[width=.9\linewidth]{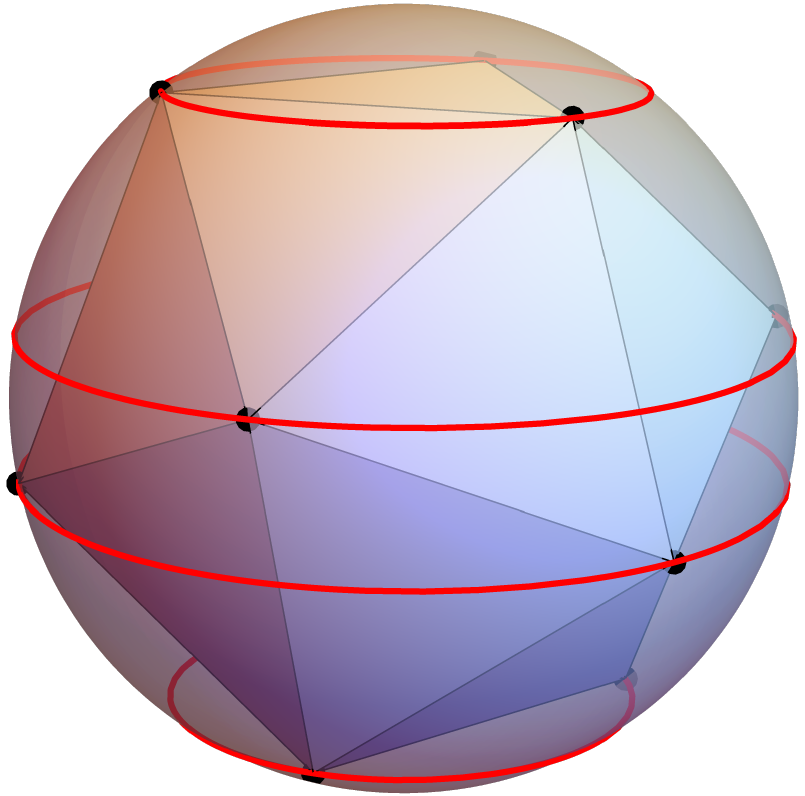}  
  \caption{$N=12$, $\mathbb{Z}_3$}
  \label{Fig:GSN12Z3}
\end{subfigure}
\quad
\begin{subfigure}{.22\textwidth}
  \centering
  \includegraphics[width=.9\linewidth]{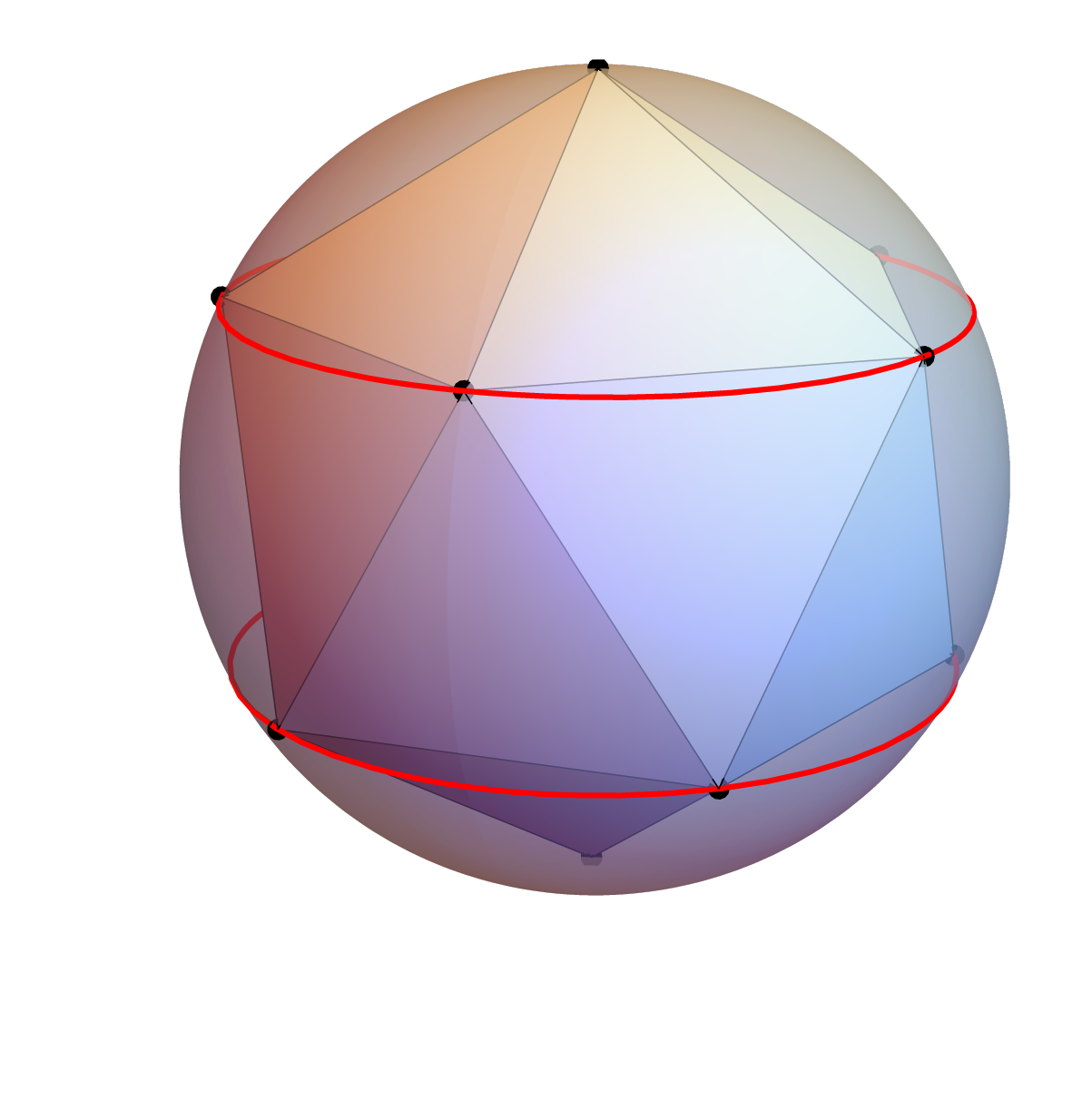}  
  \caption{$N=12$, $\mathbb{Z}_4$}
  \label{Fig:GSN12Z4}
\end{subfigure} 
\caption{Ground states for $N=4,5,6, 12$ and their $\mathbb{Z}_m$-symmetries.}
\label{Fig:GS-Known}
\end{figure}

\begin{figure}[ht]
\centering
\begin{subfigure}{.22\textwidth}
  \centering
  \includegraphics[width=.9\linewidth]{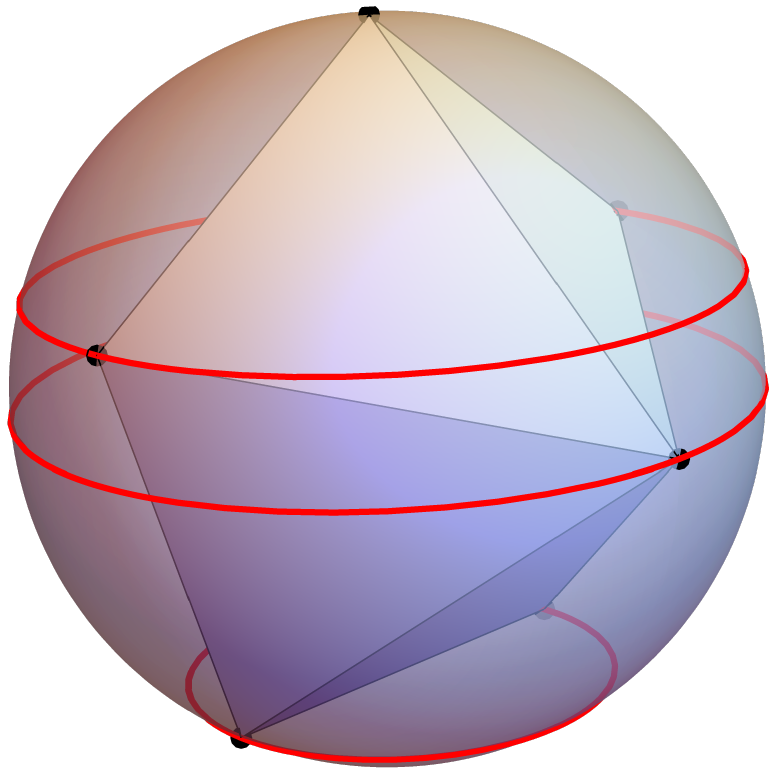}  
  \caption{$N=7$, $\mathbb{Z}_2$}
  \label{Fig:GSN7Z2}
\end{subfigure} 
\begin{subfigure}{.22\textwidth}
  \centering
  \includegraphics[width=.9\linewidth]{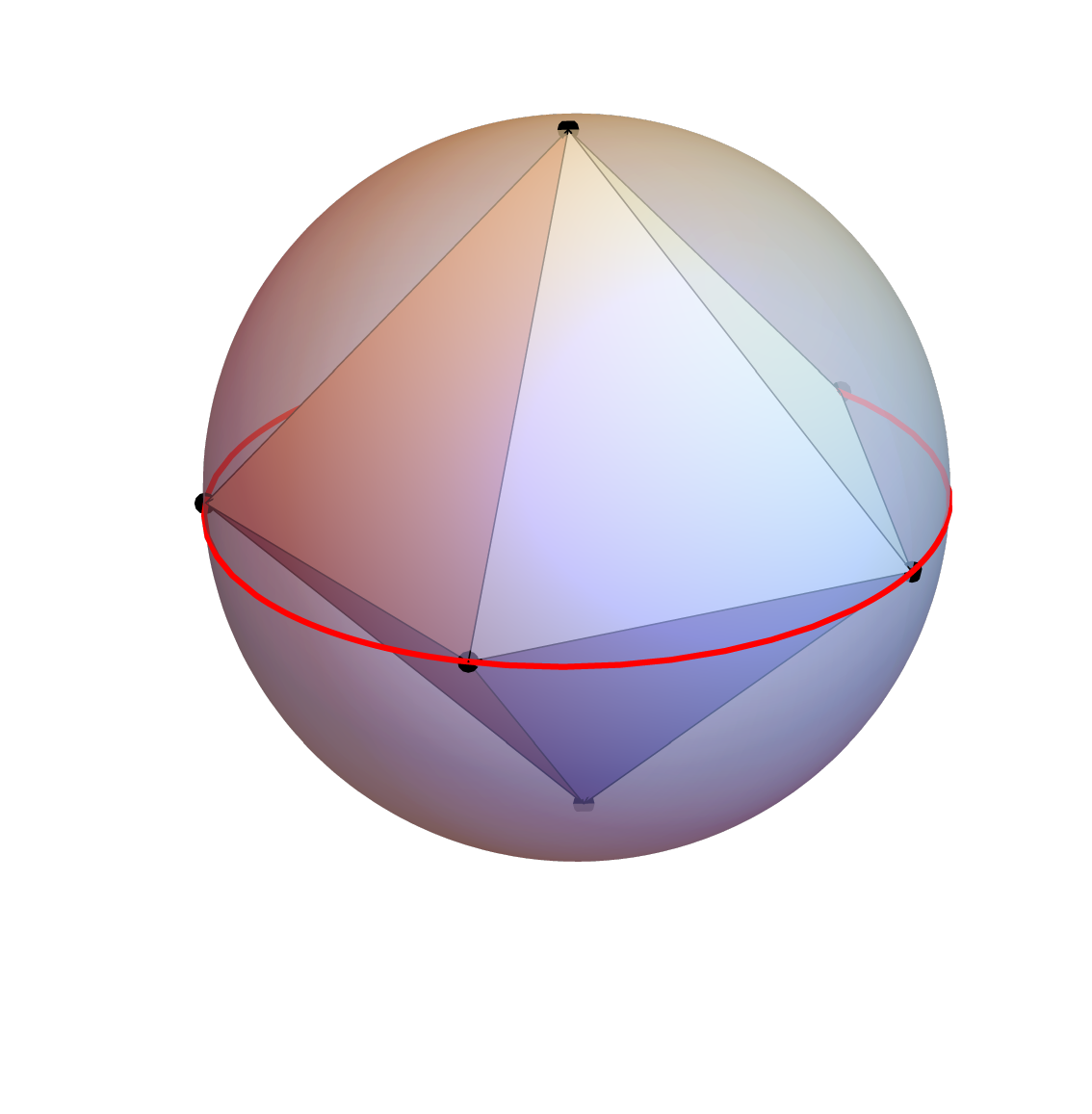}
  \caption{$N=7$, $\mathbb{Z}_5$}
  \label{Fig:GSN7Z5}
\end{subfigure} \qquad
\begin{subfigure}{.22\textwidth}
  \centering
  \includegraphics[width=.9\linewidth]{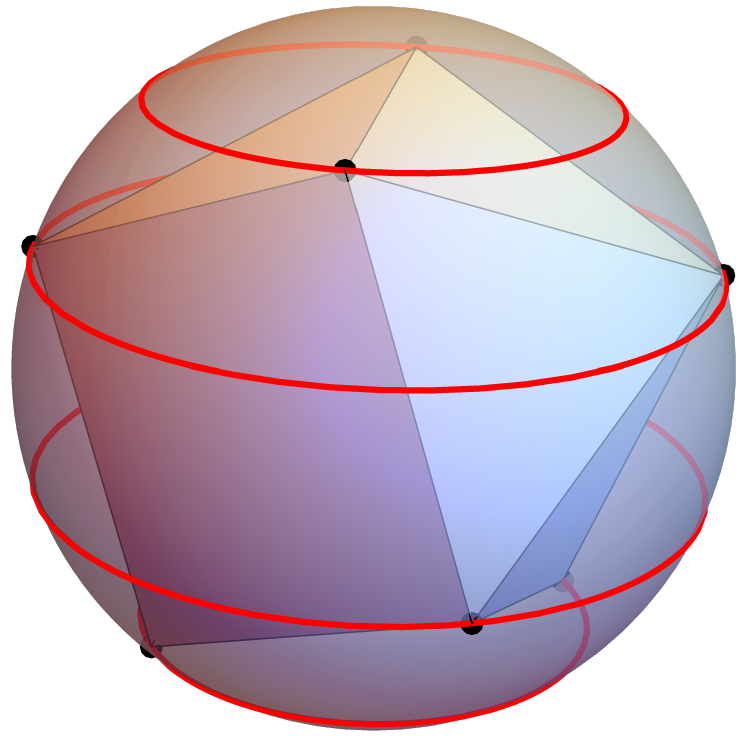}  
  \caption{$N=8$, $\mathbb{Z}_2$}
  \label{Fig:GSN8Z2}
\end{subfigure} 
\begin{subfigure}{.22\textwidth}
  \centering
  \includegraphics[width=.9\linewidth]{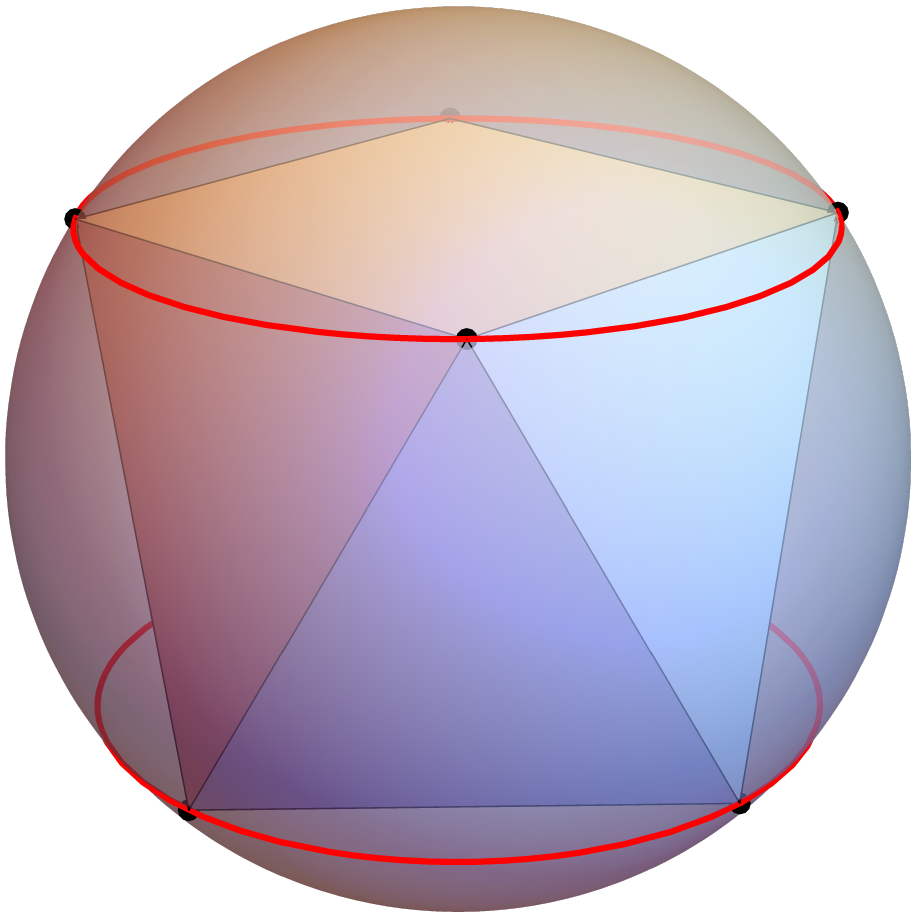}
  \caption{$N=8$, $\mathbb{Z}_4$}
  \label{Fig:GSN8Z4}
\end{subfigure}  \\
\begin{subfigure}{.22\textwidth}
  \centering
  \includegraphics[width=.9\linewidth]{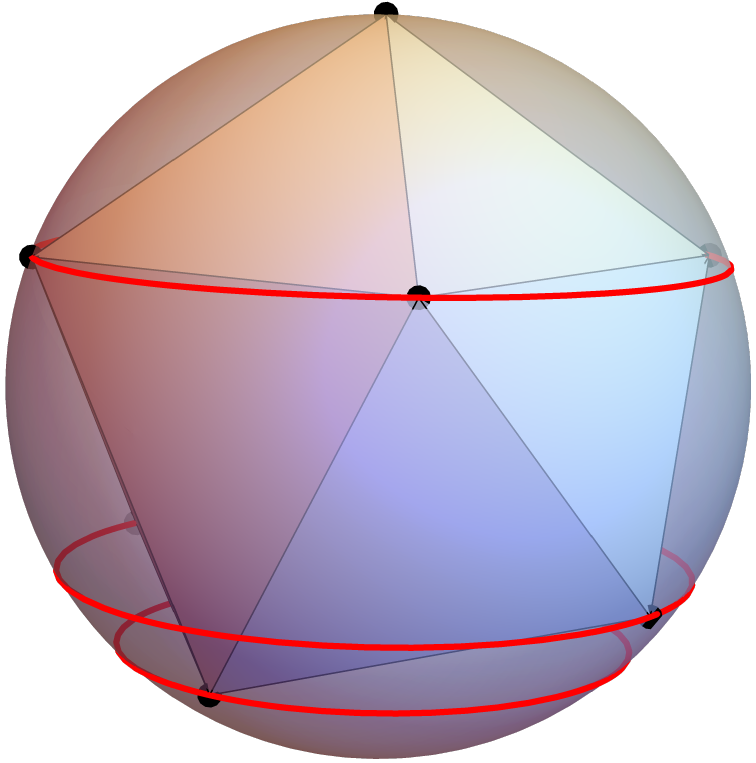}  
  \caption{$N=9$, $\mathbb{Z}_2$}
  \label{Fig:GSN9Z2}
\end{subfigure} 
\begin{subfigure}{.22\textwidth}
  \centering
  \includegraphics[width=.9\linewidth]{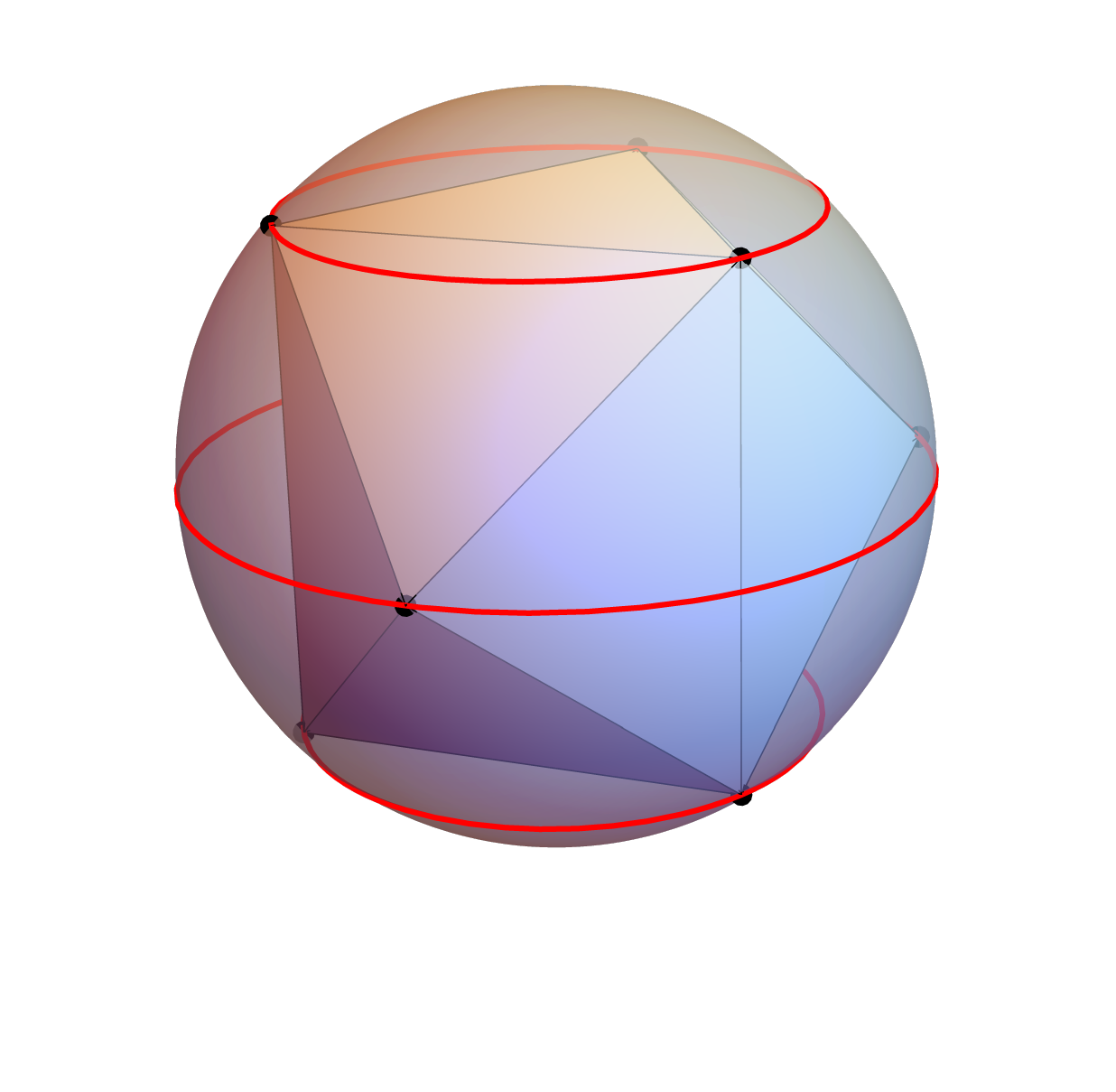}
  \caption{$N=9$, $\mathbb{Z}_3$}
  \label{Fig:GSN9Z3}
\end{subfigure} \qquad
\begin{subfigure}{.22\textwidth}
  \centering
  \includegraphics[width=.9\linewidth]{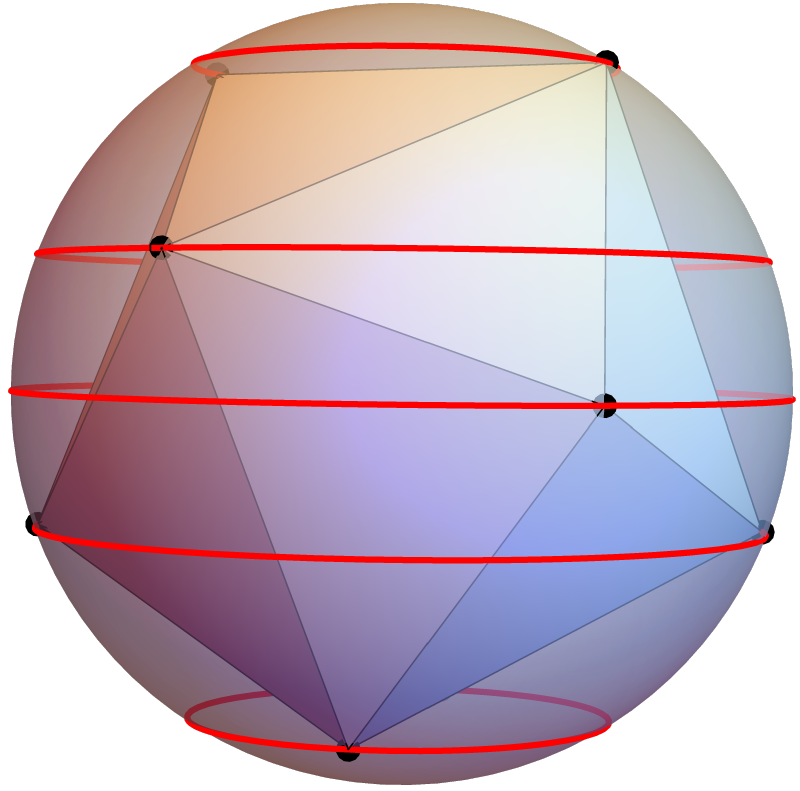}  
  \caption{$N=10$, $\mathbb{Z}_2$}
  \label{Fig:GSN10Z2}
\end{subfigure} 
\begin{subfigure}{.22\textwidth}
  \centering
  \includegraphics[width=.9\linewidth]{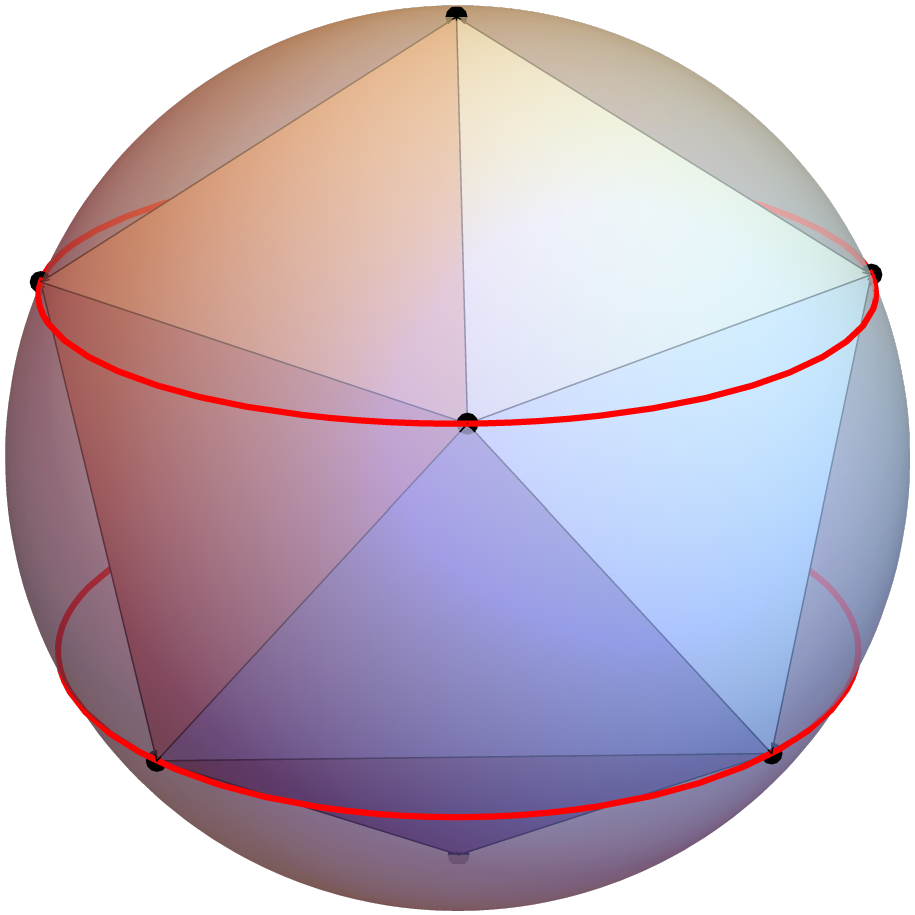}
  \caption{$N=10$, $\mathbb{Z}_4$}
  \label{Fig:GSN10Z4}
\end{subfigure} \\
\begin{subfigure}{.22\textwidth}
  \centering
  \includegraphics[width=.9\linewidth]{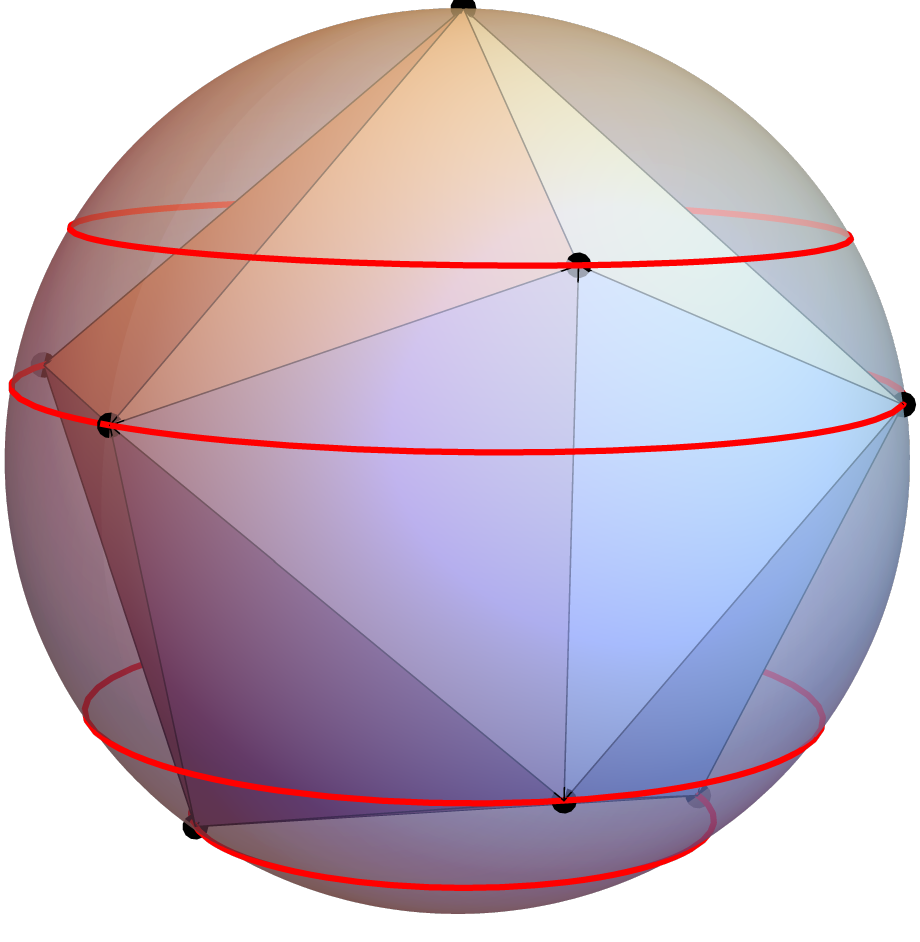}  
  \caption{$N=11$, $\mathbb{Z}_2$}
  \label{Fig:GSN11Z2}
\end{subfigure} 
\caption{Conjectured ground states for $N=7, 8, 9, 10, 11$ and their $\mathbb{Z}_m$-symmetries.}
\label{Fig:GS-Conjecture}
\end{figure}

\subsubsection*{Stability.}

Our nonlinear stability analysis relies on  the \emph{energy-momentum
method} of Patrick \cite{Patrick92}. Such method concludes Lyapunov stability modulo a subgroup, which 
in our case generically translates into orbital stability of a periodic orbit (see Proposition \ref{prop:stability}). 
The method examines positivity of a certain Hessian matrix which  we  
block-diagonalise exploiting  the $\Z_m$-symmetries following closely the construction of Laurent-Polz, Montaldi and Roberts \cite{Mo11}.
 Our approach differs from    \cite{Mo11}  in two ways. On the one hand, instead of using spherical coordinates, 
  we work with the extrinsic geometry induced by our embedding of  $M$ on $\R^{3N}$. This allows us to interpret
  tangent vectors to $M$ as vectors in the ambient space $\R^{3N}$, which is convenient for the implementation of CAPs.
   Secondly, we refine the block diagonalisation of   \cite{Mo11}  by exploiting  the complex 
  structure of some blocks as determined by Theorem \ref{th:complexblocks}.
   The great technicality  involved in this block diagonalisation and its refinement is compensated 
  by  obtaining larger 
ranges of application of CAPs,  since the method prevents  clustering of eigenvalues of large matrices.

A summary of the stability test which indicates the matrix blocks which need to be computed, and their dimension,
according to the values of $m$, $n$ and $p$ is given in subsection \ref{ss:Summary}.

\subsubsection*{Computer-assisted proofs (CAPs).}

%

The  CAPs in this paper are obtained via a finite dimensional
Newton-Kantorovich like theorem (see \cite{MR0231218} for the original
version), which is similar to the well-known interval Newton's method \cite%
{MR694437,MR0231516} and Krawczyk's operator approach \cite%
{MR255046,MR1057685}. This method allows us to find zeros of the mapping 
$F$ in \eqref{eq:F-intro} to find enclosures of the branches of $\Z_m$-symmetric RE and also, via a suitable formulation, to validate eigenvalues 
 necessary for the stability test.

As mentioned above, the CAPs are implemented in an
 INTLab code available in \cite{KevinCode}, whose input is a numerical approximation of a $\Z_m$-symmetric RE (configuration and angular speed).  If the approximation has sufficient precision and is not too close to a bifurcation,  the code 
 proves existence of a branch of
RE with this symmetry, determines an  enclosure and validates the stability test.

The range of applications of our framework and our code is  exemplified by 
the results presented in section \ref{s:CAPs-application}. 
In particular,  Theorem \ref{th:gs} proves  stability of the equilibrium configurations 
in Table \ref{ground-states} for $N=8,9,10,11$ (the case $N=7$ is degenerate and our test is inconclusive, and the stability
for other values of $N$ in the table was known).
We then prove existence and   
find  enclosures of  segments of stable branches of RE emanating from  the equilibria in Table \ref{ground-states}. These results 
are presented in 
 section \ref{ss:CAP-RE-groundstates}, and allow us
to conjecture which  branches of RE minimise $H$ 
on the level sets   $\Phi^{-1}(\mu)\subset M$ for  $\mu$ close to zero for the different values of $N$.
Finally, in section \ref{ss:RE-totalcollision}, we  investigated RE near total collision. The results of \cite{Mo11} imply that for   
 $N\geq 10$ the $\Z_m$-symmetric RE  consisting of only 1-ring are unstable. Our results  in section \ref{ss:RE-totalcollision} provide an educated first
 guess of the shape of the  RE which minimises $H$ on the level sets of $\Phi$ near total collision  for $N=10,11,12$.
 
\subsection{Structure of the paper} We start by reviewing  some known preliminary material  in section \ref{s:prelim}.
This serves to introduce  concepts and   notation  used throughout the paper. Moreover, 
it is useful  to recall known details and results on RE, 
the energy-momentum method and the CAPs that we use. We then focus on results on existence 
of branches of RE in section \ref{s:existence}. Our main original results in this section are Theorem \ref{th:existenceRE}
on the existence of branches of RE near equilibrium for arbitrary vortex strengths, 
Theorem   \ref{th-main-symmetry} on the discrete $\Z_m$-reduction in the case of equal vorticities, and 
Theorem \ref{thm:continuation} and Corollary \ref{cor:existenceREsymmetric}
on the existence of local branches of $\Z_m$-symmetric RE. Section \ref{sec:Stability} 
focuses on the application of the energy-momentum method for nonlinear  stability of RE. Most of this section 
is a development of the block-diagonalisation of \cite{Mo11} which, as mentioned above, is better suited for the implementation of CAPs.  
Our  refinement of this construction is given in Theorem \ref{th:complexblocks} which exploits the complex structure 
of some blocks.  A useful summary of the blocks that need to be constructed and whose eigenvalues
need to be calculated according to the
values of $m$, $n$ and $p$ is given in subsection \ref{ss:Summary}. In Section \ref{sec:CAPs} we 
explain how CAPs are implemented using the setting of sections  \ref{s:existence} and 
\ref{sec:Stability} to establish existence and stability of branches of $\Z_m$-symmetric RE.
Finally, we give examples of the range of applications of our framework and our code in section \ref{s:CAPs-application}.
The paper contains a series of appendices which complement the main body of the text.

\section{Preliminaries}
\label{s:prelim}

In this section we recall the equations of motion and basic known properties of the equations of motion
of the $N$-vortex on the sphere. This allows us to introduce the notation and concepts used ahead. In particular
 our working definition of RE is given in Definition \ref{def:RE}. We then
recall some properties of RE which follow from the general theory of RE for Hamiltonian systems with symmetry
which are used in our construction.
We also recall the energy momentum method for nonlinear stability  in subsection \ref{ss:stability-prelim}
and the Newton-Kantorovich like theorem for our CAPs in subsection \ref{ss:CAP-prelim}.

\subsection{Equations of motion} 

The equations of motion of the $N$-vortex problem on the sphere are a 
Hamiltonian system on the phase space $M$ obtained as the cartesian product of $N$ copies of the unit sphere in $\R^3$ minus the {\em collision set}
$\Delta$. That is, $M:=(S^2\times \dots \times S^2)\setminus\Delta$ where 
\begin{equation*}
S^2:=\{w \in \R^3 \, :\, \|w\|=1\} \qquad \mbox{and} \qquad \Delta:=\{(v_1,\dots, v_N)\in M \, :\, v_i=v_j \,\, \mbox{for some $i\neq j$} \}.
\end{equation*}
The Hamiltonian 
function $H:M\to \R$ and symplectic form $\Omega$ on $M$ are given by
\begin{equation*}
H(v):=-\frac{\Gamma_i \Gamma_j}{4\pi}\sum_{i<j} \ln \left ( \left \| v_i -v_j  \right \|^2 \right ), \qquad \Omega:=\sum_{i=1}^N\Gamma_i \pi^*_i \omega_{S^2}. 
\end{equation*}
Here $v=(v_1,\dots, v_N)\in M$ specifies the position of each of the vortices whose respective (constant) intensities are denoted by the scalars $\Gamma_i\neq 0$. 
The map 
$\pi_i:M\to S^2$ denotes the  projection onto the $i^{th}$ factor of the cartesian product $S^2\times \dots \times S^2$ and $\omega_{S^2}$ is the standard area form on 
$S^2$. The resulting equations of motion are
\begin{equation}
\label{eq:motion-diff-vorticities}
\dot v_j=-\frac{1}{\Gamma_j} v_j \times \nabla_{v_j} H = \frac{1}{4\pi}\sum_{i\neq j}  \Gamma_i \frac{v_i \times v_j} {\|v_i-v_j\|^2}, \qquad i=1,\dots, N,
\end{equation}
where $\times$ denotes the cross product in $\R^3$. Owing to the Hamiltonian structure,  the Hamiltonian function $H$ is a first integral.

\subsection{Rotational symmetries, center of vorticity, equilibria  and relative equilibria}

\paragraph{Rotational symmetries.}
The position  of the vortices in the sphere is specified by $v=(v_1,\dots, v_N)\in M$. This tacitly assumes a choice of an inertial frame 
whose origin lies at the center of the sphere
 which defines cartesian coordinates for the $v_j$'s.  There is freedom in the  choice of orientation of this inertial frame and this is reflected by 
the invariance of the problem under the $SO(3)$ action that
simultaneously rotates all vortices. Specifically we consider the $SO(3)$ action on $M$ defined by
\begin{equation*}
(g,v)\mapsto g.v:=(gv_1,\dots, gv_N)\in M, \qquad g\in SO(3), \; v=(v_1,\dots, v_N)\in M.
\end{equation*}
For $N\geq 3$ this action is free since  the collisions do not belong to $M$. It is also a proper action since $SO(3)$ is compact.

It is easily checked that both the Hamiltonian and the symplectic form are invariant under this action and  as a consequence, the equations of motion \eqref{eq:motion-diff-vorticities}
are $SO(3)$-equivariant.
 
 \paragraph{Center of vorticity.}
 One may easily show  that  the components of the map  
 \begin{equation}
 \label{eq:Phi-general-defn}
\Phi:M\to \R^3, \qquad \Phi(v)=\Gamma_1v_1+\dots + \Gamma_Nv_N,
\end{equation}
are  first integrals of the equations of motion  \eqref{eq:motion-diff-vorticities}. It is common to refer to $\Phi(v)$ as the {\em center of vorticity} 
of the vortex configuration $v=(v_1,\dots , v_N)$. The existence of these first integrals is in fact an instance of Noether's theorem. 
Upon the standard identification of the dual Lie algebra $\so(3)^*$ with $\R^3$, the function $\Phi$ can be geometrically interpreted as the {\em momentum map} associated to the
symplectic action of $SO(3)$ on the symplectic manifold $(M,\Omega)$. 

For $N\geq 3$, the freeness of the action   guarantees that
$\Phi$ is a submersion onto its image. As a consequence, if $\mu\in \R^3$ is such that $\Phi(v)=\mu$ for some $v\in M$, then
$\Phi^{-1}(\mu)$ is a smooth submanifold of $M$ of codimension 3 and $T_v\Phi^{-1}(\mu)=\ker d\Phi (v)$.

The  momentum map $\Phi$ plays a  fundamental role in the study of relative equilibria. 
An important property of $\Phi$ is its  equivariance with respect to the coadjoint representation which, in our interpretation, is simply the 
standard linear action of $SO(3)$ on $\R^3$, namely, one has
\begin{equation}
 \label{eq:Mom-map-equiv}
\Phi(g.v)=g\Phi(v), \quad \forall v\in M, \; g\in SO(3).
\end{equation}

\paragraph{Equilibrium points and ground states.} The simplest solutions of the equations of motion are the equilibria which are in one-to-one correspondence with the critical points of $H:M\to \R$.
Due to the $SO(3)$ invariance of the problem, equilibrium points are never isolated: if $v_0\in M$ is an equilibrium of \eqref{eq:motion-diff-vorticities}
so is $g.v_0$ for any $g\in SO(3)$. In other words, the  orbit $SO(3).v_0 \subset M$ is comprised of equilibrium points. 

If all the vortex strengths $\Gamma_i$ have the same sign, then $H$ is bounded from below and it assumes a global minimum on $M$. The 
minimising configurations are called \defn{ground states} and 
form a very  important kind of equilibria from the physical point of view.  Again, the $SO(3)$-invariance of $H$ implies that if $v_0\in M$ is a ground state then  the orbit $SO(3).v_0 \subset M$ consists of ground states.
The ground states of the system are stable in a sense that will be made precise below.
As mentioned in the introduction, proving that a given equilibrium is the ground state of the system  is a very difficult problem for which very little is known.

\paragraph{Relative equilibria.} The next  type of fundamental solutions are  the so-called {\em relative equilibria} (RE) which are the main subject of this work. This kind of solutions
may exist for any system of differential equations which is equivariant with respect to the action of a continuous  symmetry group. There are several equivalent definitions
of relative equilibria. For our purposes, it suffices to define them  as 
solutions of the equations  that at the same time are orbits of a one-parameter subgroup of the symmetry group. Thus, in our case, relative equilibria have the form
$v(t)=e^{\xi t}.a$ for a fixed $a \in M$ and a skew-symmetric matrix $\xi\in \so(3)$. This is a periodic solution in which
 the vortices steadily rotate along the axis passing through the center of the sphere in the direction of the vector $\xi_e \in \R^3$ satisfying $\langle \xi_e\rangle =\ker \xi$
and whose angular velocity $\omega_e$ satisfies $2\omega_e^2=\mbox{trace}(\xi \xi^T)$. 

Just like equilibrium points, relative equilibria are never isolated:
if $v(t)=e^{\xi t}.a$ is a relative equilibrium and $g\in SO(3)$ then, by equivariance of \eqref{eq:motion-diff-vorticities}, 
$\tilde v(t):=g.v(t)$ is also a solution. However, we may
write
\begin{equation*}
\tilde v(t)=g.v(t)=(ge^{\xi t}g^{-1}).(g.a)=e^{(g \xi g^{-1}) t}.(g.a) =e^{\tilde \xi t} .\tilde a,
\end{equation*}
where $\tilde \xi := g \xi g^{-1}\in \so(3)$ and $\tilde a:=g.a\in M$, which shows that $\tilde v(t)$ is also a relative equilibrium
with axis of rotation $\tilde \xi_e=g\xi_e$ and the same angular frequency $\tilde \omega_e=\omega_e$. Physically, this means that the orientation 
of the axis of rotation of a relative equilibrium is unessential and may be arbitrarily chosen via a suitable  rotation of the inertial frame. On the other hand, relative
equilibria having distinct angular frequencies are not related by a rotation. Therefore, for the rest of the paper, we will
restrict our attention to relative equilibria for which  the axis of rotation is the $z$-axis and we will give a prominent role to the angular velocity $\omega$. 
In accordance with this,  we search for solutions of  \eqref{eq:motion-diff-vorticities} of the form
\begin{equation}
\label{eq:RE-ansatz}
v(t)=e^{\omega J_3 t}.a =(e^{\omega J_3 t}a_1, \dots, e^{\omega J_3 t}a_N), \qquad J_3:=\begin{pmatrix} 0 & -1 &0  \\ 1 & 0 &0 \\ 0 & 0 & 0\end{pmatrix},
\end{equation}
for a certain frequency $\omega \in \R$ and a configuration $a=(a_1,\dots, a_N)\in M$. Our discussion leads to the following definition.

\begin{definition} 
\label{def:RE} A \defn{relative equilibrium}  (RE) of the equations of motion  \eqref{eq:motion-diff-vorticities} is a pair $(a,\omega)\in M\times \R$ such
that $v(t)$ given by \eqref{eq:RE-ansatz} is a solution of 
 \eqref{eq:motion-diff-vorticities}. We say that  $\omega$ is  the \defn{angular velocity}  and $a\in M$ is  the   \defn{configuration}
of the relative equilibrium.
\end{definition}

\begin{remark}
\label{Rmk:REdef}
\begin{enumerate}
\item Equilibrium points  may be interpreted as RE
with zero angular velocity.
\item If $(a,\omega)$ is a RE with $\omega \neq 0$, then so is  $(e^{J_3 \theta}.a,\omega)$ for any $\theta \in [0,2\pi)$. This is a consequence
of the autonomous nature of the equations of motion  \eqref{eq:motion-diff-vorticities}: if $v(t)$ is a solution, then so is $v(t+t_0)$ for any $t_0\in \R$. In other words, 
$a$  is determined up to a time-shift along the solution with initial condition $a$. 
Considering that the solution is periodic, this time-shift corresponds to 
a certain $SO(2)$ action. Such $SO(2)$ action is precisely  the one-parameter subgroup of $SO(3)$ of rotations about the $z$ axis. Namely, if we 
denote  $e_3:=(0,0,1)\in \R^3$, and define
\begin{equation}
\label{eq:def-SO(2)}
SO(2):=\{ g\in SO(3) \, : \, ge_3=e_3\} = \left \{ e^{J_3\theta}= \begin{pmatrix} \cos \theta & -\sin \theta & 0 \\ \sin \theta & \cos \theta & 0 \\ 0 & 0 & 1 \end{pmatrix} \, : \, \theta \in [0,2\pi)\right \},
\end{equation}
then the orbit $SO(2).a$ coincides with the dynamical orbit of the solution $v(t)=e^{\omega J_3 t}.a$, and is comprised of relative equilibrium configurations with the same angular velocity $\omega$.
\item Some authors  instead define the whole orbit $SO(3).a\subset M$ of the relative equilibrium configuration $a$ as {\em  the} relative equilibrium since it projects to
a single equilibrium point of the reduced system on the quotient space $M/SO(3)$.
As explained in item (ii), our approach selects instead
an $SO(2)$-orbit of representatives. This is essential for the continuation of relative equilibria as a function of the angular velocity $\omega$
considered ahead.
\end{enumerate}
\end{remark}

For any $\omega\in \R$ define the \defn{augmented Hamiltonian}
\begin{equation}
\label{eq:aug-Ham}
H_\omega:M\to \R, \qquad H_\omega(v):=H(v) - \omega \Phi_3(v),
\end{equation}
where $\Phi_3:M\to \R$ denotes the third component of the momentum map $\Phi$ given by \eqref{eq:Phi-general-defn}. Namely, 
$\Phi_3(v)=e_3\cdot \Phi(v)$ where $e_3=(0,0,1)\in \R^3$, and $\cdot$ denotes the standard euclidean product in $\R^3$.
The following proposition is a particular instance of  general results on relative equilibria on symplectic manifolds (see e.g. \cite{lom}).
We present an elementary proof for completeness. 

\begin{proposition}
\label{prop:RE-as-crit-pts}
The following statements hold.
\begin{enumerate}
\item $(a,\omega)\in M\times \R$ is a relative equilibrium if and only if $a\in M$ is a critical point of $H_\omega$.
\item If $(a,\omega)\in M\times \R$ is a relative equilibrium  with $\omega\neq 0$ then $e_3\times \Phi(a)=0$.
\end{enumerate}
\end{proposition}
\begin{proof}
(i) Let $\omega\in \R$ and introduce the time dependent change of  variables, 
\begin{equation*}
y_{j}(t)=e^{-\omega J_{3}t}v_{j}(t)\text{,} \qquad j=1,\dots, N.
\end{equation*}%
According to Definition \ref{def:RE}, it is clear that  $(a,\omega)\in M\times \R$ is a relative equilibrium if and only if $a\in M$ is
an equilibrium solution for the equations satisfied by $y(t)=(y_1(t), \dots, y_N(t))$. A direct calculation shows that 
\begin{equation*}
\dot{y}_{j}=\frac{1}{4\pi} \sum_{i \neq j)}^{N}\Gamma_i \frac{y_{i}\times y_{j}}{\left \|
y_{j}-y_{i}\right \|^{2}}-\omega e_{3}\times y_{j} = -y_{j}\times \left( \nabla _{y_{j}}H_\omega(y)\right). 
\end{equation*}
This shows that the equations for $y(t)$ are Hamiltonian on the symplectic manifold $(M,\Omega)$ with respect to the Hamiltonian 
function $H_\omega$. As a consequence, the equilibrium points of $y(t)$ are in one-to-one correspondence with the critical points of $H_\omega$.

(ii) By conservation of $\Phi$ we have $\Phi(a)= \Phi(v(t))$ where $v(t)$ is given by \eqref{eq:RE-ansatz}.
Therefore,
\begin{equation*}
\begin{split}
\Gamma_1 a_1+\dots +\Gamma_Na_N &=\Gamma_1e^{\omega J_3t}a_1+\dots+\Gamma_N e^{\omega J_3t}a_N \\ & =e^{\omega J_3t}(\Gamma_1 a_1+\dots+ \Gamma_Na_N)
=e^{\omega J_3t} \Phi(a).
\end{split}
\end{equation*}
Differentiating with respect to $t$ and evaluating at $t=0$ gives $\omega J_3 \Phi(a)=0$, which
for $\omega\neq 0$ is equivalent to $e_3\times \Phi(a)=0$.

%
\end{proof}

A fundamental observation which follows directly from \eqref{eq:Mom-map-equiv} and the definition of the augmented Hamiltonian $H_\omega$ is that, for $\omega\neq 0$, 
 $H_\omega$ is no longer $SO(3)$-invariant but only
$SO(2)$-invariant with  $SO(2)$ given by \eqref{eq:def-SO(2)}. As a consequence of this $SO(2)$-symmetry, if $a\in M$ is a critical point of $H_\omega$ so is $e^{J_3\theta}.a$
for $\theta\in [0,2\pi)$. This 
 observation (together with  Proposition \ref{prop:RE-as-crit-pts}) is consistent with item 
(ii) of Remark \ref{Rmk:REdef}.

Another useful interpretation of RE which follows  from the Proposition \ref{prop:RE-as-crit-pts} and the 
Lagrange multiplier theorem  is stated next. We refer the reader to \cite[Chapter 4]{lom} for details.

\begin{proposition}
\label{prop:RE-as-crit-pts-of-restricted-Ham}
$(a,\omega)\in M\times \R$ is a relative equilibrium if and only if $a\in M$ is a critical point of the restriction of
$H$ to the level set $\Phi^{-1}(\mu)\subset M$ where $\mu=(0,0,\Phi_3(a))\in \R^3$.
\end{proposition}

Recall from  our previous discussion on the ground states, that if all the vortex strengths $\Gamma_i$ have equal sign, then the Hamiltonian $H$ is bounded from below.
Therefore, given $\mu \in \R^3$ belonging to the image of $\Phi$, there exists a minimal energy relative equilibrium $(a,\omega)\in M\times \R$ satisfying that  $H(a)$ is the minimum value of the restriction
of $H$ to $\Phi^{-1}(\mu)\subset M$. These minimising RE are stable (in a sense that we make precise below). The
results in section \ref{s:CAPs-application} allow us to  conjecture which are
the positions of these RE in the case of identical vortex strengths for specific values of $N$,  and to prove that they are at least local minima of the restriction of  $H$ to $\Phi^{-1}(\mu)\subset M$ for certain values of $\mu$ using computer-assisted proofs.

\subsection{The case of identical vortex strengths.} 
\label{ss:equal-strengths}

Most of the results of this paper assume that the 
vortex strengths $\Gamma_i=\Gamma_j$ for all $i,j=1,\dots, N$. With an appropriate scaling of units, we can assume that 
 the Hamiltonian $H$ and the symplectic form $\Omega$ are given by
 \begin{equation}
 \label{eq:Ham-Om-equal}
H(v)=-\sum_{i<j} \ln \left ( \left \| v_i -v_j  \right \|^2 \right ), \qquad \Omega=\sum_{i=1}^N\ \pi^*_i \omega_{S^2},
\end{equation}
and the equations of motion are 
\begin{equation}
\label{eq:motion}
\dot v_j=- v_j \times \nabla_{v_j} H =\sum_{i\neq j}   \frac{v_i \times v_j} {\|v_i-v_j\|^2}, \qquad i=1,\dots, N.
\end{equation}
Finally, the centre of vorticity, or momentum map $\Phi$, whose general form is given by \eqref{eq:Phi-general-defn} simplifies to 
 \begin{equation}
 \label{eq:Phi}
\Phi:M\to \R^3, \qquad \Phi(v)=v_1+\dots + v_N.
\end{equation}

The assumption that the vortices have equal strengths introduces a discrete symmetry  by permutation
of the vortices.  Specifically, if $S_N$ is the symmetric group of order $N$ and  $\hat G:=S_N\times SO(3)$ then $\hat G$ defines a (left)
action  on  $M$ by
\begin{equation}
\label{eq:Full-symmetry}
(\sigma,g).(v_1,\dots, v_N):=(gv_{\sigma^{-1}(1)}, \dots, gv_{\sigma^{-1}(N)}), \quad  (\sigma,g)\in S_N\times SO(3)= \hat G, \quad (v_1,\dots, v_N)\in M.
\end{equation}
The above action is  symplectic, preserves the Hamiltonian $H$, and the equations \eqref{eq:motion} are 
equivariant. However, in contrast with the $SO(3)$-action on $M$, the $\hat G$-action  is not free and this will allow us to extract valuable information.

\paragraph{Equilibrium points and ground states.} We first
recall the following known result about equilibria in the case of identical vorticities (see Theorem 10.1 in \cite{Beltran20} and the references therein).
\begin{proposition}
\label{prop:zeroMomentum}
Let $a\in M$ be an equilibrium point of the equations of motion \eqref{eq:motion} for identical vortices,  then $\Phi(a)=0$.
\end{proposition}

We now recall from the introduction, that determining the ground states in the case of equal vortex strengths corresponds to Smale's problem $\# 7$ which 
has only been solved for a few values of $N$ (2, 3, 4, 5, 6 and 12).  Table \ref{ground-states}  lists
 the known ground states together with our conjecture  of their position for the values $N=7, 8, 9, 10, 11$. The table  also indicates their $\mathbb{Z}_m$-symmetries since they will be essential in our
 study of relative equilibria emanating from these configurations at a later stage of the paper. The contents of the table for $N\geq 4$ are illustrated in  Figures \ref{Fig:GS-Known} and \ref{Fig:GS-Conjecture}.
  Note that the center of vorticity of all these configurations  vanishes by virtue of Proposition \ref{prop:zeroMomentum}.

\subsection{Stability of relative equilibria}
\label{ss:stability-prelim}

Our discussion  of stability of  relative equilibria relies on the following definition.
\begin{definition} (Patrick \cite{Patrick92}).
\label{def:Gmu-stable}
Let $K< SO(3)$ be a subgroup. A relative equilibrium $(a,\omega)\in M\times \R$ is  \defn{$K$-stable} 
or \defn{ stable modulo $K$} if for any  neighbourhood $V$ of $K.a$ there exists a neighbourhood $U\subset V$ of $a$ which is invariant under the flow
of the equations of motion \eqref{eq:motion-diff-vorticities} for $t\geq 0$.
\end{definition}

The idea of the definition is that if $(a,\omega)\in M\times \R$ is  a   $K$-stable  relative equilibrium, then the solutions with
 initial conditions near $a$ 
will stay close to 
the $K$-orbit through $a$ for all $t\geq 0$.

Let $(a,\omega)$ be a relative equilibrium and suppose that $\Phi(a)=\mu\in \R^3$. As shown by
Patrick  \cite{Patrick92},   if $(a,\omega)$ projects to a
Lyapunov stable equilibrium point of the symmetry reduced symplectic system, then the RE is $SO(3)_\mu$-stable where $SO(3)_\mu$ is the isotropy subgroup of $\mu$ under the coadjoint action 
(represented as $(g,\mu)\mapsto g\mu$). In other words, for the unreduced system \eqref{eq:motion-diff-vorticities}, initial 
conditions near $a$  may only drift along the orbits of $SO(3)_\mu$.  Therefore, the appropriate notion of stability of the RE $(a,\omega)$  is that of  $SO(3)_\mu$-stability.

There are  two possibilities for $SO(3)_\mu$:
\begin{enumerate}
\item If $\mu=0$ then  $SO(3)_\mu=SO(3)_0=SO(3)$.

\item If  $\mu\neq 0$ then  $SO(3)_\mu$ is the set of rotation matrices around $\mu$ which is isomorphic to $SO(2)$. (If $\omega\neq 0$, then 
 $SO(3)_\mu$ exactly equals the representation of $SO(2)$ in \eqref{eq:def-SO(2)} since $\mu=\Phi(a)$ is parallel to $e_3$ by  
Proposition \ref{prop:RE-as-crit-pts}(ii).)
\end{enumerate}

According to Proposition \ref{prop:zeroMomentum}, the first case ($\mu=0$) is always encountered at  equilibria when all  vortices have equal strengths. 
For instance, all equilibrium  configurations in Table \ref{ground-states} have $\mu=0$ and  are $SO(3)_0=SO(3)$-stable.\footnote{this is only a conjecture if $N=7$. For other values of $N$ see the discussion in Section \ref{ss:CAP-minimisers} for proofs and earlier references of this fact.}
This means that a solution whose
 initial condition is
close to any of these equilibria will in general not stay close to it, but will be constrained to evolve in such way that the vortices are at every time near  the vertices of a rotated version of such equilibrium. 

On the other hand, the second case ($\mu\neq 0$) is the generic one. For $\omega\neq 0$, considering that $SO(2).a$ coincides with the 
dynamical orbit 
 of the solution $v(t)=e^{\omega J_3 t}.a$ (item (ii) of Remark \ref{Rmk:REdef}), we come to the following conclusion that
 we state as a proposition due to its relevance.
 
 \begin{proposition}
 \label{prop:stability}
 Let $(a,\omega)\in M\times \R$ be a RE with $\omega\neq 0$.
 If $\mu=\Phi(a)\neq 0$
 then  $SO(3)_\mu$-stability of  the RE $(a,\omega)$ is equivalent to orbital stability of the periodic orbit $v(t)=e^{\omega J_3 t}.a$.
\end{proposition}

\subsubsection{Energy momentum method}
\label{ss:energy-momentum-method}

The stability analysis of RE that we employ in this work relies on the {\em energy-momentum method} of Patrick \cite{Patrick92} 
and will allow us to determine 
sufficient conditions for $SO(3)_\mu$-stability. 

Let $(a,\omega)\in M\times \R$ be a relative equilibrium and let $\mu=\Phi(a)\in \R^3$.  
From our discussion above,
we know  that  $a$ is a critical point of the augmented Hamiltonian $H_\omega:M\to \R$. The energy-momentum method
 examines the signature of the restriction
of the Hessian
$d^2H_\omega(a):T_aM\times T_aM\to \R$ to  a subspace $\mathcal{N}_a\subset T_aM$ satisfying the property 
\begin{equation}
\label{eq:symp-slice-condition}
\mathcal{N}_a\oplus (\so(3)_\mu \cdot a)=  \ker d\Phi(a).
\end{equation}
Here $\so(3)_\mu$ denotes the Lie algebra of $SO(3)_\mu$  and $\so(3)_\mu \cdot a\subset T_aM$ is the tangent space to
the $SO(3)_\mu$-orbit through $a$ at $a$. The subspace $\mathcal{N}_a$ is called a \defn{symplectic slice} and 
is a symplectic subspace of $T_aM$. The above property states that
it is tangent to the level set  $\Phi^{-1}(\mu)$ at $a$ and is also   a direct complement of the tangent space to 
the  $SO(3)_\mu$-orbit through $a$. 

If  the restriction of $d^2H_\omega(a)$ to the symplectic slice $\mathcal{N}_a$ is definite then \cite[Theorem 5.1.1]{lom} implies that $(a,\omega)$ is $SO(3)_\mu$-stable in $\Phi^{-1}(\mu)$ and $SO(3)$-stable in $M$. The results of Patrick \cite{Patrick92} then imply
that $(a,\omega)$  is $SO(3)_\mu$-stable in $M$. In particular we conclude that:

 {\em  if
the restriction of $d^2H_\omega(a)$ to $\mathcal{N}_a$ is  positive definite, then the RE $(a,\omega)$ is $SO(3)_\mu$-stable}.

%
%

%

\subsection{computer-assisted proofs}
\label{ss:CAP-prelim}

As mentioned in the introduction, the CAPs in this paper are obtained via a finite dimensional Newton-Kantorovich 
like theorem (see \cite{MR0231218} for the original version), which is similar to the well-known interval Newton's method \cite{MR694437,MR0231516} and Krawczyk's operator approach \cite{MR255046,MR1057685}. The central idea is to compute an approximate solution to the problem and to show that a Newton-like operator is a contraction on a ball centered at the numerical approximation. Additionally, we combine  predictor-corrector continuation methods (e.g. see \cite{MR910499,MR2359327}) together with the uniform contraction principle (e.g. see \cite{MR660633}) to obtain one-dimensional branches of solutions. Let us give some details. 

Consider a smooth map $F: \R^{d}\times \R \rightarrow \R^d$, $(x,\omega)\mapsto F(x,\omega)$,
  for some $d\in \mathbb{N}$. Now consider 
 two numerical approximations $(\bar{x}_0,\omega_0), (\bar{x}_1,\omega_1)$ 
 such that $F(\bar{x}_i, \omega_i) \approx 0$ (for $i = 0,1$). 
 Given $s\in [0,1]$, let $\bar{x}_{s}  \bydef s\bar{x}_1 + (1-s)\bar{x}_{0}$ and $\omega_{s} \bydef  s\omega_1 + (1-s)\omega_{0}$.

Choose a norm $\| \cdot\|$ on $\R^{d}$, and given a point $x \in \R^{d}$, define the closed ball of radius $r>0$ by $\overline{B_{r}(x)} = \{ \xi \in \R^{d} : \| \xi - x \| \le r \}$.

\begin{theorem}[Uniform Newton-Kantorovich-like theorem] \label{thm:Uniform Bounds}
Let $F$ be as above and $A \in M_{d}(\R)$ be such that $A\approx D_{x} F(\bar{x}_{0}, \omega_0)^{-1}$. Consider the non-negative bounds $Y$, $\hat{Y}$, $Z$ and the positive number $r_*$ satisfying
\begin{align*}
        \|A F(\bar{x}_{0}, \omega_0)\| &\leq Y, \\
        \|A[F(\bar{x}_{s}, \omega_s) - F(\bar{x}_{0}, \omega_0)]\| &\leq \hat{Y}, \quad \textit{for s} \in [0,1] \\
         \|A[I - D_{x}F(c, \omega_s)]\| &\leq Z, \quad \textit{for all c} \in \overline{B_{r_*}(\bar{x}_{s})} \textit{and s} \in [0,1].
\end{align*}
Define the radii polynomial
\begin{equation} \label{eq:radii_polynomial}
    p(r) \bydef  (Z -1)r + Y + \hat{Y}.
\end{equation}
If there exists $r_0$ such that $0 < r_0 \leq r_*$ satisfies $p(r_0) < 0$, then there exists a function
\[
\tilde{x} : \{ \omega_s : s\in [0,1]\} \rightarrow \bigcup\limits_{s\in[0,1]} \overline{B_{r_*}(\bar{x}_{s})}
\]
such that $F(\tilde{x}(\omega_s), \omega_s) = 0$ for all $s \in [0,1]$.
\end{theorem}

The above Newton-Kantorovich-like theorem will be used to show existence of branches of relative equilibria (see Section~\ref{ss:CAP-existence}) and to prove stability via a rigorous control of the eigenvalues of a specific matrix (see Section~\ref{ss:CAP-stability}).

\section{Existence of branches of relative equilibria}
\label{s:existence}

This section is concerned with  the  existence of branches of RE. We first present a general non-constructive existence result
valid for arbitrary vortex strengths in subsection \ref{ss:existence-different-vortex-strengths} and then focus on the case of identical vortex strengths. 
We develop a constructive framework to prove existence of symmetric RE in subsection \ref{ss:existenceRE-equal-strengths} which will be   implemented using CAPs in subsection \ref{ss:CAP-existence}.   

\subsection{Existence of relative equilibria for arbitrary vortex strengths}
\label{ss:existence-different-vortex-strengths}

We start by recalling the following classical definition that is needed in the formulation of our results.
\begin{definition}[Bott (1954) \cite{Bott54}]
\label{def:Bott}
Let $P \subset M$ be a submanifold and suppose each point in $P$ is a critical point of
$f:M\to \R$. We say that $P$ is a \defn{nondegenerate critical submanifold} if, in addition,
for each $x\in P$ one has\footnote{By  $\ker \left ( d^2f(x) \right )$ we mean 
the null-space of the bilinear form $d^2f(x):T_xM\times T_xM\to \R$, namely, $\ker d^2f(x)=\left \{ {\bf w} \in T_xM \,  : \, d^2f(x)( {\bf w}, \tilde{{\bf w}})=0 \; \forall  \tilde{{\bf w}}\in
T_xM \, \right \}$.} $\ker \left ( d^2f(x) \right ) =T_xP$.
\end{definition}
Note that it is obvious that
$T_xP\subset \ker \left ( d^2f(x) \right )$ and the definition is really saying that $d^2f(x)$ is non-degenerate in the 
directions transversal to $T_xP$ for all $x\in P$.

\begin{theorem}
\label{th:existenceRE}
Let $N\geq 3$ and suppose that $a_0\in M$ is an equilibrium of  \eqref{eq:motion-diff-vorticities} with the property that
 $SO(3).a_0$ is a nondegenerate critical manifold  of 
$H_{0}:M\rightarrow \mathbb{R}$. Then there exists $\omega_0>0$ such that for any 
$\omega
\in (-\omega _{0},\omega _{0})$ there exist at least two distinct relative equilibria with angular velocity $\omega$. 
These relative equilibria approach the orbit $SO(3).a_0$ as $\omega \to 0$.
\end{theorem}

The proof that we give below is not constructive and will not be used to determine RE at a later stage of
the paper.
 It relies on topological argument and 
techniques for the persistence of relative equilibria under symmetry breaking (e.g. \cite{Fo}) which are well suited to our problem since the symmetries of the augmented Hamiltonian
$H_\omega$ break from  $SO(3)$ to only $SO(2)<SO(3)$ as $\omega$ passes through zero. 

\begin{proof}
Let $\mathcal{N}\subset T_{a_0}M$ be a vector subspace which is a direct complement of $T_{a_0}(SO(3).a_0)$. The Palais slice theorem guarantees the existence of
an  $SO(3)$ invariant tubular neighbourhood $U$ of $SO(3).a_0$ which is diffeomorphic to $SO(3)\times \mathcal{N}_0$, where $\mathcal{N}_0\subset \mathcal{N}$ is
a neighbourhood of $0$. We may thus represent a point  $v\in U$ as a pair $v=(g,x)\in SO(3)\times \mathcal{N}_0$. In particular, the  point 
$a_0=(e,0)$ (where $e$ denotes the identity element in $SO(3)$) and elements of the orbit $SO(3).a_0$ have the form $(g,0)$. Our hypothesis that the orbit of $a_0$ is a nondegenerate critical manifold
of $H_0$ implies that for all $g\in SO(3)$ we have
\begin{equation*}
dH_0(g,0)=0 \qquad  \mbox{and} \qquad \left . d^2H_0(g,0) \right |_\mathcal{N} \quad \mbox{is non-degenerate}.
\end{equation*}
In particular, for fixed $g_0\in SO(3)$ and $\omega\in \R$, the map $f_{ g_0,\omega}:\mathcal{N}_0  \to \mathcal{N}^*$ given by $f_{g_0,\omega}(x)=d_xH_\omega(g_0,x)$ 
satisfies 
\begin{equation*}
f_{g_0,0}(0)=0 \qquad \mbox{and} \qquad D_xf_{g_0,0}(0) \quad \mbox{is invertible}.
\end{equation*}
Therefore, by the implicit function theorem we have 
\begin{equation*}
d_xH_\omega(g,\tilde x_{g_0}(g,\omega))=0,
\end{equation*}
for a unique function $\tilde x_{g_0}:U_{g_0} \times I_{g_0} \to \mathcal{N}_0$, where $U_{g_0}$ is a neighbourhood  of $g_0\in SO(3)$,  $I_{g_0}\subset \R$ is 
a small interval around zero, and $\tilde x_{g_0}$ satisfies $\tilde x_{g_0}(g_0,0)=0$.

The above argument may be repeated  varying the group element $g_0\in SO(3)$. Using  compactness of $SO(3)$, one may obtain (see \cite{Fo} for details)
 a uniform version of the implicit
function theorem, namely, the existence of $\omega_0>0$ and a unique function $\tilde x:SO(3)\times (-\omega_0, \omega_0)\to \mathcal{N}_0$ satisfying 
$\tilde x(g,0)=0$ such that 
\begin{equation*}
d_xH_\omega(g,\tilde x(g,\omega))=0, \qquad \mbox{for all $g\in SO(3)$, $\omega\in (-\omega_0,\omega_0)$}.
\end{equation*}

For any $\omega\in (-\omega_0,\omega_0)$ define $\psi_\omega :SO(3)\to \R$ by 
$\psi_\omega (g) = H_\omega (g, \tilde x(g,\omega))$.
By compactness of $SO(3)$ the function $\psi_\omega$ attains its maximum and its minimum at certain $g_1, g_2\in SO(3)$ and hence
$H_\omega$ has an extremum at $a_1:=(g_1, \psi_\omega(g_1))$ and  $a_2:=(g_2, \psi_\omega(g_2))$. Therefore,  $(a_1,\omega)$ and $(a_2,\omega)$ are the desired relative equilibria.
%
%
%
\end{proof}

\begin{remark}
\begin{enumerate}
\item The proof given above follows  the ideas  from  \cite[Theorem 2.2]{Fo} (see also references therein), although  our case is simpler 
because the $SO(3)$ action on $M$ is free
when the total number of vortices  $N\geq 3$. 
\item There are several existence results of RE in  Lim, Montaldi and Roberts \cite[Section 3]{MR1811389}.  Their  approach relies on   the discrete symmetries arising from 
permutations of identical vortices, and the momentum value plays a prominent role.
 In particular, section 3.3 in their paper is devoted to bifurcations from zero momentum.
Our  Theorem \ref{th:existenceRE}   is complementary to their results since we do not assume any equality between the vortex intensities and
the emphasis is given to the angular velocity $\omega$ instead of the   momentum. It is important
to notice that equilibrium points (for which  $\omega=0$) may  have
non-vanishing momentum  if the intensities of the vortices are distinct as is seen by the example of 
two antipodal vortices with different strengths.

\item It is unclear to us how general is the non-degeneracy assumption in the theorem. It would be reasonable to expect that it
is generically satisfied, but in the course of our investigations (see Remark \ref{rmk:N7deg}) 
we discovered that it does not hold for the pentagonal bipyramid corresponding to $N=7$ in Table \ref{ground-states}  where all vortex intensities are assumed to coincide.

%
%
%
%
%
%
%
%

\end{enumerate}

\end{remark}

%
%

\subsection{Existence of symmetric relative equilibria in the case of identical vortex strengths}
\label{ss:existenceRE-equal-strengths}

From now on we suppose that all vortices have equal strengths, so $H$, $\Omega$, $\Phi$ and the equations of motion 
are given by \eqref{eq:Ham-Om-equal},
\eqref{eq:motion} and 
 \eqref{eq:Phi}. We also recall from section \ref{ss:equal-strengths}
 that the total symmetry group is $\hat G=S_N\times SO(3)$ and acts on $M$ symplectically by \eqref{eq:Full-symmetry},
 and, with respect to this action, $H$ is invariant and  the equations of motion \eqref{eq:motion} are equivariant.
 
 \subsubsection{$\Z_m$-symmetric  configurations and discrete reduction} 
 \label{SS:discrete-reduction}
 
%

Consider a positive integer $m\geq 1$ and denote by $g_m\in SO(3)$ the matrix
\begin{equation}
\label{eq:def-gm}
g_m:=\exp \left (\frac{2\pi}{m} J_3 \right ) = \begin{pmatrix} \cos \frac{2\pi}{m}  & - \sin  \frac{2\pi}{m} & 0 \\ \sin  \frac{2\pi}{m} & \cos \frac{2\pi}{m}  & 0 \\ 0 & 0 & 1 \end{pmatrix}.  
\end{equation}
Assume that  $N=mn+p$ with $p=0,1$ or $2$, and let  $\tau_m\in S_N$ be the following permutation written as the product of
 $n$  disjoint cycles 
\begin{equation}
\label{eq:def-tau}
  \tau_m :=(1,\dots, m)(m+1,\dots, 2m)(2m+1,\dots, 3m)\cdots ((n-1)m+1, \dots, nm).
\end{equation}

\begin{definition}
\label{def:Z_m-symmetric}
Let $v=(v_1,\dots, v_N)\in M$. We say that $v$ is \defn{$\Z_m$-symmetric} if $(\tau_m,g_m).v=v$.
\end{definition}

The definition is tailored to identify the configurations $v\in M$  comprised of
$n$ latitudinal rings, each consisting of a regular $m$-gon and in the presence of $p$ vortices at the poles ($p=0,1$ or $2$). Note that if $m=1$ there is really no  ring
and the configuration may have no symmetries. In this case we follow the convention that $n=N$ and $p=0$.

Let $K_m$ be the subgroup of $\hat G$ generated by $(\tau_m,g_m)\in \hat G$. It is clear that $K_m$ is isomorphic to the 
cyclic group  $\Z_m$. The set of all $\Z_m$-symmetric configurations is then  
 $$\mbox{Fix}(K_m):=\{ v\in M \, : \, (\sigma,g).v=v \; \forall (\sigma,g)\in K_m\}\subset M.$$ 
Considering that $K_m$ is a subgroup of  $\hat G$ and the equations \eqref{eq:motion} are $\hat G$-equivariant, it
follows that $\mbox{Fix}(K_m)$ is invariant under the flow of \eqref{eq:motion}.
 Therefore, we may
 speak of $\Z_m$-symmetric solutions,  $\Z_m$-symmetric equilibria and  $\Z_m$-symmetric relative equilibria. 

We wish  to understand the structure of the set $\mbox{Fix}(K_m)$ and the restriction of \eqref{eq:motion} to it.
Suppose first that $m\geq 2$ and
consider the set $M_n$ which is  the cartesian product of $n$ copies of the unit sphere in $\R^3$ minus the collisions and poles. Namely,
\begin{equation*}
M_n:=\underbrace{ S^2\times \cdots \times S^2}_n \setminus (\Delta_n \cup P_n),
\end{equation*}
 where 
\begin{equation} 
\label{eq:def-Deltan-Pn}
\Delta_n:=\{u\in (S^2)^n \, : \, u_i=u_j \,
\mbox{for some $i\neq j$}  \}, \quad \mbox{and} \quad P_n:=\{u\in (S^2)^n \, : \, u_i=(0,0,\pm1) \,
\mbox{for some $i$}  \} .
\end{equation}
Elements of $M_n$ should be interpreted as {\em generators} of the $n$  latitudinal rings
that make up a $\Z_m$-symmetric configuration. Note that the removal of the set $P_n$ in the definition of $M_n$ above is necessary since the presence of a ring generator at the pole leads to a collision if $m\geq 2$. For $m=1$ we will
instead assume that $M_n=M$ (we do not need to remove $P_n$ since the presence of a ``generator" at the pole
does not lead to a collision).

 For our purposes it is convenient to  equip $M_n$ with the symplectic form $\omega_{M_n}:=m\sum_{j=1}^{n}\pi _{j}^{\ast }\omega _{S^{2}}$
and consider the  Hamiltonian function $h:M_n\to \R$ given by
\begin{equation}
\label{eq:h(u)}
h(u):=\sum_{j=1}^{n}\left( -\frac{m}{4}\sum_{i=1}^{m-1}\ln \left\|
g_m^{i}u_{j}-u_{j}\right\| ^{2}-\frac{m}{2}\sum_{j<j^{\prime }\leq
n}\sum_{i=1}^{m}\ln \left\| g_m^{i}u_{j}-u_{j^{\prime }}\right\| ^{2}-%
\frac{m}{2}\sum_{f\in F_p}\ln \left\| u_{j}-f\right\| ^{2}\right) ,
\end{equation}
where $u=(u_1,\dots, u_n)\in M_n$ and $F_p$ is the ordered set defined according to the value of $p=N-mn\in \{0,1,2\}$ as 
\begin{equation}
\label{eq:defFp}
F_0:=\emptyset, \qquad F_1:=\{(0,0,1)\},  \qquad F_2:=\{(0,0,1),(0,0,-1)\}.
\end{equation}
The Hamiltonian vector field on $M_n$ corresponding to  $h$ and the symplectic form $\omega_{M_n}$
 is given by the \defn{reduced system}\footnote{note that our use of the terminology ``reduced system" does not mean that we are passing to the orbit space of a group action. As follows from Theorem \ref{th-main-symmetry}, it is more appropriate to think of a ``discrete reduction" in which one restricts
 the system to a connected component of the invariant set $\mbox{Fix}(K_m)$.}
\begin{equation}
\label{eq:reduced-system}
\dot u_j = -\frac{1}{m} u_j \times \nabla_{u_j} h(u), \qquad j=1,\dots, n.
\end{equation}

Finally, consider  the map $\rho:M_n \to M$  defined 
according to the value of $p=N-mn\in \{0,1,2\}$,  by
\begin{equation}
\label{eq:def-rho}
\rho (u_{1},\dots ,u_{n}) =(g_m^{1}u_{1},\dots
,g_m^{m}u_{1},g_m^{1}u_{2},\dots
,g_m^{m}u_{2},\dots,g_m^{1}u_{n},\dots ,g_m^{m}u_{n} ,F_p),
\end{equation}
where $F_p$ is given by \eqref{eq:defFp}. This mapping produces a $\Z_m$-symmetric configuration
on $M$ out of the ring generators $(u_1,\dots, u_n)$ in the natural way.

 The following theorem is inspired by  \cite[Theorem 3.5]{Luis}.
Among other things, it states that if an initial condition is made up of $n$ latitudinal rings, each of which is 
made up of a regular
 $m$-gon 
(and possibly $p$ poles), then its evolution by \eqref{eq:motion} will preserve such structure and 
the reduced system \eqref{eq:reduced-system} describes the evolution 
 of the ring generators $u_1, \dots, u_n$. 
 
Although  the discussion above and the theorem are primarily intended to be applied when $m\geq 2$, they also hold 
trivially for $m=1$  (see Remark \ref{rmk:m1.1}). It is convenient  to allow  $m=1$ 
in our discussion to study  RE which possess no symmetries at a later stage of the paper. 
%
%
%
%
%

\begin{theorem}
\label{th-main-symmetry} Let $N\geq 3$. The following statements hold.

\begin{enumerate}
\item The set $\mbox{\em Fix}(K_m)$ is an embedded submanifold of $M$ whose connected components 
are  diffeomorphic to $M_n$ and are invariant under the flow of  \eqref{eq:motion}.

\item The mapping $\rho$ defined by \eqref{eq:def-rho}  is a diffeomorphism from $M_n$ 
to  a connected component of $\mbox{\em Fix}(K_m)$  which satisfies $h=H\circ \rho$ and conjugates the flows of \eqref{eq:reduced-system}  and \eqref{eq:motion}.
That is, $t\rightarrow u(t)$ is a solution of \eqref{eq:reduced-system} if
and only if $t\mapsto \rho (u(t))$ is a solution of \eqref{eq:motion}.

%
%

\item Let $SO(2)$ be given by \eqref{eq:def-SO(2)}. The diagonal action of $SO(2)$ on $M_n$ is free, symplectic and has momentum
map $\phi:M_n\to \R$ given by 
\begin{equation}
\label{eq:phi}
\phi(u)= m \left ( \sum_{i=1}^nu_i \right ) \cdot e_3,
\end{equation}
and, up to perhaps a constant, $\phi = \Phi_3\circ \rho$.
Moreover, the reduced Hamiltonian $h$ given by \eqref{eq:h(u)} is $SO(2)$-invariant and hence,  as a consequence, \eqref{eq:reduced-system} is $SO(2)$-equivariant
and $\phi$ is a first integral.

\item Assume that $m\geq 2$. The centre of vorticity of elements of $\mbox{\em Fix}(K_m)$ is parallel to $e_3$ i.e. $\Phi(v)\times e_3=0, \, \forall v\in \mbox{\em Fix}(K_m)$.
%
%
\end{enumerate}
\end{theorem}

\begin{proof}
(i)-(ii) The mapping $\rho$ defined by \eqref{eq:def-rho} is injective and takes values on $\mbox{Fix}(K_m)$. First note that if $p=0$ it is easy to show that  
\begin{equation*}
\mbox{Fix}(K_m)=\{(g_m^{1}u_{1},\dots
,g_m^{m}u_{1},g_m^{1}u_{2},\dots
,g_m^{m}u_{2},\dots,g_m^{1}u_{n},\dots ,g_m^{m}u_{n}) \, : \, u=(u_1,\dots, u_n) \in M_n\},
\end{equation*}
 which shows that $\rho$ is onto and hence bijective. It is clear that $\rho^{-1}$ is smooth so in this case $\rho$ is a diffeomorphism. Now, if $p=1$ or $p=2$, then  $\mbox{Fix}(K_m)$
 is the disjoint union of two diffeomorphic connected components  $\mbox{Fix}(K_m)^\pm$ which are given by
 \begin{equation*}
\mbox{Fix}(K_m)^\pm:=\{(g_m^{1}u_{1},\dots
,g_m^{m}u_{1},g_m^{1}u_{2},\dots
,g_m^{m}u_{2},\dots,g_m^{1}u_{n},\dots ,g_m^{m}u_{n},(0,0,\pm 1)) \, : \, u=(u_1,\dots, u_n) \in M_n\},
\end{equation*}
if $p=1$, and instead by
 \begin{equation*}
 \begin{split}
\mbox{Fix}(K_m)^+&:=\{(g_m^{1}u_{1},\dots
,g_m^{m}u_{1},\dots,g_m^{1}u_{n},\dots ,g_m^{m}u_{n},(0,0, 1), (0,0,-1)) \, : \, u=(u_1,\dots, u_n) \in M_n\} ,\\
\mbox{Fix}(K_m)^-&=\{(g_m^{1}u_{1},\dots
,g_m^{m}u_{1},\dots,g_m^{1}u_{n},\dots ,g_m^{m}u_{n},(0,0, -1), (0,0,1)) \, : \, u=(u_1,\dots, u_n) \in M_n\},
\end{split}
\end{equation*}
if $p=2$. In any case, arguing as in the case $p=0$, shows that $\rho$ is a diffeomorphism from $M_n$ onto  $\mbox{Fix}(K_m)^+$.  The invariance of
(the connected components of) $\mbox{Fix}(K_m)$ follows from the $\hat G$-equivariance of \eqref{eq:motion} since $K_m$ is a subgroup of $\hat G$.

Now denote by $\mbox{Fix}(K_m)^+$ the image of $M_n$ by $\rho$ for any $p=0,1,2$ (in other words,  $\mbox{Fix}(K_m)^+$ is defined as above for $p=1,2$ and instead equals
$\mbox{Fix}(K_m)$ if $p=0$).

In order to prove that $\rho$  conjugates the flows of \eqref{eq:reduced-system}  and \eqref{eq:motion}, we proceed as in the proof of  \cite[Theorem 3.5]{Luis} and argue that
$\rho$ is in fact a symplectomorphism from $M_n$ equipped with $\omega_{M_n}$ onto 
$\mbox{Fix}(K_m)^+$ equipped with the restriction of $\Omega$ (given by \eqref{eq:Ham-Om-equal}). Below we show that $h=H\circ \rho$  (with $H$ given by \eqref{eq:Ham-Om-equal}). Therefore
(see e.g. \cite[Proposition 5.4.4]{Marsden-Ratiu}), $\rho$ pulls-back  the restriction of the
Hamiltonian vector field of  $H$ with respect to $\Omega$ to $\mbox{Fix}(K_m)^+$ onto the Hamiltonian vector field of $h$ with respect to $\omega_{M_n}$. Hence, the solution curves of these vector fields are mapped
onto each other by $\rho$. In other words, $\rho$ maps solutions of 
 \eqref{eq:reduced-system}  into solutions of \eqref{eq:motion} as required.
 
The proof that $\Omega$ pulls-back to $\omega_{M_n}$ by $\rho$ is a simple generalisation of the proof of \cite[Lemma 3.10]{Luis} that we omit. We now show that indeed $h=H\circ \rho$ which completes the
proof. Starting from \eqref{eq:Ham-Om-equal}
and \eqref{eq:def-rho},  we compute for $u=(u_1,\dots, u_n)\in M_n$: 
\begin{equation}
H\circ \rho(u)=-\frac{1}{4}\sum_{(i,j)\neq (i^{\prime },j^{\prime })}\ln \left\|
g_m^{i}u_{j}-g_m^{i^{\prime }}u_{j^{\prime }}\right\| ^{2}-\frac{1}{2}%
\sum_{f\in F_p}\sum_{i=1}^{m}\sum_{j=1}^{n}\ln \left\|
g_m^{i}u_{j}-f\right\| ^{2}.  \label{eq:aux1-proposition-simp-red-Ham}
\end{equation}
Using   $g_m^{i}f=f$ for $f\in F_p$, gives%
\begin{equation}
\frac{1}{2}\sum_{f\in F_p}\sum_{i=1}^{m}\sum_{j=1}^{n}\ln \left\|
g_m^{i}u_{j}-f\right\| ^{2}=\frac{1}{2}\sum_{f\in F_p}\sum_{i=1}^{m}\sum_{j=1}^{n}\ln \left\|
g_m^{i}(u_{j}-f)\right\| ^{2}=\frac{m}{2}\sum_{j=1}^{n}\sum_{f\in F_p}\ln
\left\| u_{j}-f\right\| ^{2}\text{.}
 \label{eq:aux2-proposition-simp-red-Ham}
\end{equation}%
On the other hand,%
%
%
\begin{equation}
 \label{eq:aux3-proposition-simp-red-Ham}
\begin{split}
\frac{1}{4}\sum_{(i,j)\neq (i^{\prime },j^{\prime })}\ln \left\|
g_m^{i}u_{j}-g_m^{i^{\prime }}u_{j^{\prime }}\right\| ^{2}& =\frac{1}{4}%
\sum_{(i,j)\neq (i^{\prime },j^{\prime })}\ln \left\| g_m^{i-i^{\prime
}}u_{j}-u_{j^{\prime }}\right\| ^{2} \\
& =\frac{1}{4}\sum_{j=1}^{n}\sum_{\substack{ i,i^{\prime }=1  \\ (i\neq
i^{\prime })}}^{m}\ln \left\| g_m^{i-i^{\prime }}u_{j}-u_{j}\right\|
^{2}+\frac{1}{4}\sum_{j\neq j^{\prime }}\sum_{i,i^{\prime }=1}^{m}\ln
\left\| g_m^{i-i^{\prime }}u_{j}-u_{j^{\prime }}\right\| ^{2} \\
& =\sum_{j=1}^{n}\left( \frac{m}{4}\sum_{i=1}^{m-1}\ln \left\|
g_m^{i}u_{j}-u_{j}\right\| ^{2}\right) +\sum_{j\neq j^{\prime }}\left( 
\frac{m}{4}\sum_{i=1}^{m}\ln \left\| g_m^{i}u_{j}-u_{j^{\prime
}}\right\| ^{2}\right) \\
& =\sum_{j=1}^{n} \left (  \frac{m}{4}\sum_{i=1}^{m-1}\ln \left\|
g_m^{i}u_{j}-u_{j}\right\| ^{2} + \frac{m}{2} \sum_{j< j^{\prime }\leq n}
\sum_{i=1}^{m}\ln \left\| g_m^{i}u_{j}-u_{j^{\prime
}}\right\| ^{2} \right ).
 \end{split}
\end{equation}%
Substitution of \eqref{eq:aux2-proposition-simp-red-Ham} and \eqref{eq:aux3-proposition-simp-red-Ham} into \eqref{eq:aux1-proposition-simp-red-Ham} and comparing with  \eqref{eq:h(u)} shows that
$H\circ \rho=h$ as required.

(iii) The action is clearly symplectic since the area form on the sphere is invariant under rotations. The action is free since the action of $SO(2)$ on $S^2$ fixes only the North and South poles
and these are removed from $M_n$ (see the definition of  $P_n$ in \eqref{eq:def-Deltan-Pn}) if $m\geq 2$. If $m=1$, 
then $M_n=M$ and the freeness follows since $n=N\geq 3$.

 The Hamiltonian vector field of $\phi$ defined by \eqref{eq:phi} with respect to the symplectic
form $\omega_{M_n}$ defines the equations
\begin{equation*}
\dot u_j = -\frac{1}{m} u_j \times \nabla_{u_j} \phi(u)=e_3 \times u_j =J_3 u_j, \qquad j=1,\dots, n.
\end{equation*}
These equations coincide with the corresponding ones for the infinitesimal generator of the action corresponding to the matrix $J_3$ in the Lie algebra of the group $SO(2)$, which proves  that $\phi$ is indeed the momentum map.
It is immediate to check that $h$ is $SO(2)$-invariant directly from \eqref{eq:h(u)} using that $g_m\in SO(2)$ and $SO(2)$ is abelian.  The preservation of $\phi$ along the flow of \eqref{eq:reduced-system} 
and the $SO(2)$-equivariance of these equations follow from general results for Hamiltonian systems with symmetry (see e.g. \cite{Marsden-Ratiu}).

Next, using the expression of $\Phi$ given in \eqref{eq:Phi}, and the description of the set  $\mbox{Fix}(K_m)$ in terms of its connected components given in the proof of (i) above,
it is easy to see that if $v\in \mbox{Fix}(K_m)$ then there exists $u=(u_1,\dots, u_n)\in M_n$ such that
\begin{equation}
\label{eq:aux_id_proofThmSymmetry2}
\Phi(v)= \sum_{j=1}^n\sum_{i=1}^m g_m^iu_j + \sigma_pe_3,
\end{equation}
where $\sigma_p=0$ if $p=0,2$ and $\sigma_1=\pm 1$ (the sign depending on whether $v\in \mbox{Fix}(K_m)^+$ or $v\in \mbox{Fix}(K_m)^-$). 
Hence, using the  identity
\begin{equation}
\label{eq:aux_id_proofThmSymmetry}
\sum_{i=1}^m g_m^i = \begin{cases}   \mbox{Id}_3 \quad &\mbox{if} \quad m=1, \\ me_3 e_3^T\quad &\mbox{if} \quad m\geq 2,  \end{cases}
\end{equation}
 we obtain $\phi = \Phi_3\circ \rho$ if $p=0,2$ and  $\phi = \Phi_3\circ \rho -1$ if $p=1$.
 
 (iv) For $m\geq 2$, using \eqref{eq:aux_id_proofThmSymmetry2} and \eqref{eq:aux_id_proofThmSymmetry} we obtain
\begin{equation*}
\Phi(v)=\left ( m \sum_{j=1}^n u_j \cdot e_3 + \sigma_p \right )e_3,
\end{equation*}
which is parallel to $e_3$ as claimed. 
%
%
\end{proof}

\begin{remark}
\label{rmk:m1.1}
Looking ahead at the continuation of RE without symmetries, it is useful to note how the above discussion 
specialises for $m=1$.  In this case $g_m$ is the $3\times 3$ identity matrix and $\tau_m$ is the identity permutation
so $K_m$ is the trivial subgroup of $\hat G$. It follows that $\mbox{Fix}(K_m)=M$, $\rho$ is the identity map on $M$,  $h=H$, and the systems \eqref{eq:motion}   and  \eqref{eq:reduced-system}
coincide.
\end{remark}

\subsubsection{Symmetric relative equilibria: definition and main properties}
In view of item  (iii) of Theorem \ref{th-main-symmetry},  the reduced system \eqref{eq:reduced-system} is equivariant with respect to the action of $SO(2)$. Therefore, we may consider existence
of relative equilibria of \eqref{eq:reduced-system} with respect to this action. In analogy with Definition \ref{def:RE}, we say that the pair $(b,\omega)\in M_n\times \R$ is 
a RE of  \eqref{eq:reduced-system}  if 
\begin{equation*}
u(t)= (e^{\omega J_3 t}b_1,\dots, e^{\omega J_3 t}b_n),
\end{equation*}
is a solution  of  \eqref{eq:reduced-system} where $b=(b_1,\dots, b_n)$.
In analogy with item (i) of Proposition \ref{prop:RE-as-crit-pts}, we have:
\begin{proposition}
\label{prop:RE-as-crit-pts-red}
 $(b,\omega)\in M_n\times \R$ is a relative equilibrium of  \eqref{eq:reduced-system}
 if and only if $b\in M_n$ is a critical point of the augmented Hamiltonian $h_\omega:= h-\omega \phi :
 M_n\to \R$, where $\phi$ is defined in item (iii) of Theorem \ref{th-main-symmetry}.
\end{proposition}

It is straightforward to prove this result proceeding in  analogy to the proof that we presented of   
Proposition \ref{prop:RE-as-crit-pts}(i). It also follows by general considerations since  
the $SO(2)$ action on $M_n$  is symplectic and has momentum map  $\phi$ as stated in Theorem \ref{th-main-symmetry}.

In view of Theorem \ref{th-main-symmetry},  we have $h_\omega=H_\omega\circ \rho$ (up to perhaps a constant). Considering that 
$\rho$ is a diffeomorphism from $M_n$ onto a connected component of  $\mbox{Fix}(K_m)$ that we denote $\mbox{Fix}(K_m)^+$, it follows that the critical points of $H_\omega$ 
in $\mbox{Fix}(K_m)^+$ are in one-to-one correspondence with the critical points of $h_\omega$ in $M_n$. Therefore, if $(b,\omega)\in M_n\times \R$ is 
a RE of  \eqref{eq:reduced-system}  then $(\rho(b),\omega)\in M\times \R$ is
a $\Z_m$-symmetric RE of  \eqref{eq:motion}. Conversely, if $(a,\omega)\in M\times \R$ is a  $\Z_m$-symmetric RE of  \eqref{eq:motion} 
belonging to   $\mbox{Fix}(K_m)^+$, there exists a RE  $(b,\omega)\in M_n\times \R$ of \eqref{eq:reduced-system} such that $a=\rho(b)$.
Therefore, the search for $\Z_m$-symmetric RE of  \eqref{eq:motion} reduces to finding critical
points of $h_\omega$. If $m\geq 2$, this is a considerable simplification with respect to finding critical points of 
$H_\omega$ since the  $SO(3)$-symmetry of $H_0$ is broken by our construction and, as a consequence, it
will be possible to  apply the implicit function theorem to   $dh_\omega$ close to $\omega=0$.

Looking ahead at the implementation, we look for critical points of $h_\omega:M_n\to \R$ 
using  the method of Lagrange multipliers.
We write $\lambda=(\lambda_1,\dots, \lambda_n)\in \R^n$ and define the function
\begin{equation}
\label{eq:defhomstar}
h_{\omega }^*:\R^{3n}\times \R^n\to \R, \qquad h_{\omega }^*(u,\lambda ):=h_{\omega }(u)+m\sum_{j=1}^{n}\lambda
_{j}R_j(u),
\end{equation}%
where $R_j(u)=\frac{1}{2}\left ( \| u_{j}\| ^{2}-1\right )$ and with $h_\omega$ in \eqref{eq:defhomstar} interpreted as a function
of $u\in \R^{3n}$ in the obvious way. The method of Lagrange multipliers states that  $\hat u\in M_n$ is a critical point
of $h_\omega:M_n\to \R$ if and only if there exists $\hat \lambda\in \R^n$ such that $(\hat u,\hat \lambda)$ is
a critical point of $h_{\omega }^*$.

\begin{remark}
\label{rmk:defhstar}
The domain
of $h_{\omega }^*$ should actually be restricted  to avoid  points $u\in \R^{3n}$  at which the 
formula for $h_{\omega}$ is not defined. These include collisions but also elements  $u\in \R^{3n}$ for which a component
$u_j\in \R^3$ is parallel to $e_3$. We have decided to overlook this detail in the definition of $h_\omega^*$ to 
facilitate the presentation.
\end{remark}

\subsubsection{Symmetric relative equilibria: existence as zeros of a map and continuation}

Suppose that $(b_0,\omega_0)\in  M_n \times \R$ is a RE of \eqref{eq:def-SO(2)}. We do not assume that 
$\omega_0\neq 0$ so $b_0$ could be an equilibrium point, and in fact, this particular case will be very relevant for us. We want
to determine a branch of RE emanating from  $(b_0,\omega_0)$, namely, we want to find RE  $(b(\omega),\omega)\in M_n
\times \R$  parametrised by 
$\omega$ close to $\omega_0$ and satisfying $b(\omega_0)=b_0$.
In view of the discussion above, such RE correspond to critical points of $h^*_\omega$ given by \eqref{eq:defhomstar} so the
problem can be reformulated as the continuation of critical points of $h^*_\omega$ as a function of $\omega$.
The complication that arises  is that, for any value $\omega$, the function  $h_{\omega }^{\ast }$ is invariant under the   action of  $ SO(2)$ 
 on $\R^{3n}\times \R^n$ given by $g.(u,\lambda)=(gu_1,\dots, gu_n, \lambda)$ where $g\in SO(2)$ (and $SO(2)$ is given by \eqref{eq:def-SO(2)} as 
 usual).  Therefore, for a fixed value of $\omega$,  a critical point of $h^*_\omega$ in fact gives rise to an   $SO(2)$-orbit of critical points.  Our strategy to 
 deal with this complication is to  
  isolate an element of the critical orbit $SO(2).b_0$  in order to  apply the implicit function theorem.
   For this
  we follow the approach of   \cite{MunozFreireGalanetal,Ize} and  consider  the zeros of the 
 augmented map:\footnote{a remark similar to Remark \ref{rmk:defhstar} applies to the domain of definition of $F$.}
 \begin{equation}
 \label{eq:augmap}
 \begin{split}
 F: \; & \mathbb{R}^{3n+n+1}\times \mathbb{R}\rightarrow \mathbb{R}^{3n+n+1}\\
 &(u,\lambda ,\alpha ;\omega ) \mapsto (\nabla _{u}h_{\omega }^{\ast }(u,\lambda
)+\alpha \mathcal{J}_{3}u,R_{1}(u),...,R_{n}(u),\mathcal{J}_{3}b_0\cdot (u-b_0)),
\end{split}
\end{equation}
where $  \mathcal{J}_{3}:=\mbox{diag}(J_3,\dots, J_3)$ is the $3n\times 3n$ diagonal block matrix with $n$ entries equal to $J_3$. 
We employ the terminology of  \cite{MunozFreireGalanetal,Ize} and refer to  $\alpha\in \R$ as the {\em unfolding parameter}. We also
recall from these references that  the condition $\mathcal{J}_{3}b_0\cdot (u-b_0)=0$ requires $ u\in \R^{3n}$ to belong
to a Poincar\'e section
of $SO(2).b_0$  through $b_0$.  In particular, such condition cannot be satisfied for $u\in SO(2). b_0$ unless $u=b_0$
so $b_0$ is the unique representative of the critical orbit $SO(2).b_0$ that produces a zero of $F$.
 Finally, note that the conditions $R_j(u)=0$ guarantee
that $u\in M_n$.

The following proposition explains how zeros of $F$ give rise to $\Z_m$-symmetric RE of our problem. The proposition is valid for all $m\geq 1$.

 \begin{proposition}
 \label{prop:zerosFareRE}
Let $N\geq 3$. The zeros of the augmented map $F$ defined by \eqref{eq:augmap}  give rise to relative equilibria of the $N$-vortex problem
consisting of  $n$-rings of regular $m$-gons and $p$-poles. More precisely, if 
$F(\hat u,\hat \lambda,\hat \alpha ;\hat \omega)=0$, then
 $\hat u \in M_n$ and $(\hat u, \hat \omega)$ is a RE of the reduced system \eqref{eq:reduced-system}. Moreover,  if we denote $\hat v:=\rho(\hat u)$ then $\hat v\in M$ is $\Z_m$-symmetric and 
  $(\hat v,\hat \omega)$ is a $\Z_m$-symmetric RE of the full system \eqref{eq:motion}.
\end{proposition}

Before presenting the proof, we note that
\begin{equation}
\label{eq:J3uneq0}
\mathcal{J}_{3}u\neq 0,\qquad  \mbox{for all $u\in M_n$.}
\end{equation}
The reason is that  $\mathcal{J}_{3}u$ is the value of the infinitesimal
generator at $u$ of the $SO(2)$-action corresponding to the Lie algebra element $J_3$. Such infinitesimal generator cannot vanish since
the action is   free (see item (iii) in Theorem \ref{th-main-symmetry}). Condition \eqref{eq:J3uneq0} will be used in the proof below and also ahead in the proof of Theorem  \ref{thm:continuation}.

\begin{proof}
Suppose that $(\hat u,\hat \lambda,\hat \alpha ;\hat \omega)\in  \mathbb{R}^{3n+n+1}\times \mathbb{R}$ belongs to the 
domain of $F$ and satisfies  $F(\hat u,\hat \lambda,\hat \alpha ;\hat \omega)=0$. The conditions $R_i(\hat u)=0$ imply that
$\| \hat u_i \|^2=1$ for all $i=1,\dots, n$. Moreover, since $F$ is well-defined at $ (\hat u,\hat \lambda,\hat \alpha ;\hat \omega)$ then 
$\hat u_i\neq \hat u_j$ for all $i,j=1,\dots, n$. For $m\geq 2$ the well-definiteness of $F$ also implies that $\hat u_i\neq (0,0,\pm 1)$ for all $i=1,\dots, n$. These conditions
imply that $\hat u \in M_n$.

Now note that the vanishing of the first component of $F(\hat u,\hat \lambda,\hat \alpha ;\hat \omega)$ implies 
\begin{equation}
\label{eq:auxZeroFProp}
\nabla _{u}h_{\hat \omega }^{\ast }(\hat u,\hat \lambda
)+\hat \alpha \mathcal{J}_{3}\hat u=0.
\end{equation}
 We claim that this condition can only hold if $\hat \alpha=0$. To see this  note that the $SO(2)$-invariance of $h^*_{\hat \omega}$ implies
\begin{equation*}
h^*_{\hat \omega}(\exp(sJ_3)\hat u_1, \dots, \exp(sJ_3)\hat u_n, \hat \lambda)=h^*_{\hat \omega}(\hat u_1, \dots, \hat u_n, \hat \lambda)
\end{equation*}
for all $s\in \R$. Differentiating with respect to $s$ and evaluating at $s=0$ gives  $\nabla
_{u}h_{\hat \omega }^{\ast }(\hat u, \hat \lambda )\cdot \mathcal{J}_{3}\hat u=0$. Therefore, taking the scalar
product on both sides of \eqref{eq:auxZeroFProp} with $ \mathcal{J}_{3}\hat u$ gives $\hat \alpha \| \mathcal{J}%
_{3}\hat u\|^{2}=0$. But this implies that $\hat \alpha=0$ because  of \eqref{eq:J3uneq0} and since $\hat u\in M_n$. 

Considering that $\hat \alpha=0$, the condition \eqref{eq:auxZeroFProp} simplifies to 
$\nabla _{u}h_{\hat \omega }^{\ast }(\hat u,\hat \lambda)=0$ which together with the relations $R_i(\hat u)=0$ imply that $(\hat u,\hat \lambda)$ is a critical point of $h_{\hat \omega }^{\ast }$. By the Lagrange multiplier theorem we conclude that
$\hat u\in M_n$ is a critical point of $h_{\hat \omega}:M_n\to \R$ and hence, by Proposition 
\ref{prop:RE-as-crit-pts-red}, $(\hat u, \hat \omega)$ is a RE of the reduced system  \eqref{eq:reduced-system}.
The conclusions about $\hat v=\rho (\hat u)$ follow directly from Theorem \ref{th-main-symmetry} (see the discussion 
after Proposition \ref{prop:RE-as-crit-pts-red}).
%
%
\end{proof}

The following theorem establishes the existence of a unique branch $(b(\omega),\omega)$ of RE emanating from
$(b_0,\omega_0)$ as discussed above, and provides a method to determine it as zeros of the augmented map 
$F$ defined  above \eqref{eq:augmap}.  The theorem requires a non-degeneracy assumption on the critical orbit $SO(2).b_0$ 
of $h_{\omega_0}$ (see Definition \ref{def:Bott}), but $\omega_0$ may be zero.   

\begin{theorem}
\label{thm:continuation} Let $m\geq 1$, $N\geq 3$, and
let $(b_0,\omega_0)\in M_n\times \R$ be a RE of the reduced system  \eqref{eq:reduced-system} and
suppose that $SO(2).b_0$ is a non-degenerate critical manifold of $h_{\omega_0}: M_n\to \R$.
There exists $\varepsilon >0$ such that the following statements hold 
for all $\omega\in (\omega_0-\varepsilon,\omega_0+\varepsilon)$:
\begin{enumerate}
\item there exists a unique $b=b(\omega)\in M_n$ with $b(\omega_0)=b_0$ and depending smoothly on $\omega$ such that  $(b(\omega),\omega)$ is a 
 a RE  of \eqref{eq:reduced-system}.
Furthermore, if we denote $a(\omega):=\rho(b(\omega))$, then $a(\omega)\in M$ is $\Z_m$-symmetric and 
$(a(\omega),\omega)$ is a  $\Z_m$-symmetric RE of the full system \eqref{eq:motion}.

\item apart from the existence of $b(\omega)$ with the properties stated above, there exist unique smooth functions  
 $\lambda=\lambda(\omega)\in \R^n$ and $\alpha = \alpha(\omega)\in \R$ 
such that  
$F(b(\omega),\lambda(\omega),\alpha(\omega);\omega)=0$.\footnote{in fact one has $\alpha(\omega)=0$ as can be 
concluded from the proof of Proposition  \ref{prop:zerosFareRE}.}
\end{enumerate}

%
\end{theorem}

%
%

\begin{proof}
Since $(b_0,\omega_0)\in M_n\times \R$ is a RE then $b_0$ is a critical point of $h_{\omega_0}$ and hence
there exists $\lambda_0\in \R^n$ such that $(b_0,\lambda_0)$ is a critical point of $h_{\omega_0}^*$ and, therefore,
$F(b_0,\lambda_0,0;\omega_0)=0$. We will apply the implicit function theorem to prove (ii). The proof of (i) then
follows from Proposition \ref{prop:zerosFareRE}.

The application of the implicit function theorem  requires
 $D_{(u,\lambda ,\alpha )}F(b_0,\lambda_0,0;\omega_0)$ to be invertible.
   To prove that this is indeed the case, 
suppose that the vector
$( \delta u, \delta \lambda, \delta \alpha)\in \R^{3n}\times \R^n \times \R$  belongs to the null-space of $D_{(u,\lambda ,\alpha )}F(b_0,\lambda_0,0;\omega_0)$. Below we show that  necessarily 
 $(\delta u, \delta \lambda, \delta \alpha)=0$.

Denote by $b_j^0$ the components of $b_0$ so that $b_0=(b_1^0,\dots ,b_n^0)\in M_n$. Similarly, suppose
that $\lambda_0=(\lambda_1^0,\dots, \lambda_n^0)\in \R^n$, that $\delta u=(\delta u_1,\dots, \delta u_n)\in \R^{3n}$ and
$\delta \lambda = (\delta \lambda_1,\dots, \delta \lambda_n)\in \R^n$. One computes
\begin{equation*}
\begin{split}
D_{(u,\lambda ,\alpha )}&F(b_0,\lambda_0,0;\omega_0)(\delta u, \delta \lambda, \delta \alpha) = \\ & \left (\nabla
_{u}^{2}h_{\omega_0}^{\ast }(b_0,\lambda_0)\delta u+(b_{1}^0\, \delta \lambda _{1},...,b_{n}^0\, \delta\lambda _{n})+
\mathcal{J}_{3}b_0\delta \alpha \, , \, b_{1}^0\cdot \delta u_{1} \, , \, \dots \, , \, b^0_{n}\cdot  \delta u_{n} \, ,\, \mathcal{J}_{3}b_0\cdot
\delta u \right ).
\end{split}
\end{equation*}%
Hence, the assumption that  $(\delta u, \delta \lambda, \delta \alpha)$ is a null-vector of $D_{(u,\lambda ,\alpha )}F(b_0,\lambda_0,0;\omega_0)$
 yields
\begin{equation}
\label{eq:aux-proof-cont-thm}
\begin{split}
&\nabla_{u}^{2}h_{\omega_0}^{\ast }(b_0,\lambda_0)\delta u+(b_{1}^0\, \delta \lambda _{1},...,b_{n}^0\, \delta\lambda _{n})+
\mathcal{J}_{3}b_0\delta \alpha=0,\\
&b^0_{j}\cdot \delta u_{j}=0, \qquad j=1,\dots, n, \\
&\mathcal{J}_{3}b_0\cdot
\delta u=0.
\end{split}
\end{equation}

First note that the embedding $M_n\hookrightarrow (\R^3)^n$ leads to the identification  
\begin{equation*}
T_{b_0}M_n=  \{{\bf w}= (w_{1}, \dots, w_{n})\in(\mathbb{R}^{3})^{n} \, :\, 
b^0_{j}\cdot w_{j} =0, \; j=1,\dots, n \, \},
\end{equation*}
so the conditions   $b^0_{j}\cdot \delta u_{j}=0$ for all $j=1,\dots, n$ imply  $\delta u\in T_{b_0}M_n$.

Next note that the   $SO(2)$-invariance of $h_{\omega_0}^{\ast }$ implies 
$\nabla _{u}^{2}h_{\omega_0}^{\ast }(b_0,\lambda_0) \mathcal{J}_{3}b_0=0$, or, equivalently, \\
 $(\mathcal{J}_{3}b_0)^T\nabla _{u}^{2}h_{\omega_0}^{\ast }(b_0,\lambda_0)=0$. Therefore, multiplying both sides
 of the first equation of \eqref{eq:aux-proof-cont-thm}
on the left by $(\mathcal{J}_{3}b_0)^T$ gives 
\begin{equation*}
-\sum_{j=1}^N(b_{j}^0)^TJ_3b_j^0\delta \lambda _{j}+ \|\mathcal{J}_{3}b_0\|^2 \delta \alpha = \|\mathcal{J}_{3}b_0\|^2 \delta \alpha=0,
\end{equation*}
which implies that $\delta \alpha=0$ in view of \eqref{eq:J3uneq0} since $b_0\in M_n$.
Therefore, the first equation of \eqref{eq:aux-proof-cont-thm} becomes
\begin{equation}
 \label{eq:aux-proof-cont-thm2}
\nabla_{u}^{2}h_{\omega_0}^{\ast }(b_0,\lambda_0)\delta u+(b_{1}^0\, \delta \lambda _{1},...,b_{n}^0\, \delta\lambda _{n})=0.
\end{equation}
Let ${\bf w}= (w_{1}, \dots, w_{n})\in T_{b_0} M_n$ and multiply both sides of \eqref{eq:aux-proof-cont-thm2}
on the left by ${\bf w}^T$. Given that  \\ ${\bf w}^T(b_{1}^0\, \delta \lambda _{1},...,b_{n}^0\, \delta\lambda _{n})
=\sum_{j=1}^n(w_j\cdot b_j^0)\delta\lambda_j=\sum_{j=1}^n 0\delta\lambda_j=0$, we obtain  
\begin{equation*}
{\bf w}^T\nabla_{u}^{2}h_{\omega_0}^{\ast }(b_0,\lambda_0)\delta u=0.
\end{equation*}
Considering that ${\bf w}^T\nabla_{u}^{2}h_{\omega_0}^{\ast }(b_0,\lambda_0)\delta u=d^2h_{\omega_0}(b_0)({\bf w},\delta u)$, and
that  the
above equation holds for arbitrary ${\bf w}\in T_{b_0} M_n$, we conclude that $\delta u \in \ker d^2h_{\omega_0}(b_0)$.
The assumption that $SO(2).b_0$ is a non-degenerate critical manifold of $h_{\omega_0}$
 then implies that $\delta u \in T_{b_0}(SO(2).b_0)$. But
$T_{b_0}(SO(2).b_0)$ is one-dimensional and is generated by the vector $\mathcal{J}_{3}b_0$, so we may write 
$\delta u = c\mathcal{J}_{3}b_0$ for some $c\in \R$.
But then, using the last equation in \eqref{eq:aux-proof-cont-thm}, we obtain $c \| \mathcal{J}_{3}b_0 \|^2=0$ which implies
$c=0$. Therefore, $\delta u=0$ and in view of  \eqref{eq:aux-proof-cont-thm2} we obtain
\begin{equation*}
b_{j}^0\, \delta \lambda _{j}=0 \qquad j=1,\dots, n.
\end{equation*}
Considering that $b_{j}^0$ is a unit vector for all $j$, 
the above conditions imply $\delta \lambda=(\delta \lambda _1,\dots, \delta \lambda _n)=0$.
\end{proof}

\begin{remark}
\label{rmk:continuation}
Recall that when $m=1$ one has   $M_n=M$ and  $h_\omega=H_\omega$ (up to perhaps a constant). Also, the systems \eqref{eq:motion} and
\eqref{eq:reduced-system} coincide and $\rho$ is the identity map on $M$. In view of these 
observations, we conclude that, for $m=1$, Theorem \ref{thm:continuation} specialises  as  a continuation result
for general (possibly non-symmetric) RE of the problem. It is  however important to notice that the 
non-degeneracy hypothesis will never be satisfied if $\omega_0=0$ since $H_0=H$ is $SO(3)$-invariant. 
It could however be satisfied for $\omega_0 \neq 0$. \end{remark}

The following corollary establishes  the existence of  $\mathbb{Z}_m$-symmetric RE arising from $\mathbb{Z}_m$-symmetric  equilibria.
In contrast with Theorem \ref{th:existenceRE}, this result is constructive and will be used to find RE emerging from the 
equilibria in Table \ref{ground-states} as zeros of the map $F$ according to their $\mathbb{Z}_m$-symmetries. Similarly to 
Remark \ref{rmk:continuation}, we note that the non-degeneracy 
hypothesis of the $SO(2)$-orbit in its statement cannot be satisfied if $m=1$ since in such case $M_n=M$ and $h=H$ which is $SO(3)$-invariant. On the other hand, such condition may be satisfied for $m\geq 2$, since 
in this case $h_0=h$ is only $SO(2)$-invariant and this is the situation met for the equilibria in Table \ref{ground-states}
(see Remark \ref{rmk:symmetric-branches} below).

\begin{corollary}
\label{cor:existenceREsymmetric}
Let  $a_0\in M$  be a $\mathbb{Z}_m$-symmetric  equilibrium of  \eqref{eq:motion} 
satisfying  $a_0=\rho(b_0)$ for a certain $b_0\in M_n$. By Theorem \ref{th-main-symmetry}, it follows
that $b_0$ is an equilibrium of \eqref{eq:reduced-system} and hence a critical point of $h$. 
Suppose that $SO(2).b_0$ is a non-degenerate critical manifold of $h:M_n\to \R$. Then there exists $\omega_0>0$ such that for any 
$\omega
\in (-\omega _{0},\omega _{0})$ there exists a unique $\mathbb{Z}_m$-symmetric RE $(a(\omega),\omega)$ depending smoothly on $\omega$ and such that  $a(\omega)\to a_0$ as $\omega \to 0$. Moreover, $a(\omega)$ is 
determined by the condition $a(\omega)=\rho(b(\omega))$ where $b(\omega)$ satisfies $F(b(\omega),\lambda(\omega),0;\omega)=0$ for a unique $\lambda=\lambda(\omega)\in \R^n$ which depends smoothly on $\omega$.
\end{corollary}

The proof follows by applying Theorem \ref{thm:continuation} to the RE $(b_0,\omega_0)$ with $\omega_0=0$.
We remark that the hypothesis that  $a_0=\rho(b_0)$ for a certain $b_0\in M_n$ is always satisfied,  perhaps after a 
reflection about the equatorial plane that places any vortices at  the poles according to the ordering  of the
sets $F_1$, $F_2$ in \eqref{eq:defFp}.

\begin{remark}
\label{rmk:symmetric-branches}
The hypothesis that $SO(2).b_0$  is a non-degenerate critical manifold of $h:M_n\to \R$ is automatically satisfied if  $SO(3).a_0$ is a non-degenerate critical manifold of $H:M\to \R$. 
Except for $N=7$, all configurations in Table \ref{ground-states} are  non-degenerate minima of $H:M\to \R$ (see Section \ref{ss:CAP-minimisers}). Therefore, Corollary \ref{cor:existenceREsymmetric} implies the existence of a branch of  $\mathbb{Z}_m$-symmetric RE
emanating from these equilibria for each symmetry indicated in Table \ref{ground-states} and illustrated in Figures \ref{Fig:GS-Known} and \ref{Fig:GS-Conjecture}.
\end{remark}

\section{Nonlinear stability analysis of relative equilibria}
\label{sec:Stability}
We now focus on the nonlinear stability of relative equilibria using the energy-momentum
method described in subsection \ref{ss:energy-momentum-method}.

Consider a relative equilibrium $(a,\omega)\in M\times \R$ 
and recall from Proposition \ref{prop:RE-as-crit-pts}  that $a$ is a critical point of the augmented Hamiltonian $H_\omega$ (defined by \eqref{eq:aug-Ham}). As explained in subsection \ref{ss:energy-momentum-method}, the energy momentum 
examines the positiveness of the restriction of $d^2H_\omega(a)$ to a symplectic slice $\mathcal{N}_a$ which is an appropriate subspace of $T_aM$.
Throughout this section we identify\footnote{Appendix \ref{s:geometry} reviews several standard constructions that may be useful to follow 
the geometric calculations of this section. }
\begin{equation}
\label{eq:tangent-space} 
T_aM=V_a:= \{{\bf w}= (w_{1}, \dots, w_{N})\in(\mathbb{R}^{3})^{N} \, :\, 
a_{j}\cdot w_{j} =0, \; j=1,\dots, N \, \},
\end{equation}
where $a=(a_1,\dots, a_N)\in M$. 
 The 
space $V_a$ is equipped with the inner product $\langle \cdot , \cdot \rangle$ which coincides with 
the restriction of the standard euclidean product on $(\R^3)^N$. Such inner product is induced by the
Riemannian metric on $S^2$ that underlies  the vortex dynamics.\footnote{We recall that vortex
dynamics on a surface require the choice of a metric, which in our case is the standard metric on  $S^2$. We refer the reader to \cite{BoKo15} 
for  a clear explanation of the role of the  geometry of the surface in the equations of vortex motion.
}

A general symplectic slice $\mathcal{N}_a$ valid for any RE $(a,\omega)$ with $\mu=\Phi(a)\neq 0$ is given in Appendix \ref{ss:NonSymmSympSlice}.
Below we suppose that $(a,\omega)$ is $\mathbb{Z}_m$-symmetric  with $m\geq 2$ and construct a specific one that takes this symmetry into account
to give a block diagonalisation of $\left . d^2H_\omega \right |_{\mathcal{N}_a}$. Our construction is greatly influenced by previous work of Laurent-Polz, Montaldi and Roberts
\cite{Mo11}. We first give a reformulation of their construction of a symmetry-adapted basis that is more suitable to the implementation of CAPs in sections \ref{ss:Isotypic},
\ref{ss:ConstructionSymmetricSympSlice} and  \ref{ss:matricesPl}, and then present
a further simplification based on the complex structure of some blocks in section \ref{ss:complex-structure}.
Finally, we present a summary of the block matrices and the stability test in section \ref{ss:Summary}.
 Despite their technicality, the constructions in this section are very useful to obtain larger 
ranges of application of CAPs since they prevent the clustering of eigenvalues of large matrices. The implementation of CAPs of stability of branches of RE based on this  approach  is discussed in section \ref{ss:CAP-stability} and 
exemplified in section \ref{s:CAPs-application}.

\subsection{Isotypic decomposition of $V_a=T_aM$ and block diagonalisation of $d^2H_\omega(a)$}
\label{ss:Isotypic}

Assume that the RE $(a,\omega)$ is $\mathbb{Z}_m$-symmetric with $m\geq 2$ (Definition \ref{def:Z_m-symmetric}). In this section we 
consider the associated  linear action of the group $K_m\cong \mathbb{Z}_m$ (defined in section \ref{SS:discrete-reduction} above)
 on $V_a=T_aM$,  
give the isotypic decomposition of $V_a$ and explain how
it yields a block diagonalisation of $d^2H_\omega(a)$.

\subsubsection{A linear action of $K_m\cong \mathbb{Z}_m$ on $V_a=T_aM$}
 Following  the notation of subsection \ref{SS:discrete-reduction},  assume that $a$ is comprised of  $n\geq 1$ rings, each with 
$m \geq 2$ vortices and $p$ poles ($p$ equal to $0$, $1$ or $2$) so that  $N=mn+p$.  Because of the permutation  symmetry of the vortices, we may arrange them 
 in a convenient fashion. We assume that  if the North or South
pole are present, then they appear in the last entry (or entries) of $a$. Moreover, in accordance with \eqref{eq:def-rho}, we assume
throughout this section that
 \begin{equation}
 \label{eq:REfull}
a=(a_{1,1},\dots, a_{1,m}, a_{2,1},\dots, a_{2,m}, \dots, \dots, \dots,  a_{n,1},\dots, a_{n,m}; p_1,p_2).
\end{equation}
In our notation, for $j\in \{1,\dots, n\}$ the entries  $(a_{j,1},\dots, a_{j,m})$ denote  the location of the $m$ vortices of the $j^{th}$ ring ordered in an easterly 
direction so that
\begin{equation}
\label{eq:ring-convention}
a_{j,k}=g_m^ku_{j},\qquad  j=1,\dots, n, \quad k=1,\dots, m,
\end{equation}
where we recall from \eqref{eq:def-gm} that $g_m=\exp \left ( \frac{2\pi}{m}J_3 \right )$ and 
\begin{equation}
\label{eq:ring-generators}
u_j:=(x_j,y_j,z_j)= a_{j,m} \in S^2,
\end{equation}
 is a  {\em generator}  of the $j^{th}$ ring. 
The entries $p_1$ and $p_2$ in \eqref{eq:REfull} are the poles (if present). 

Under the assumption that $a$ is given as in \eqref{eq:REfull}, it follows that $a\in \mbox{Fix}(K_m)$, where we recall from subsection \ref{SS:discrete-reduction}
that $K_m$ denotes  the subgroup of $\hat G$, isomorphic to $\mathbb{Z}_m$, which is
 generated by $(\tau_m,g_m)\in \hat G$ with $\tau_m$ given by \eqref{eq:def-tau}. 
 We simplify the notation and denote
 $$\kappa:=(\tau_m,g_m)\in \hat G$$
  in what follows.
 
 The linearisation of the $K_m$ action on
$M$ at $a$ defines a linear  $K_m$-action on $V_a$ which we denote in terms of the representation
$\rho:K_m\to GL(V_a)$ determined by
\begin{equation}
\label{eq:rep-fmla}
\rho(\kappa){\bf w}= (g_m w_{\tau_m^{-1}(1)}, \dots, g_m w_{\tau_m^{-1}(N)}), \qquad \forall 
{\bf w}=(w_1, \dots, w_N)\in V_a.
\end{equation}
 Such representation is orthogonal with respect to the inner
product $\langle \cdot , \cdot \rangle$ on $V_a$, namely
\begin{equation*}
 \langle {\bf w}_1, {\bf w}_2 \rangle = \langle  \rho( \kappa){\bf w}_1,  \rho( \kappa){\bf w}_2 \rangle \quad \forall 
 {\bf w}_1,  {\bf w}_2 \in V_a.
\end{equation*}
Moreover, considering that $H_\omega$ is also $K_m$-invariant, it follows that $d^2H_\omega(a)$ is $K_m$-invariant, i.e.,
\begin{equation}
\label{eq:inv-d2H}
d^2H_\omega(a)(\rho ( \kappa) {\bf w}_1, \rho( \kappa) {\bf w}_2)=
d^2H_\omega(a)({\bf w}_1,  {\bf w}_2) \quad \forall {\bf w}_1,  {\bf w}_2\in V_a.
\end{equation}

\subsubsection{Isotypic decomposition of $V_a$ and block-diagonalisation of $d^2H_\omega(a)$}

 For $j\in \{1,...,n\}$ and $k\in \{1,...,m\}$ define the
vectors $b_{j,k},\,c_{j,k}\in \R^{3}$ by 
\begin{equation}
\label{eq:def-b-c}
b_{j,k}:= J_{3}a_{j,k},\qquad c_{j,k}:= b_{j,k}\times a_{j,k}.
\end{equation}%
Now define the following vectors in $\R^{3N}$ (given below in terms of block
vectors in $\R^{3}$): 
\begin{equation} \label{eq:BjkCjk_delta}
\begin{split}
B_{j,k}& =(0,\dots ,0,b_{j,k},0,\dots ,0;0,0), \quad  j=1,\dots, n, \; k=1,\dots, m, \\
C_{j,k} &=(0,\dots ,0,c_{j,k},0,\dots ,0;0,0), \quad  j=1,\dots, n, \; k=1,\dots, m, \\
\delta x_{1}& =(0,\dots ,0;e_{1},0), \\  \delta y_{1}& =(0,\dots ,0;e_{2},0), \\
\delta x_{2}& =(0,\dots ,0;0,e_{1}), \\  \delta y_{2} &=(0,\dots ,0;0,e_{2}), 
\end{split}%
\end{equation}%
where $e_{1}=(1,0,0)$ and $e_{2}=(0,1,0)$. (The nonzero entry of $B_{j,k}$
and $C_{j,k}$ occurs at the slot where $a_{j,k}$ appears in the expression %
\eqref{eq:REfull} for $a$.)

\begin{lemma}
\label{Lemma:basis1}
The vectors $B_{j,k}$, $C_{j,k}$, $\delta x_{s}$, $\delta y_{s}$ with $j\in \{1,\dots
,n\} $, $k\in \{1,\dots ,m\}$, $s\in \{1,\dots ,p\}$ form an orthogonal basis of $V_{a}$.
\end{lemma}
\begin{proof}
First note that the vectors $B_{j,k}, C_{j,k}$ are non-zero and  belong to $V_a$ 
since $b_{j,k}$ and $c_{j,k}$ are non-zero and perpendicular
to $a_{j,k}$. Similarly, the vectors $\delta x_{s}, \delta y_{s}\in V_a$ because $e_1$ and  $e_2$
are perpendicular to $p_s=\pm e_3$. It is also easy to check  that all of these vectors are mutually  
orthogonal. In particular,  they are linearly independent and
since there are $mn+p=\dim M=\dim V_a$ of them, they form a basis. 
\end{proof}
%
%

For the rest of the section we denote 
$$\zeta:=\frac{2\pi}{m}.$$ 
For $j\in \{1,\dots ,n\} $, $l\in \{0,1,\dots ,m-1\}$ we   define the complex  
vectors $\hat B_{j,l}, \hat C_{j,l} \in V_a\oplus iV_a$ by
\begin{equation} \label{eq:hatBhatC}
\hat{B}_{j,l} :=\sum_{k=1}^{m}e^{ilk\zeta }B_{j,k}, \qquad
\hat{C}_{j,l} :=\sum_{k=1}^{m}e^{ilk\zeta }C_{j,k}. 
\end{equation}
It is clear that  $\mbox{Re}(\hat{B}_{j,l}), \mbox{Im}(\hat{B}_{j,l}),
\mbox{Re}(\hat{C}_{j,l}), \mbox{Im}(\hat{C}_{j,l})\in V_a$
for any value of the indices $j,l$.  It is also useful to notice that $\hat{B}_{j,l}$ and $\hat{C}_{j,l}$
are real vectors (their imaginary part is zero) for $l=0$ and also when $l=\frac{m}{2}$ ($m$  even).

For $l=0,1, \dots ,\left[ \frac{m}{2%
}\right] $, we define the following \emph{real} subspaces of $V_a$:
\begin{equation*}
\begin{split}
V_{1}&=\bigoplus_{j=1}^{n}\left( \mbox{Re}\,\hat{B}_{j,1}\oplus \mbox{Im}\,%
\hat{B}_{j,1}\oplus \mbox{Re}\,\hat{C}_{j,1}\oplus \mbox{Im}\,\hat{C}%
_{j,1}\right) \bigoplus_{s=1}^{p}\left( \delta x_{s}\oplus \delta
y_{s}\right), \\
V_{l}&=\bigoplus_{j=1}^{n}\left( \mbox{Re}\,\hat{B}_{j,l}\oplus \mbox{Im}\,%
\hat{B}_{j,l}\oplus \mbox{Re}\,\hat{C}_{j,l}\oplus \mbox{Im}\,\hat{C}%
_{j,l}\right), \qquad l=0,2,3 \dots ,\left[ \frac{m}{2%
}\right] .
\end{split} 
\end{equation*}%

\begin{lemma}
\label{lemma:isotypic-dec}
\begin{enumerate}
\item The space $V_a$  decomposes as a direct sum   $V_a=\bigoplus_{l=0}^{\left[ \frac{m}{2}\right]}V_l$ 
of mutually orthogonal subspaces and   $d^2H_\omega(a)$ block diagonalises with respect to this decomposition. Namely, if ${\bf w}_1\in V_{l_1}$ and  ${\bf w}_2\in V_{l_2}$ with $l_1\neq l_2$ then
 \begin{equation*}
d^2H_\omega(a)({\bf w}_1, {\bf w}_2)=0.
\end{equation*}

 \item The sets $\tilde{\mathcal B}_l$ below are orthogonal bases of $V_l$.
\begin{equation*}
\begin{split}
&\tilde{\mathcal B}_0= \left \{ \, \hat{B}_{j,0} \, , \,\hat{C}_{j,0} \,  \right  \}_{j=1}^n, \\
&\tilde{\mathcal B}_1=\begin{cases}  \left \{ \, \hat{B}_{j,1} \, , \,\hat{C}_{j,1} \,   \right  \}_{j=1}^n \cup 
  \left \{ \delta x_s, \delta y_s   \right \}_{s=1}^p , \qquad &\mbox{if $m=2$}, \\
  \left  \{ \, \mbox{\em Re } (\hat{B}_{j,1}) \, , \, \mbox{\em Im} (\hat{B}_{j,1})  ,  \mbox{\em Re } (\hat{C}_{j,1}), \mbox{\em Im} (\hat{C}_{j,1})  \,  \right  \}_{j=1}^n \cup 
 \left  \{ \delta x_s, \delta y_s   \right \}_{s=1}^p \qquad &\mbox{if $m>2$},
\end{cases} \\
&\tilde{\mathcal B}_l=  \left  \{ \, \mbox{\em Re } (\hat{B}_{j,l}) \, , \, \mbox{\em Im} (\hat{B}_{j,l})  ,  \mbox{\em Re } (\hat{C}_{j,l}), \mbox{\em Im} (\hat{C}_{j,l})  \,
 \right  \}_{j=1}^n, \qquad  2\leq l < \left [ \frac{m}{2} \right ], \\
&\tilde{\mathcal B}_{\left [ \frac{m}{2} \right ]}= \begin{cases}
\left \{ \hat{B}_{j, \left [ \frac{m}{2} \right ]}, \hat{C}_{j,  \left[ \frac{m}{2} \right ]} \, \right   \}_{j=1}^n, \qquad  &\mbox{if $m\geq 4$ is even},\\
\left \{ \, \mbox{\em Re } (\hat{B}_{j,  \left [ \frac{m}{2} \right ]}) \, , \, \mbox{\em Im} (\hat{B}_{j,   \left [ \frac{m}{2} \right ]})  ,  
 \mbox{\em Re } (\hat{C}_{j,  \left [ \frac{m}{2} \right ]}), \mbox{\em Im} (\hat{C}_{j,   \left [ \frac{m}{2} \right ]})  \,  \right  \}_{j=1}^n,\qquad  &\mbox{if $m\geq 5$ is odd.} 
\end{cases}
\end{split}
\end{equation*}
In particular,  the dimension of the subspaces $V_l$ is as indicated in Table \ref{eq:dimVl}.
\end{enumerate}
\end{lemma}

\begin{table}[h]
\begin{center}
\begin{tabular}[c]{|c|c|c|c|c|c|}%
\hline
  & $\dim V_0$ & $\dim V_1$ & $\dim V_l, \;  2\leq l < \left [ \frac{m}{2} \right ]$ & $\dim V_{\left [ \frac{m}{2} \right ]}  $ \\ \hline
$m=2$ & $2n$ & $2n+2p$ & - & -  \\ \hline
$m=3$ & $2n$ & $ 4n+2p$ & - & -  \\ \hline
$m\geq 4 \quad \mbox{ even}$ & $2n$  &  $4n+2p$ &  $4n$& $2n$ \\ \hline
$m\geq 5 \quad \mbox{ odd}$ & $2n$ & $4n+2p$  &  $4n$& $4n$  \\ \hline
\end{tabular}
\end{center}
\caption{Dimension of the subspaces of the isotypic decomposition $V_a=\bigoplus_{l=0}^{\left[ \frac{m}{2}\right]}V_l$ according
to the values of $m$, $n$ and $p$.}
\label{eq:dimVl}
\end{table}

%
%

\begin{proof}
The mutual orthogonality of  the spaces $V_l$ is easily established using the definitions \eqref{eq:BjkCjk_delta} of $B_{j,k}$, $C_{j,k}$, 
$\delta x_s$, $\delta y_s$,  and \eqref{eq:hatBhatC} of $\hat B_{j,l}$, $\hat C_{j,l}$.
Moreover, the  vectors $ B_{j,k}$, $C_{j,k}$, $\delta x_s$, $\delta y_s$ also 
 form an orthogonal  basis of the {\em complex} vector space $V_a+iV_a$ 
  equipped with the restriction of the standard Hermitian inner product on $\C^{3N}$.
By standard results of the discrete Fourier transform, it follows that the
vectors $\hat B_{j,l}$, $\hat C_{j,l}$, $j=1,\dots, n$, $l=0,\dots, m-1$, together with 
$\delta x_s$,  $\delta y_s$, $s=1, 2$, are also an orthogonal  basis of $V_a+iV_a$. Now let ${\bf w}\in V_a$
(a real vector).
There exist complex scalars $\beta_{j,l}, \gamma_{j,l}, \lambda_s, \mu_s \in \C$ such that
\begin{equation*}
{\bf w}=\sum_{l=0}^{m-1} \left (\sum_{j=1}^{n} \beta_{j,l}\hat B_{j,l} +  \gamma_{j,l}\hat C_{j,l} \right ) +
 \sum_{s=1}^p\lambda_s \delta x_s +\mu_s \delta y_s.
\end{equation*}
Denote by $\overline{z}$ the complex conjugation of $z$ (which may be a scalar or a vector).
Using $\overline{{\bf w}}= {\bf w}$ together with the relations $\overline{\hat{B}_{j, l}}=\hat{B}_{j,m-l}$, 
 $\overline{\hat{C}_{j, l}}=\hat{C}_{j,m-l}$, $\overline{\delta x_s}=\delta x_s$, 
 $\overline{\delta y_s}=\delta y_s$, gives
\begin{equation*}
{\bf w}=\sum_{l=0}^{m-1} \left (\sum_{j=1}^{n} \overline{\beta_{j,l}}\hat{B}_{j,m-l}
 + \overline{ \gamma_{j,l}}\hat{C}_{j,m-l} \right ) +
 \sum_{s=1}^p\overline{\lambda_s} \delta x_s +\overline{\mu_s} \delta y_s.
\end{equation*}
Therefore, $\overline{\beta_{j,l}}=\beta_{j,m-l}$, $\overline{\gamma_{j,l}}=\gamma_{j,m-l}$, and
$\overline{\lambda_s}=\lambda_s$,  $\overline{\mu_s}=\mu_s$,
and we can write
\begin{equation*}
{\bf w}=\sum_{l=0}^{\left[ \frac{m}{2}\right] } \left (\sum_{j=1}^{n} \beta_{j,l}\hat B_{j,l} +
\overline{\beta_{j,l}} \overline{\hat B_{j,l} } +\gamma_{j,l}\hat C_{j,l}+\overline{\gamma_{j,l}}
\overline{\hat C_{j,l}} \right ) +
 \sum_{s=1}^p\lambda_s \delta x_s +\mu_s \delta y_s.
\end{equation*}
This shows that the real and  imaginary parts of the vectors $\hat{B}_{j,l}$ and $\hat{C}%
_{j,l}$ for $l=0,1,\dots \left[ \frac{m}{2}\right] $ and $j=1,\dots, n$, together with 
$\delta x_s$,  $\delta y_s$, $s=1, 2$,  form a basis of the real space $V_{a}$. In particular this proves that 
$V_a=\bigoplus_{l=0}^{\left[ \frac{m}{2}\right]}V_l$ as required. Moreover, recalling that $\hat{B}_{j,l}$ and $\hat{C}%
_{j,l}$ are real vectors for $l=0$ and $l=\frac{m}{2}$, $m$ even, this also proves item (ii).

On the other hand, we claim that  
\begin{equation}
\label{eq:Bjk-Cjk-as-eigenvectors}
\begin{split}
&\rho(\kappa) \hat{B}_{j,l}= e^{-il\zeta }\hat{B}_{j, l} \qquad \rho(\kappa) \hat{C}_{j,l}= e^{-il\zeta }\hat{C}_{j, l}, \qquad
j=1,\dots, n, \; l=0,\dots, \left[ \frac{m}{2}\right],
\\
&\rho(\kappa)\left( \delta x_s+ i\delta y_{s}\right) =e^{-i\zeta }\left( \delta x_{s}+i\delta y_{s}\right), \qquad s=1,\dots, p,
\end{split}
\end{equation}
where the identities are interpreted as equalities between   real and imaginary parts of both sides of the equations.

To prove \eqref{eq:Bjk-Cjk-as-eigenvectors} start by noticing that, in view of \eqref{eq:ring-convention} and
\eqref{eq:def-b-c}, we have  $e^{J_{3}\zeta }b_{j,k}=b_{j,k+1}$ and $e^{J_{3}\zeta
}c_{j,k}=c_{j,k+1}$ (with the index $k$ taken modulo $m$). These identities together with 
\eqref{eq:rep-fmla} imply
\begin{equation*}
\rho(\kappa)B_{j,k} =B_{j,k+1}, \qquad \rho(\kappa)C_{j,k} =C_{j,k+1},
\end{equation*}%
where, again, the index $k$ is taken modulo $m$. As a consequence,
\[
\rho(\kappa)\hat{B}_{j,l}=\sum_{k=1}^{m}e^{ilk\zeta }B_{j,k+1}=e^{-il\zeta
}\sum_{k=1}^{m}e^{il(k+1)\zeta }B_{j,k+1}=e^{-il\zeta }\hat{B}_{j,l}~,
\]%
and similarly $\rho(\kappa)\hat{C}_{j,l}=e^{-il\zeta }\hat{C}_{j,l}$, as stated. The other identities in 
\eqref{eq:Bjk-Cjk-as-eigenvectors} follow from 
\begin{equation*}
e^{J_{3}\zeta } \,  e_1=\cos \zeta  \, e_1+ \sin \zeta \,  e_2, \qquad 
e^{J_{3}\zeta }\,  e_2=-\sin \zeta \,  e_1+ \cos \zeta \, e_2,
\end{equation*}
and the fact that the permutation $\tau_m$ fixes the last entries where the poles appear according 
to our convention (see 
the definition of $\tau_m$ in  \eqref{eq:def-tau}).

Equations  \eqref{eq:Bjk-Cjk-as-eigenvectors} show that the subspaces $V_l$ are subrepresentations 
of $K_m$. In fact,  $V_a=\bigoplus_{l=0}^{\left[ \frac{m}{2}\right]}V_l$ is the $K_m$-isotypic decomposition 
of $V_a$.  The block diagonalization of $d^2H_\omega$ with respect
to this decomposition is then a consequence of the well-known Schur's lemma.
\end{proof}

\subsection{Construction of a  symmetry-adapted basis of a symplectic slice $\mathcal{N}_a$}
\label{ss:ConstructionSymmetricSympSlice}

We will  now  construct a symplectic slice $\mathcal{N}_a$  (i.e. a vector subspace of $V_a$ satisfying \eqref{eq:symp-slice-condition}) by specifying its components with respect to the
isotypic decomposition of $V_a$ given in 
Lemma \ref{lemma:isotypic-dec}. For this matter we first investigate the position of the subspaces $\ker d\Phi(a)$ and $\so(3)_\mu\cdot a$ with respect
to the isotypic decomposition of $V_a$. This is respectively done in subsections  \ref{ss:kerdPhi-isotypic} and \ref{ss:group-orbit-isotypic}. Using this information, we 
proceed to define the sought symplectic slice $\mathcal{N}_a$ in subsection \ref{ss:def-symplectic-slice}.

\subsubsection{Position of  $\ker (d\Phi(a))$ relative to the isotypic decomposition of $V_a$ of Lemma  \ref{lemma:isotypic-dec} }
\label{ss:kerdPhi-isotypic}
%

%
%

Considering that  $\Phi:M\to \R^3$ is given by $\Phi(v)=v_1+\dots + v_N$, its derivative
at $a$ is given by
\begin{equation}
\label{eq:dphi(a)}
d\Phi(a): V_a \to \R^3, \qquad d\Phi(a)( {\bf w})= w_1+\dots + w_N,
\end{equation}
where $ {\bf w}=(w_1,\dots, w_N) \in V_a$. Therefore,
\begin{equation}
\label{eq:kerdPhi}
\ker d\Phi (a) =\left  \{  {\bf w}=(w_1,\dots, w_N) \in V_a \, :\, w_1+\dots + w_N=0 \right \}.
\end{equation}
If $N\geq 3$ then   $ d\Phi (a)$ is onto and hence $\ker d\Phi (a) $ is a subspace of $V_a$ of codimension 3. 
The next proposition indicates the  position of $\ker d\Phi (a) $ relative to the isotypic decomposition of $V_a$. In its statement, and in what 
follows, we find it convenient to  denote
\begin{equation}
\label{eq:def-etaj}
\eta_j:=x_j-i y_j \in \C, \qquad j=1,\dots, n,
\end{equation}
where we recall that $u_j=(x_j,y_j,z_j)$ is the generator  of the $j^{th}$ ring
(Eq \eqref{eq:ring-generators}).

\begin{proposition}
\label{prop:kerdPhi-isotypic} Consider the decomposition  $V_a=\bigoplus_{l=0}^{\left[ \frac{m}{2}\right]}V_l$  established in Lemma \ref{lemma:isotypic-dec}. The following statements hold.
\begin{enumerate}
\item[(i)]  $V_l\subset \ker d\Phi(a)$ for all $l\geq 2$.
\item[(ii)] $V_0 \cap \ker d\Phi(a)$ is a codimension 1 subspace of $V_0$ (i.e. $\dim (V_0 \cap \ker d\Phi(a))= \dim V_0 -1$) and a basis for $V_0 \cap \ker d\Phi(a)$ is given by
\begin{equation}
\label{eq:0basis-aux}
\tilde {\tilde{ \mathcal{B}_0}}:=
\left \{ \hat B_{j,0} \right \}_{j=1}^n \, \bigcup  \,\left \{  | \eta_1 |^2 \hat C_{j,0} -  | \eta_j |^2 \hat C_{1,0}  \right \}_{j=2}^n.
\end{equation}
\item[(iii)]$V_1 \cap \ker d\Phi(a)$ is a codimension 2 subspace of $V_1$ (i.e. $ \dim (V_1 \cap \ker d\Phi(a))= \dim V_1 -2$) and a basis of  $V_1 \cap \ker d\Phi(a)$ is given by
$\mathcal{B}_1:=\mathcal{B}^{(1)}_1\cup \mathcal{B}^{(2)}_1\cup \mathcal{B}^{(3)}_1\cup \mathcal{B}^{(4)}_1$ where
  \begin{enumerate}
\item[$\bullet$] If $m\geq 3$ we define  
\begin{equation}
\label{eq:basisU1}
\begin{split}
&\mathcal{B}^{(1)}_1:= \left \{ \, \mbox{\em Re}(z_1  \hat B_{1,1} + i \, \hat C_{1,1}) , \mbox{\em Im}( z_1  \hat B_{1,1} + i \, \hat C_{1,1})\, \right \} , \\
&\mathcal{B}^{(2)}_1:= \left \{ \, \mbox{\em Re} ( \eta_j \hat B_{1,1} -\eta_1 \hat B_{j,1})\, , \, \mbox{\em Im} ( \eta_j \hat B_{1,1} -\eta_1 \hat B_{j,1}) \,  \right \}_{j=2}^n,\\
 &\mathcal{B}^{(3)}_1:= \left \{ \,   \mbox{\em Re} (z_j\eta_j  \hat B_{1,1} + i \eta_1 \hat C_{j,1})  \, , \, \mbox{\em Im} (z_j\eta_j  \hat B_{1,1} + i \eta_1 \hat C_{j,1})  \, \right \}_{j=2}^n,\\
&\mathcal{B}^{(4)}_1:=\left \{ \mbox{\em Re} \left  (2 \hat B_{1,1} + im\eta_1(\delta x_s+i \delta y_s) \right ) \, , \, \mbox{\em Im} \left  (2 \hat B_{1,1} + im\eta_1(\delta x_s+i \delta y_s) \right )\,
\right \}_{s=1}^p.
\end{split}
\end{equation}
\item[$\bullet$] If $m=2$ we take instead 
\begin{equation}
\label{eq:basisU1meq2}
\begin{split}
&\mathcal{B}^{(1)}_1:= \left \{ z_1 \mbox{\em Re} (\eta_1 \bar{\eta_j})\hat B_{1,1} -z_1 |\eta_1|^2\hat B_{j,1}- \mbox{\em Im} (\eta_1 \bar{\eta_j})\hat C_{1,1}   \right \}_{j=2}^n,  \\
&\mathcal{B}^{(2)}_1:=   \left \{ z_1z_j \mbox{\em Im} (\eta_1 \bar{\eta_j})\hat B_{1,1} -z_1 |\eta_1|^2\hat C_{j,1}+z_j \mbox{\em Re} (\eta_1 \bar{\eta_j})\hat C_{1,1}   \right \}_{j=2}^n, \\
&\mathcal{B}^{(3)}_1:=\left \{z_1 \mbox{\em Im} (\eta_1)\hat B_{1,1} + \mbox{\em Re} (\eta_1)\hat C_{1,1}-2z_1 |\eta_1|^2 \delta x_s   \,\right \}_{s=1}^p,\\
&\mathcal{B}^{(4)}_1:=\left \{ \,  z_1 \mbox{\em Re} (\eta_1)\hat B_{1,1} - \mbox{\em Im} (\eta_1)\hat C_{1,1}-2z_1 |\eta_1|^2 \delta y_s  \, \right \}_{s=1}^p.
\end{split}
\end{equation}
\end{enumerate}
\end{enumerate}
%
 \end{proposition}

The proof of Proposition \ref{prop:kerdPhi-isotypic} requires knowledge of the value of 
  $d\Phi(a)$ acting on the basis vectors of 
the subspaces $V_l$ of Lemma \ref{lemma:isotypic-dec}. This information is contained in the following proposition.

\begin{proposition}
The following identities hold 
\begin{equation}
\label{eq:mom-map-formulas}
\begin{split}
&d\Phi(a) \hat B_{j,l}=0, \qquad l=0,2,\dots \left [ \frac{m}{2} \right ],  \quad j=1,\dots, n, \\
&d\Phi(a) \hat B_{j,1}=\begin{cases} 2 (-y_j,x_j,0) \qquad &m= 2, \quad j=1,\dots, n,  \\ -i \frac{m}{2}\eta_j \left ( 1,i ,0  \right ), \qquad &m\geq 3, \quad j=1,\dots, n, 
\end{cases} \\
&d\Phi(a) \hat C_{j,0}=-m | \eta_j |^2 e_3, \qquad  j=1,\dots, n, \\
&d\Phi(a) \hat C_{j,1}=\begin{cases} 2z_j(x_j,y_j,0),   \qquad  &m= 2, \quad j=1,\dots, n, \\
  \frac{m}{2}z_j \eta_j (1,i,0), \qquad &m\geq 3, \quad  j=1,\dots, n, \end{cases} \\
&d\Phi(a) \hat C_{j,l}= 0, \qquad  j=1,\dots, n, \quad l=2, \dots, \left [ \frac{m}{2} \right ], \\
&d\Phi(a)  (\delta x_s+i \delta y_s)=(1,i,0).
\end{split}
\end{equation}
\end{proposition}
\begin{proof}
We will make use of the following formula which can be verified using standard trigonometric identities
\begin{equation}
\label{eq:LemmaMommapaux1}
e^{ikl\zeta} e^{kJ_3 \zeta}=\frac{1}{2} \left ( A e^{k(l-1)J_3 \zeta} + \overline{A} e^{k(l+1)J_3 \zeta} 
\right ) + e^{ikl\zeta} e_3 e_3^T, \qquad A=\left( 
\begin{array}{ccc}
1 & -i & 0 \\ 
i & 1 & 0 \\ 
0 & 0 & 0%
\end{array}%
\right).
\end{equation}
 We will also use the following identities which follow from 
geometric series calculations
\begin{equation}
\label{eq:LemmaMommapaux2}
\begin{small}
\sum_{k=1}^m e^{k(l-1)J_3 \zeta} =\begin{cases} m \I_3, \quad &l=1,  \\
m e_3 e_3^T, \quad &l= 0,2, \dots,\left [ \frac{m}{2} \right ]   \end{cases}, \quad 
\sum_{k=1}^m e^{k(l+1)J_3 \zeta} =\begin{cases}2 \I_3, \quad &m=2, \, l=1, \\
 m e_3 e_3^T, \quad &\begin{array}{l} \mbox{in any other case}, \\ 
 0\leq l \leq \left [ \frac{m}{2} \right ], \end{array} 
 \end{cases}
 \end{small}
\end{equation}
where $\I_3$ denotes the $3\times 3$ identity matrix.

We begin by noticing that 
\begin{equation*}
d\Phi(a)B_{j,k}=b_{j,k}=J_3a_{j,k}=J_3 e^{k J_3 \zeta}u_j.
\end{equation*}
Therefore, using \eqref{eq:LemmaMommapaux1} we find
\begin{equation*}
\begin{split}
d\Phi(a)\hat B_{j,l}&=\sum_{k=1}^m e^{i k l} J_3 e^{k J_3 \zeta}u_j  \\
&=\frac{1}{2}J_3 A\left ( \sum_{k=1}^m e^{k(l-1)J_3 \zeta} \right )u_j + 
\frac{1}{2}J_3 \overline A\left ( \sum_{k=1}^m e^{k(l+1)J_3 \zeta} \right ) u_j+\left (\sum_{k=1}^me^{ikl \zeta} \right )J_3 e_3e_3^T u_j.
\end{split}
\end{equation*}
Using $J_3e_3=0$ it is seen that the last term on the right vanishes. On the other hand, for $l\neq 1$ the other 
two terms also vanish in view of \eqref{eq:LemmaMommapaux2} and since $Ae_3=\overline{A}e_3=0$. Finally, for 
$l=1$ we must distinguish the cases $m=2$ and $m\geq 3$. In the former case we have in view of \eqref{eq:LemmaMommapaux2},
\begin{equation*}
\begin{split}
d\Phi(a)\hat B_{j,1}=J_3(A+\overline{A})u_j=2 (-y_j,x_j,0).
\end{split}
\end{equation*}
For $m\geq 3$  using again \eqref{eq:LemmaMommapaux2} and
$\overline{A}e_3=0$, we have instead,
\begin{equation*}
\begin{split}
d\Phi(a)\hat B_{j,1}=\frac{m}{2} J_3 A u_j+\frac{m}{2} J_3 \overline{A} e_3e_3^T u_j =
\frac{m}{2}(-i \eta_j, \eta_j,0)=-i\frac{m}{2}\eta_j(1,i,0).
\end{split}
\end{equation*}
The above calculations show that all given formulas for $d\Phi(a)\hat B_{j,l}$
in \eqref{eq:mom-map-formulas} indeed hold. In order to prove those for 
$d\Phi(a)\hat C_{j,l}$ we proceed analogously. We first notice that
\begin{equation*}
\begin{split}
d\Phi(a)C_{j,k}&=c_{j,k}=\left (J_3a_{j,k} \right ) \times a_{j,k} =
\left ( J_3 e^{k J_3 \zeta}u_j \right ) \times\left (  e^{k J_3 \zeta}u_j \right )
= e^{k J_3 \zeta}\left ( J_3 u_j \right ) \times u_j =e^{k J_3 \zeta}(z_j u_j -e_3) \\
&= z_j e^{k J_3 \zeta}u_j -e_3.
\end{split}
\end{equation*}
Therefore, 
\begin{equation*}
\begin{split}
d\Phi(a)\hat C_{j,l}&=z_j \sum_{k=1}^me^{i kl\zeta} e^{k J_3 \zeta}u_j -\sum_{k=1}^me^{i kl\zeta}e_3.
\end{split}
\end{equation*}
Using  \eqref{eq:LemmaMommapaux1} we obtain
\begin{equation*}
 d\Phi(a)\hat C_{j,l}=
 \frac{z_j}{2} A\left ( \sum_{k=1}^m e^{k(l-1)J_3 \zeta} \right )u_j + 
\frac{z_j}{2} \overline A\left ( \sum_{k=1}^m e^{k(l+1)J_3 \zeta} \right ) u_j+(z_j^2-1)\left (\sum_{k=1}^me^{i kl\zeta}\right )e_3.
\end{equation*}
The last term on the right hand side equals 
$m(z_j^2-1)e_3=-m|\eta_j|^2e_3$ if $l=0$, and it instead 
vanishes if
$1\leq l \leq \left [ \frac{m}{2} \right ]$. Indeed, this follows from
\begin{equation*}
\sum_{k=1}^me^{i kl\zeta}=\begin{cases} m &\mbox{if $l=0$}, \\ 
0 &\mbox{if $1\leq l \leq \left [ \frac{m}{2} \right ]$} \end{cases}.
\end{equation*}
Thus, in view of  \eqref{eq:LemmaMommapaux2}, and since $Ae_3=\overline{A}e_3=0$,
we conclude that the given formulas
for $d\Phi(a)\hat C_{j,l}$ indeed hold for $l=0,2,\dots, \left [ \frac{m}{2} \right ]$. We now
treat the case $l=1$. Using again  \eqref{eq:LemmaMommapaux2} and 
$\overline{A}e_3=0$, we find that 
for $m\geq 3$ we have
\begin{equation*}
 d\Phi(a)\hat C_{j,1}=
 \frac{m}{2}z_j A u_j = \frac{m}{2}z_j\eta_j\left (1,i,0 \right ).
\end{equation*}
On the other hand, if $m=2$ we obtain, again in view of  \eqref{eq:LemmaMommapaux2},
\begin{equation*}
 d\Phi(a)\hat C_{j,1}=
z_j ( A + \overline{A}) u_j= 2z_j\left (x_j,y_j,0 \right ).
\end{equation*}
Finally, it is obvious from their definition, that $d\Phi(a)\delta x_s=e_1$ and $d\Phi(a)\delta y_s=e_2$ which immediately yields the last identity in \eqref{eq:mom-map-formulas}.
\end{proof}

Having established the validity of   \eqref{eq:mom-map-formulas}, we now give a proof of Proposition \ref{prop:kerdPhi-isotypic}.

\begin{proof}[Proof of Proposition \ref{prop:kerdPhi-isotypic}]
(i) Using the bases of  $V_l$ given in item (ii) of Lemma \ref{lemma:isotypic-dec}, it follows immediately from  \eqref{eq:mom-map-formulas} that $V_l\subset \ker d\Phi(a)$ for $l\geq 2$. 

(ii) It is clear that $V_0$ is not contained  in $\ker d\Phi(a)$ since $d\Phi(a)\hat C_{1,0}\neq 0$ and therefore $\dim V_0\cap \ker d\Phi(a) \leq \dim V_0 -1$. 
On the other hand, using  \eqref{eq:mom-map-formulas}, it is easily verified 
that the set $\tilde {\tilde {\mathcal{B}_0}}$ is  contained  in $V_0\cap \ker d\Phi(a)$.
But, given that $\tilde {\mathcal{B}_0}$ as given in Lemma  \ref{lemma:isotypic-dec} is a basis of $V_0$, it is easily
seen that $\tilde {\tilde {\mathcal{B}_0}}$ is linearly independent.  Considering that it   has $2n-1=\dim V_0-1$ elements, it must
be a basis of $V_0\cap \ker d\Phi(a)$ and hence the dimension of this space is $\dim V_0-1$ as asserted. 

(iii) Using that $\ker d\Phi(a)$ is a codimension 3 subspace of $V_a$ and
$V_a=\bigoplus_{l=0}^{\left [ \frac{m}{2} \right ]} V_l$, a dimension count,  which takes into account items (i) and (ii), implies that  $V_1\cap \ker d\Phi(a)$ is a codimension 2 subspace of $V_1$.
 In view of
 item (ii) of Lemma \ref{lemma:isotypic-dec}, the elements of $\mathcal{B}_1=\mathcal{B}_1^{(1)}\cup \mathcal{B}_1^{(2)}\cup \mathcal{B}_1^{(3)}\cup \mathcal{B}_1^{(4)}$ 
 (given by either \eqref{eq:basisU1} or \eqref{eq:basisU1meq2}) are linearly independent. Moreover, using again  \eqref{eq:mom-map-formulas} one checks that
 they  are contained in $V_1\cap \ker d\Phi(a)$. To finish the proof  that $\mathcal{B}_1$ is a basis it suffices to do a count of its elements and compare with Table \ref{eq:dimVl}.
 Regardless of the value of $m$, one sees that the cardinality of  $\mathcal{B}_1$ equals $\dim V_1-2$.
\end{proof}

\subsubsection{Position of $\so(3)_\mu\cdot a$ relative to the isotypic decomposition of $V_a$}
\label{ss:group-orbit-isotypic}

As above, let  $(a,\omega)\in M\times \R$ be a $\mathbb{Z}_m$-symmetric ($m\geq 2$) RE  and denote by $\mu =\Phi(a)\in \R^3$. Recall that $SO(3)_\mu$ is the
   isotropy group of $\mu\in \R^3$ with respect to the standard action of $SO(3)$ on $\R^3$,  that $\so(3)_\mu$ denotes its Lie algebra, and 
   $\so(3)_\mu\cdot a\subset T_aM=V_a$ is the tangent space to the $SO(3)_\mu$-orbit through $a$. The following proposition specifies the position 
   of $\so(3)_\mu\cdot a$ with respect to the isotypic  decomposition of $V_a$ given in  Lemma \ref{lemma:isotypic-dec}.

\begin{proposition}
\label{prop:symmetry-isot-decomp}
If $\mu\neq 0$  then $\so(3)_\mu\cdot a$ is a 1-dimensional subspace of $V_a$ which is contained in $V_0\cap \ker d\Phi(a)$. Moreover, it is  generated by 
\begin{equation}
\label{eq:def-sa}
s_a:=\sum_{j=1}^n\hat B_{j,0}.
\end{equation}
 On the other hand, if $\mu=0$ then  $\so(3)_0\cdot a$ coincides with $\so(3)\cdot a$ which is a 3-dimensional subspace of $V_a$ which satisfies
 \begin{equation*}
\dim \left ( (\so(3) \cdot a)\cap V_0\cap \ker d\Phi(a) \right ) = 1, \qquad \dim  \left ( (\so(3) \cdot a)\cap V_1\cap \ker d\Phi(a) \right )= 2.
\end{equation*}
Furthermore, also in this case,  the intersection $ (\so(3) \cdot a)\cap V_0\cap \ker d\Phi(a)$ is  generated by $s_a$ given by \eqref{eq:def-sa}.
 \end{proposition}
\begin{proof}
  It is easy to see that    $\so(3)_\mu=\{\xi \in \so(3) \, : \, \xi \mu =0 \}\subseteq \so(3)$.  Denoting $a=(a_1,\dots, a_N)$, we have
\begin{equation*}
\so(3)_\mu\cdot a =\{ (\xi a_1,\dots, \xi a_N) \, : \, \xi \in   \so(3)_\mu \} \subset V_a.
\end{equation*}
Now, using \eqref{eq:dphi(a)}, for $\xi \in \so(3)_\mu$ we have
\begin{equation*}
d\Phi(a)( \xi a_1,\dots, \xi a_N)=  \xi a_1+\dots+ \xi a_N= \xi (a_1+\dots + a_N)=\xi \mu =0,
\end{equation*}
which shows that $\so(3)_\mu\cdot a \subset \ker d\Phi(a)$ regardless of the value of $\mu$.

Suppose that $\mu \neq 0$. Then it is clear that $\so(3)_\mu$ is 1-dimensional and, in virtue of  Theorem \ref{th-main-symmetry}(iv), it is generated by $J_3$. In view of the discussion above it follows that 
$\so(3)_\mu\cdot a$ is generated by $(J_3a_1,\dots, J_3a_N)\in V_a$. Tracing back the definitions of the vectors $\hat B_{j,0}$ it is seen that the  vector $(J_3a_1,\dots, J_3a_N)$ is
precisely  $s_a$.
The assertion that $s_a\in V_0$ is immediate since it is a linear combination of vectors in its basis (see Lemma \ref{lemma:isotypic-dec}).

Suppose now that $\mu=0$.  It is clear that $\so(3)_0=\so(3)$ and $\so(3)\cdot a$ is 3-dimensional since the $SO(3)$ action on $M$ is 
free. It is also clear that  in this case $s_a\in \so(3)\cdot a \cap V_0$. Thus, to finish the proof, we only need to show that the intersection $(\so(3)\cdot a) \cap V_1$ is 2-dimensional. For this 
matter,  we use that $\so(3)\cdot a =\{ (\xi a_1,\dots, \xi a_N) \, : \, \xi \in   \so(3) \} \subset V_a$, and take $\xi$ given by $J_1, J_2\in \so(3)$ where 
\begin{equation*}
J_1 :=\begin{pmatrix} 0 & 0 & 0 \\ 0 & 0 & -1 \\ 0 & 1 & 0 \end{pmatrix} \qquad \mbox{and}  \qquad J_2 :=\begin{pmatrix} 0 & 0 & 1 \\ 0 & 0 & 0 \\ -1 & 0 & 0 \end{pmatrix}.
\end{equation*}
This leads to the conclusion that the following two  independent vectors
\begin{equation*}
t_a^{(1)}:=(e_1\times a_1,\dots, e_1\times a_N), \qquad t_a^{(2)}:=(e_2\times a_1,\dots, e_2\times a_N),
\end{equation*}
belong to $\so(3)\cdot a \subset V_a$. We  will now show that both $t_a^{(1)}, t_a^{(2)}\in V_1$. For this recall the definition of the linear map $\rho(\kappa):V_a\to V_a$ 
given by \eqref{eq:rep-fmla}. Using 
 the refinement of the notation of $a$ of \eqref{eq:REfull} and the convention \eqref{eq:ring-convention}, 
a simple calculation (which uses $g_m\left ( (e_1+ie_2)\times a_{j,k} \right )= \left (g_m(e_1+ie_2)\right ) \times \left (g_ma_{j,k}\right )=e^{-i\zeta} (e_1+ie_2)\times a_{j,k+1}$, where 
$k$ is taken modulo $m$) yields
\begin{equation*}
\rho(\kappa)\left (t_a^{(1)}+i t_a^{(2)}\right ) = e^{-i\zeta} \left  (t_a^{(1)}+i t_a^{(2)} \right ).
\end{equation*}
Comparing this with \eqref{eq:Bjk-Cjk-as-eigenvectors}, it is seen that the vector $t_a^{(1)}+i t_a^{(2)}\in V_a+iV_a$ lies in the $e^{-i\zeta}$-eigenspace of $\rho(\kappa)$. Therefore,
$t_a^{(1)}$ and $t_a^{(2)}$ belong to the $V_1$ component in the isotypic decomposition of $V_a$. 
\end{proof}

\subsubsection{Definition of the symplectic slice}
\label{ss:def-symplectic-slice}

 Assume that $\mu\neq 0$  and recall from  subsection \ref{ss:energy-momentum-method} that a symplectic slice is any subspace of $\mathcal{N}_a\subset V_a$ which satisfies
$\ker d\Phi(a)= \mathcal{N}_a\oplus (\so(3)_\mu \cdot a)$. Proposition \ref{prop:symmetry-isot-decomp} implies that, under our assumption
that $\mu\neq 0$, we have   $\so(3)_\mu \cdot a=\langle s_a \rangle$ and hence 
 $\mathcal{N}_a$  should satisfy
\begin{equation}
\label{eq:symp-slice-def}
\ker d\Phi(a)=\mathcal{N}_a\oplus \langle s_a \rangle.
\end{equation}
In particular,  $\mathcal{N}_a$ is  a $2N-4$-dimensional subspace of $V_a$. 

We now claim that a symplectic slice can be defined
as the direct sum 
\begin{equation}
\label{eq:decomposition-slice}
\mathcal{N}_a:= \bigoplus_{l=0}^{\left [ \frac{m}{2} \right ]}U_l,
\end{equation}
with the following choice of subspaces $U_l$ which are contained in $V_l$:
\begin{itemize}
\item $U_l:=V_l$ for all $l\geq 2$,
\item $U_1:=V_1\cap \ker d\Phi(a)$,
\item $U_0$ is any codimension $1$ subspace of $V_0\cap \ker d\Phi(a)$ not containing $s_a$.
\end{itemize}
Indeed, it is immediate to check, using Propositions \ref{prop:kerdPhi-isotypic} and \ref{prop:symmetry-isot-decomp},  that  \eqref{eq:symp-slice-def} indeed holds under the
above conditions. For future reference we collect this information in  the following lemma that makes
a specific choice of $U_0$ by spelling out a basis. Such choice of $U_0$ is arbitrary and
not the most natural from the geometric point of view, but has the benefit of being simple to implement for our CAPs.

 \begin{lemma}
\label{lemma:symp-slice-def}
Suppose that $\mu=\Phi(a)\neq 0$. Let $U_0\subset V_a$ be the subspace with basis
\begin{equation}
\label{eq:basisU0}
\mathcal{B}_0:= \left \{ \hat B_{j,0} \,  ,\,  | \eta_1 |^2 \hat C_{j,0} -  | \eta_j |^2 \hat C_{1,0}  \right \}_{j=2}^n,
\end{equation}
$U_1=V_1\cap \ker d\Phi(a)$ and $U_l:=V_l$ for all $l\geq 2$. Then $\mathcal{N}_a := \bigoplus_{l=0}^{\left [ \frac{m}{2} \right ]}U_l$ is a symplectic slice satisfying $U_l\subset V_l$. In particular, $\dim \mathcal{N}_a=2N-4$.
\end{lemma}
\begin{proof}
In view of the discussion before the statement of the lemma, we only need to show that the
subspace $U_0$ defined above is a codimension $1$ subspace of $V_0\cap \ker d\Phi(a)$ not containing $s_a$.
To see this note that $\mathcal{B}_0$ is obtained by removing the vector $\hat B_{1,0}$ from the 
 basis $\tilde {\tilde{ \mathcal{B}_0}}$ of $V_0\cap \ker d\Phi(a)$ given in \eqref{eq:0basis-aux} and,
 therefore, the vector $s_a$ in \eqref{eq:def-sa} cannot be written as a linear combination of the 
 vectors in $\mathcal{B}_0$.
\end{proof}

The dimensions of the spaces $U_l$ in terms of the parameters $m$, $n$, $p$ are given in Table \ref{eq:dimUl}. 
Recalling that $N=mn+p$, one can verify  that in all cases we indeed
have $\dim \mathcal{N}_a=2N-4$. 

\begin{table}[h]
\begin{center}
\begin{tabular}[c]{|c|c|c|c|c|c|}%
\hline
  &$ \dim U_0 $& $\dim U_1$ &$ \dim U_l, \;  2\leq l < \left [ \frac{m}{2} \right ]$&$\dim U_{\left [ \frac{m}{2} \right ]} $ \\ \hline
$m=2$ & $2n-2$ & $2n+2p-2$ & - & -  \\ \hline
$m=3 $& $2n-2 $&  $4n+2p-2$ & - & -  \\ \hline
$m\geq 4 \quad \mbox{even} $&$ 2n-2$ &  $4n+2p-2$ &  $4n$& $2n$ \\ \hline
$m\geq 5 \quad \mbox{odd}$ & $2n-2$ & $4n+2p-2 $&  $4n$& $4n$  \\ \hline
\end{tabular}
\end{center}
\caption{Dimension of the subspaces in the  decomposition of the symplectic slice
 $\mathcal{N}_a=\bigoplus_{l=0}^{\left[ \frac{m}{2}\right]}U_l$ according
to the values of $m$, $n$ and $p$ (valid in the case $\mu=\Phi(a)\neq 0$).}
\label{eq:dimUl}
\end{table}

The following lemma provides the necessary modifications to the symplectic slice that are needed
when $\mu=0$.
\begin{lemma}
\label{lemma:basis-Bl-mu0}
Suppose that $\mu=\Phi(a)= 0$. Let $U_0, \dots,  U_{\left [ \frac{m}{2} \right ]}$ be 
defined as in Lemma \ref{lemma:symp-slice-def} and let $\tilde U_1\subset U_1$ be any subspace satisfying 
$U_1=\tilde U_1\oplus ((\so(3)\cdot a)\cap U_1)$. Then
\begin{equation*}
\tilde {\mathcal{N}_a}:= U_0 \oplus \tilde U_1 \bigoplus_{l=2}^{\left [ \frac{m}{2} \right ]}U_l,
\end{equation*}
 is a symplectic slice satisfying $\tilde U_1\subset V_1$ and $U_l\subset V_l$ for  all other $l\neq 1$.  In particular, $\dim (\tilde{ \mathcal{N}_a})=2N-6$.
\end{lemma}

The proof of the lemma is again an immediate consequence of Propositions \ref{prop:kerdPhi-isotypic} and \ref{prop:symmetry-isot-decomp} ($\tilde U_1$ is a codimension 2 subspace chosen such that $\tilde {\mathcal{N}_a}\oplus (\so(3)\cdot a)=\ker d\Phi(a)$). 
For further reference we note that the relation between the symplectic slices $\tilde {\mathcal{N}_a}$ and $\mathcal{N}_a$ defined above is
\begin{equation}
\label{eq:meq0slice}
\mathcal{N}_a=\tilde {\mathcal{N}_a}\oplus ((\so(3)\cdot a)\cap U_1).
\end{equation}

\subsection{Block diagonalisation of $\left . d^2H_\omega(a) \right |_{\mathcal{N}_a}$ and construction of the block matrices}
\label{ss:matricesPl}

As a consequence of item (i) of Lemma \ref{lemma:isotypic-dec}
and the definition of the spaces $U_l\subset V_l$,  the symmetric bilinear 
form  $\left . d^2H_\omega(a) \right |_{\mathcal{N}_a}$
block diagonalises with respect to the symmetric decomposition \eqref{eq:decomposition-slice}.

In order to obtain a matrix representation of each of the blocks it is necessary to use Lagrange multipliers as explained
in Appendix \ref{s:geometry}. Namely, given that $a=(a_1,\dots, a_N)\in M$ is a critical point of the augmented Hamiltonian $H_\omega:M\to \R$,
there exist Lagrange multipliers $(c_1,\dots, c_N)\in \R^N$ such that $(a,c)$,
 interpreted as a point 
in $(\R^3)^N\times \R^N$, is a critical point of the function
\begin{equation*}
H^*_\omega:(\R^3)^N\times \R^N\to \R, \qquad H^*_\omega(x_1,\dots, x_N,\lambda_1,\dots, \lambda_N)=H_\omega(x_1,\dots, x_N)+\sum \lambda_jR_j(x).
\end{equation*}
 Here $R_j(x)=\frac{1}{2}(1-\|x_j\|^2)$ and on the right hand side we think of $H_\omega$ as a function whose domain is
 (an open subset of ) $(\R^3)^N$ in the natural way. We have,
 \begin{equation}
\label{eq:Bil-Form-id}
d^2H_\omega(a)({\bf w}_1, {\bf w}_2)={\bf w}_1^T \nabla_x^2 H_{\omega}^{\ast}(a,c) {\bf w}_2,
\qquad \forall {\bf w}_1, {\bf w}_2\in V_a,
\end{equation}
where $\nabla_x^2H^*_\omega(a,c)$ denotes the standard Hessian matrix of the function 
$(\R^3)^N\ni x\mapsto H_\omega^*(x,c)\in \R$ evaluated at $a$, and ${\bf w}_1, {\bf w}_2\in V_a\subset (\R^3)^N$ are column vectors
(see Appendix \ref{s:geometry}). 

Now, in order to give a matrix form for the restriction of $ d^2H_\omega(a) $ to the block $U_l$ we need a basis for $U_l$.
A basis  $\mathcal{B}_0$ for  $U_0$ is given by \eqref{eq:basisU0}. For $U_1$ it is given   by \eqref{eq:basisU1} or
\eqref{eq:basisU1meq2}  (according to whether $m$ differs  or equals  $2$). Finally, considering that 
for $2\leq l \leq \left [ \frac{m}{2}\right ]$ one has $U_l=V_l$, we can 
simply set $\mathcal{B}_l:=\tilde {\mathcal{B}_l}$, where  $\tilde {\mathcal{B}_l}$ is the basis of $V_l$ 
given in the statement of item (ii) of Lemma \ref{lemma:isotypic-dec}.

In view of the above considerations, the block matrix $\mathcal{P}_l$ representing   $\left .  d^2H_\omega(a) \right |_{U_l} $ 
in the $\mathcal{B}_l$ basis is given by 
\begin{equation} \label{eq:mathcalP}
\mathcal{P}_l=P_l^T \nabla ^2_v H^\ast_\omega(a)P_l, \qquad 0\leq l \leq \left [ \frac{m}{2}\right ],
\end{equation}
where  $P_l$ denotes the real matrix of size $3N\times \dim U_l $  whose columns are the
vectors of  $\mathcal{B}_l$,  $0\leq l \leq \left [ \frac{m}{2}\right ]$. Note that $\mathcal{P}_l$ is a square matrix of size
$\dim U_l$, which is specified in Table  \ref{eq:dimUl} according to the values of $m$, $n$, $p$.

From the block diagonalisation of $\left . d^2H_\omega(a) \right |_{\mathcal{N}_a}$, it follows that 
\begin{equation}\label{eq:signature_pl}
\mbox{signature}\left (\left . d^2H_\omega(a) \right |_{\mathcal{N}_a} \right )
=\sum_{l=0}^{\left [ \frac{m}{2} \right ]} \mbox{signature}\left ( \mathcal{P}_l \right ).
\end{equation}
In particular, $\left . d^2H_\omega(a) \right |_{\mathcal{N}_a} $ is positive definite if and only if each block
$ \mathcal{P}_l $ is positive definite for all $l=0,\dots, \left [ \frac{m}{2} \right ]$.

\subsection{Complex structure of some blocks of the diagonalisation  
of $\left . d^2H_\omega(a) \right |_{\mathcal{N}_a}$.}
\label{ss:complex-structure}

A further simplification in the calculation of the signature of 
$\left . d^2H_\omega(a) \right |_{\mathcal{N}_a}$ is gained by noticing that some of the
block matrices $\mathcal{P}_l$ in its diagonalisation have a complex structure. 
As a consequence, the signature of $\mathcal{P}_l$
 is determined by the signature  of a complex Hermitian matrix $\mathcal{Q}_l$ which
 has half the size of $\mathcal{P}_l$.
This simplification is only possible when $m\geq 3$, and for the blocks $\mathcal{P}_l$ with
\begin{equation}
\label{eq:cond-l-complex}
l=\begin{cases} 1,\dots, \left [ \frac{m}{2} \right ], \quad & \mbox{if $m$ is odd}, \\
1,\dots, \left [ \frac{m}{2} \right ]-1, \quad & \mbox{if $m$ is even}.
\end{cases} 
\end{equation}
For the rest of this subsection it is  assumed that $m\geq 3$. 

\subsubsection{Definition of the Hermitian blocks $\mathcal{Q}_l$}
\label{ss:DefQl}

Let $Q_1$ be the  
$3N\times (2n+p-1)$ complex matrix  whose columns are the vectors
\begin{equation}
\label{eq:columnsQ1}
\begin{split}
&z_1  \hat B_{1,1} + i \, \hat C_{1,1} , \\
& \eta_j \hat B_{1,1} -\eta_1 \hat B_{j,1}, \qquad j=2, \dots, n,\\
 &z_j\eta_j  \hat B_{1,1} + i \eta_1 \hat C_{j,1}, \qquad   j=2,\dots, n,\\
&2 \hat B_1^1 + im\eta_1 \left (\delta x_s+i \delta y_s) \right ), \qquad s=1,\dots, p.
\end{split}
\end{equation}
On the other hand, for  $l>1$ satisfying \eqref{eq:cond-l-complex} we define $Q_l$ as the $3N\times 2n$ 
complex matrix whose columns are the vectors
\begin{equation}
\label{eq:columnsQl}
\begin{split}
 \hat B_{j,l}, \;   \hat C_{j,l} \qquad j=1, \dots, n.
 \end{split}
\end{equation}

For all $l$ satisfying  \eqref{eq:cond-l-complex}  we define the Hermitian matrix
\begin{equation} \label{eq:mathcalQ}
\mathcal{Q}_l:=\bar{Q_l}^T\nabla ^2_v H^\ast_\omega(a)Q_l,
\end{equation}
 Note that  $\mathcal{Q}_1$ is $  (2n+p-1)\times  (2n+p-1)$ whereas
 $\mathcal{Q}_l$ is $2n\times 2n$ for all other values of $l$. Comparing with Table \ref{eq:dimUl},
 it is seen that  $\mathcal{Q}_l$ has indeed half the dimension of  $\mathcal{P}_l$.

\subsubsection{The signature of  $\mathcal{P}_l$ is twice the signature of $\mathcal{Q}_l$ }

The relation between the signature of the complex Hermitian matrices $\mathcal{Q}_l$
and the real symmetric matrices $\mathcal{P}_l$ is given by the following.
\begin{theorem} \label{th:complexblocks}
Let $l$ as in  \eqref{eq:cond-l-complex} and consider 
the block matrices $\mathcal{P}_l$ and $\mathcal{Q}_l$  defined above. Then
\begin{equation*}
i_+(\mathcal{P}_l)=2i_+(\mathcal{Q}_l), \qquad i_-(\mathcal{P}_l)=2i_-(\mathcal{Q}_l),
 \qquad \dim \ker  \mathcal{P}_l=2\dim \ker \mathcal{Q}_l,
\end{equation*}
where $i_+$ and $i_-$ denote the positive and negative indices of inertia of the corresponding
matrices. In particular,  $\mbox{signature}(\mathcal{P}_l)=2\mbox{signature}(\mathcal{Q}_l)$, and $\mathcal{P}_l$ is positive definite
if and only if $\mathcal{Q}_l$ is positive definite.
\end{theorem}

The  proof of Theorem \ref{th:complexblocks} relies on a general linear 
algebra result,  which we state as Lemma \ref{eq:lemma-complexgeneral} below and whose proof is given in Appendix \ref{app:proof-complex-lemma}.

\begin{lemma}
\label{eq:lemma-complexgeneral}
Let $V$ be a real vector space of dimension $2d$  and $f:V\times V\to \R$ a symmetric bilinear
form. Suppose that there exists a basis $\{\alpha_1,\beta_1,\dots, \alpha_d,\beta_d\}$ of 
$V$ satisfying
\begin{equation}
\label{eq:complex-structure-ids}
f(\alpha_j,\alpha_k)=f(\beta_j,\beta_k), \qquad f(\alpha_j,\beta_k)=-f(\alpha_k,\beta_j), \qquad \forall j,k=1,\dots, d.
\end{equation}
Let $\mathcal{Q}$ be the $d\times d$ complex matrix with entries
\begin{equation*}
\mathcal{Q}_{jk}=f(\alpha_j,\alpha_k)+f(\beta_j,\beta_k)+i \left (f(\alpha_j,\beta_k)-f(\beta_j,\alpha_k) \right ),  \qquad \forall j,k=1,\dots, d.
\end{equation*}
Then $\mathcal{Q}$ is Hermitian and 
\begin{equation*}
i_+(f)=2i_+(\mathcal{Q}), \qquad i_-(f)=2i_-(\mathcal{Q}), \qquad \dim \ker  f=2\dim \ker \mathcal{Q},
\end{equation*}
where $i_+$ and $i_-$ denote the positive and negative indices of inertia of the corresponding
form/matrix. In particular,  $\mbox{signature}(f)=2\mbox{signature}(\mathcal{Q})$, and $f$ is positive definite
if and only if $\mathcal{Q}$ is positive definite.
\end{lemma}

%
%

\begin{remark}
\label{Rmk:MatQ}
For our purposes it is convenient to notice that the expression for $\mathcal{Q}_{jk}$ in the statement
of the lemma coincides with the formal expansion of $f(\alpha_j+i\beta_j,\alpha_k+i\beta_k)$ assuming that $f$ 
is a sesquilinear form with the convention that is conjugate-linear in the  first component and linear in the
second.
\end{remark}

The applicability of Lemma \ref{eq:lemma-complexgeneral} to our analysis relies
on the following observation.
\begin{proposition}
\label{prop:complex-block}
 Fix $l$ satisfying   \eqref{eq:cond-l-complex} and denote 
 \begin{equation*}
 \begin{split}
&\alpha_j:=\mbox{\em Re}(\hat B_{j,l}), \qquad \beta_j:=\mbox{\em Im}(\hat B_{j,l}), \qquad j=1,\dots, n, \\
&\alpha_{n+j}:=\mbox{\em Re}(\hat C_{j,l}), \qquad \beta_{n+j}:=\mbox{\em Im}(\hat C_{j,l}), \qquad j=1,\dots, n.
\end{split}
\end{equation*}
Then we have
\begin{equation}
\label{eq:inv-prop-d2H}
\begin{split}
&d^2H_\omega (a) \left  (\alpha_j, \alpha_k \right )=
d^2H_\omega  (a)\left  ( \beta_{j} ,\beta_{k} \right ), \\
&d^2H_\omega (a) \left  ( \alpha_j ,\beta_k\right )=
-d^2H_\omega  (a)\left (\beta_j ,\alpha_k \right ), \qquad \forall j,k=1,\dots, 2n.
\end{split} 
\end{equation}
Furthermore,  for $l=1$ denote
\begin{equation*}
\alpha_{2n+s}:=\delta x_s, \qquad \beta_{2n+s}:=\delta y_s, \qquad s=1,\dots, p,
\end{equation*}
then \eqref{eq:inv-prop-d2H} holds on the extended index range  $j,k=1,\dots, 2n+p$. 
\end{proposition}
\begin{proof}
The proof follows by combining the $K_m$-invariance  of $d^2H_\omega (a)$ (see  \eqref{eq:inv-d2H}) with the 
relations  \eqref{eq:Bjk-Cjk-as-eigenvectors} which   imply
\begin{equation*}
\rho(\kappa)\alpha_j =\cos l\zeta \, \alpha_{j} + \sin l\zeta \, \beta_{j}, \qquad 
 \rho(\kappa)\beta_j=- \sin l\zeta \, \alpha_{j}+\cos l\zeta \, \beta_{j}, 
 \end{equation*}
for all $j=1,\dots,2n$ if $l\neq 1$ and all  $j=1,\dots,2n+p$ if $l=1$.
Therefore, abbreviating $f:=d^2H_\omega(a)$, and using its $K_m$-invariance, we have
\begin{equation*}
f(\alpha_j,\alpha_k)=f(\rho(\kappa)\alpha_j,\rho(\kappa)\alpha_k)=f(\cos l\zeta \, \alpha_{j} + \sin l\zeta \, \beta_{j} \, , \,
\cos l\zeta \, \alpha_{k} + \sin l\zeta \, \beta_{k}),
\end{equation*}
which by bilinearity of $f$ implies
\begin{equation}
\label{eq:aux-inv-prop-1}
\sin^2 l\zeta f(\alpha_{j},\alpha_{k})=\sin^2 l\zeta f(\beta_{j},\beta_{k})
+\sin  l\zeta \cos l\zeta \left (  f(\beta_{j},\alpha_{k})+f(\alpha_{j},\beta_{k})\right ).
\end{equation}
On the other hand, again by $K_m$-invariance  of $d^2H_\omega (a)$, we have
\begin{equation*}
f(\alpha_{j},\beta_{k})=f(\rho(\kappa)\alpha_{j},\rho(\kappa)\beta_{k})=f(\cos l\zeta \, \alpha_{j} + \sin l\zeta \, \beta_{j},
- \sin l\zeta \, \alpha_{k}+\cos l\zeta \, \beta_{k}),
\end{equation*}
which  by bilinearity of $f$ implies
\begin{equation}
\label{eq:aux-inv-prop-2}
\sin^2 l\zeta f(\alpha_{j},\beta_{k})=-\sin^2 l\zeta f(\beta_j,\alpha_{k})
+\sin  l\zeta \cos l\zeta \left (  f(\beta_{j},\beta_{k})-f(\alpha_j,\alpha_k)\right ).
\end{equation}
 
 Now recalling that we are working under the assumption that $m\geq3$,  $l$ satisfies \eqref{eq:cond-l-complex} and  that 
$\zeta=\frac{2\pi}{m}$. These conditions imply that 
that $\sin l\zeta \neq 0$ which, via a simple manipulation of \eqref{eq:aux-inv-prop-1} and \eqref{eq:aux-inv-prop-2}
proves that 
\begin{equation*}
 f(\alpha_j,\alpha_k)= f(\beta_j,\beta_k), \qquad
  f(\alpha_j,\beta_k)= -f(\alpha_k,\beta_{j}),
\end{equation*}
as required. \end{proof}

\begin{proof}[Proof of Theorem \ref{th:complexblocks}]
For $l>1$ the proof follows at once from Lemma \ref{eq:lemma-complexgeneral} and 
Proposition \ref{prop:complex-block}. Indeed, the hypothesis of the lemma are verified in 
view of \eqref{eq:inv-prop-d2H} and it is easy to check that the Hermitian matrix $\mathcal{Q}$ 
in the lemma is precisely the matrix $\mathcal{Q}_l$ defined in subsection \ref{ss:DefQl} above (this is a consequence of 
\eqref{eq:Bil-Form-id} and  Remark \ref{Rmk:MatQ}) .

The conclusion for $l=1$ follows by complementing Lemma \ref{eq:lemma-complexgeneral} and 
Proposition \ref{prop:complex-block} with the 
following observation. Suppose that  $\{\alpha_1,\beta_1,\dots, \alpha_d,\beta_d\}$ are vectors in $\R^{D}$ 
satisfying \eqref{eq:complex-structure-ids} for a certain real bilinear form $f:\R^D\times \R^D\to \R$. 
Let $z_{jk}\in \C$ with $j,k\in \{1,\dots, d\}$. Then the vectors $\tilde \alpha_j$, $\tilde \beta_j$ defined by
\begin{equation*}
\tilde \alpha_j := \mbox{Re}\left ( \sum_k z_{jk} (\alpha_k+i\beta_k) \right ), \qquad 
\tilde \beta_j := \mbox{Im}\left ( \sum_k z_{jk} (\alpha_k+i\beta_k) \right ),
\end{equation*}
also satisfy \eqref{eq:complex-structure-ids}. As a consequence of this observation and Proposition 
 \ref{prop:complex-block}, it follows that 
the real and imaginary parts of the vectors \eqref{eq:columnsQ1}, which form a basis of $U_1$,
also satisfy the hypothesis \eqref{eq:complex-structure-ids} of Lemma  \ref{eq:lemma-complexgeneral}.
Again, as a  consequence of 
\eqref{eq:Bil-Form-id} and  Remark \ref{Rmk:MatQ}, 
it is straightforward to check that the Hermitian matrix $\mathcal{Q}$ 
in the lemma coincides with $\mathcal{Q}_1$ as defined in subsection \ref{ss:DefQl} above.
\end{proof}

\subsection{Summary of the stability analysis}
\label{ss:Summary}
We present here a summary of the stability test which is useful for implementation.

Let $(a,\omega)\in M\times \R$ be a $\Z_m$-symmetric RE and assume that  $m\geq 2$. We first recall the block matrices which constitute our block 
diagonalisation of  $\left . d^2H_\omega(a) \right |_{\mathcal{N}_a}$. 
The  generic block will be denoted by $\mathcal{M}_l$ where  the index $l$ runs from $0$ to $\left [ \frac{m}{2} \right ]$ (so there are  
 $1+\left [ \frac{m}{2} \right ]$ blocks). According to the notation of the previous sections, we shall write $\mathcal{M}_l=\mathcal{P}_l$ if the block
is  real symmetric,  and  instead  $\mathcal{M}_l=\mathcal{Q}_l$ if the block is  complex Hermitian. The Table \ref{eq:QlPl-summary} 
below specifies whether
  $\mathcal{M}_l$ equals $\mathcal{P}_l$ or $\mathcal{Q}_l$ according to the number $m$ of vortices in each ring for each value of $l$. The table also indicates (in parenthesis below each matrix) the size of $\mathcal{M}_l$
 in terms of the number $n$ of rings and the number $p$ of poles (recall that $N=mn+p$). Moreover, according to   \eqref{eq:mathcalP} and \eqref{eq:mathcalQ}, 
 each such matrix $\mathcal{M}_l$ is of the form 
 \begin{equation}
 \label{def:M_l}
\mathcal{M}_l=\bar M_l^T\nabla ^2_v H^\ast_\omega(a)M_l,
\end{equation}
 where the columns of the matrix $M_l$, which corresponds to either $P_l$ or $Q_l$ in the text above, 
 are determined by the corresponding   equation at the bottom of each entry.  In all cases, 
 the columns of $M_l$  are given in terms of
 the vectors $\hat B_{j,l}$, $ \hat C_{j,l}$ and the  scalars $z_j \in \R$, $\eta_j\in \C$ defined by  \eqref{eq:hatBhatC}, \eqref{eq:ring-generators}  and \eqref{eq:def-etaj},
  which in turn are determined  in terms of the RE configuration $a$ written with the convention \eqref{eq:REfull}. 
 
 On the other hand, if the RE $(a,\omega)\in M\times \R$ has no symmetries, i.e. $m=1$,   there is no block decomposition.
In this case we denote by $\mathcal{M}_0=\mathcal{P}_0$ the
 matrix representation of $\left . d^2H_\omega(a) \right |_{\mathcal{N}_a}$  given by \eqref{def:M_l}
with $l=0$, where  the $2N-4$ columns of the  matrix $M_0$ are given by \eqref{eq:asymmetricSlice},
which are also determined in terms of the RE configuration $a$ written with the convention \eqref{eq:REfull}. 
For completeness, this information is also included in Table \ref{eq:QlPl-summary}.

\begin{table}[h]
\begin{center}
\begin{tabular}[c]{|c|c|c|c|c|}%
\hline
& $l=0$ & $l=1$ & $ l \geq 2$ \\
\hline
$m=1 $  &
$ \begin{array}{c}
\mathcal{P}_0 \\
\mbox{($2N-4$}) \\
\eqref{eq:asymmetricSlice}
\end{array}$ & 
- & -   \\ 
\hline
$m=2$   &$
 \begin{array}{c}
\mathcal{P}_0 \\
\mbox{($2n-2$}) \\
\eqref{eq:basisU0}
\end{array} $& $
\begin{array}{c}
\mathcal{P}_1 \\
\mbox{($2n+2p-2$}) \\
\eqref{eq:basisU1meq2}
\end{array} $& -   \\ \hline
$m=3 $  &$ \begin{array}{c}
\mathcal{P}_0 \\
\mbox{($2n-2$})\\
\eqref{eq:basisU0}
\end{array} $&  $ \begin{array}{c}
\mathcal{Q}_1 \\
\mbox{($2n+p-1$})\\
\eqref{eq:columnsQ1}
\end{array}$ & -   \\ \hline
$m\geq 4 \quad \mbox{even} $
 &$ \begin{array}{c}
\mathcal{P}_0 \\
\mbox{($2n-2$})\\
\eqref{eq:basisU0}
\end{array}$ &  $\begin{array}{c}
\mathcal{Q}_1 \\
\mbox{($2n+p-1$})\\
\eqref{eq:columnsQ1}
\end{array} $&  
$\begin{array}[c]{c|c}
 \begin{array}{l}
\mathcal{Q}_l \quad 2\leq l <\frac{m}{2} \\
\mbox{($2n$})\\
\eqref{eq:columnsQl}
\end{array}  & \begin{array}{c}
\mathcal{P}_{\frac{m}{2}} \quad  l =\frac{m}{2} \\
\mbox{($2n$})\\
\eqref{eq:columnsQl}^*
\end{array}
\end{array}$
  \\ \hline
$m\geq 5 $\quad \mbox{odd}  
 & $\begin{array}{c}
\mathcal{P}_0 \\
\mbox{($2n-2$})\\
\eqref{eq:basisU0}
\end{array} $&$  \begin{array}{c}
\mathcal{Q}_1 \\
\mbox{($2n+p-1$})\\
\eqref{eq:columnsQ1}
\end{array} $&  $ \begin{array}{l}
\mathcal{Q}_l \quad 2\leq l \leq \frac{m-1}{2} \\
\mbox{($2n$})\\
\eqref{eq:columnsQl}
\end{array}  $  \\ \hline
\end{tabular}\\
\end{center}
\vspace{0.2cm}
\begin{small} $^*$For $m\geq 4$, even, the seemingly complex columns of $M_{\frac{m}{2}}$ 
given by \eqref{eq:columnsQl} are in fact real (see   
Lemma \ref{lemma:isotypic-dec}). \end{small}
\caption{Summary of the block matrices that need to be computed for the stability analysis. See text for details.}
\label{eq:QlPl-summary}
\end{table}

Once  the block matrices $\mathcal{M}_l$, $l=0,\dots, \left [ \frac{m}{2} \right ]$, have been computed, our stability test proceeds as follows:
\begin{itemize}
\item[$\bullet$] If $\mu=\Phi(a)\neq 0$ (which implies  $\omega\neq 0$ by Proposition \ref{prop:zeroMomentum}), 
the RE is orbitally stable if  $\left . d^2H_\omega(a) \right |_{\mathcal{N}_a}$ is positive definite  and the stability test
proceeds by checking positive definiteness of each block $\mathcal{M}_l$.
\item[$\bullet$] If $\omega=0$ then $\mu=\Phi(a)= 0$ (by Proposition \ref{prop:zeroMomentum}) and
 the
RE is  $SO(3)$-stable if  $\left . d^2H_0(a) \right |_{\tilde {\mathcal{N}_a}}$ is positive definite (note the change 
 from $\mathcal{N}_a$ to  $\tilde{\mathcal{N}_a}$).  In this case, owing  to the
$SO(3)$-invariance of $H_0=H$ we know that $d^2H_0(a) $ vanishes along $\so(3)\cdot a$ which gives rise to a 
 2-dimensional null space of  $\left . d^2H_0(a) \right |_{ {\mathcal{N}_a}}$. If $m\geq 2$, by 
 Lemma \ref{lemma:basis-Bl-mu0} and, in particular, Eq. \eqref{eq:meq0slice},  
 this   2-dimensional null-space belongs to $U_1$ so  the  block $\mathcal{M}_1$ will have a multiplicity 2 zero eigenvalue. The stability test
proceeds by checking that all other eigenvalues of $\mathcal{M}_1$ are positive and the positive
 definiteness of all other blocks $\mathcal{M}_l$, $l\neq 1$. On the other hand, if $m=1$, i.e. the RE is asymmetric,
 the stability test consists of checking  that $\mathcal{P}_0$ is positive semi-definite with a 2-dimensional kernel.
\end{itemize}

\begin{remark}
The analysis described above does not cover the case $\omega\neq 0$ and $\mu=\Phi(a)=0$. We did not find any RE with this
property during our investigation and it is unclear to us that they exist when all vortices have equal strengths. On the other hand, existence
of such RE
can be shown explicitly if the vorticities are  allowed to be distinct   for $N= 3$.
\end{remark}


\section{CAPs of existence and stability of branches of relative equilibria}
\label{sec:CAPs}

We now explain how CAPs are implemented using the setting developed in sections \ref{s:existence} and \ref{sec:Stability} to 
establish existence and stability of branches of $\Z_m$-symmetric RE of \eqref{eq:motion}.

%
%

\subsection{Existence}
\label{ss:CAP-existence}
%
We start with a a non-degenerate $\Z_m$-symmetric  RE $(a_0,\omega_0)\in M\times \R$ approximated numerically 
by $(\overline{b}_0, \overline{\omega}_0)\in M_n\times \R$, i.e. $a_0\approx\rho(\overline{b}_0)$.
Based on the setting of  subsection \ref{ss:existenceRE-equal-strengths}, we determine a (local) continuation  branch of  $\Z_m$-symmetric RE, $m\geq 1$, parametrised
by $\omega$,   as unique zeros of  the function $F$ defined by \eqref{eq:augmap}  as follows. 
Consider a small  interval $[\omega_0, \omega_1]$ and set $\omega_s \bydef s\omega_1 + (1-s)\omega_0$ for $s \in [0,1]$. Let $\overline{x}_s = (\overline{b}_s, \overline{\lambda}_s, \overline{\alpha}_s)$ be a numerical branch segment such that $F(\overline{b}_s, \overline{\lambda}_s, \overline{\alpha}_s, \omega_s) \approx 0$ for $s \in [0,1]$.  Using Theorem~\ref{thm:Uniform Bounds},  we  obtain  existence and rigorous bounds of a unique $\tilde{x}(\omega_s) = (\tilde{b}(\omega_s), \tilde{\lambda}(\omega_s), \tilde{\alpha}(\omega_s))$ satisfying
\[
	F(\tilde{b}(\omega_s), \tilde{\lambda}(\omega_s), \tilde{\alpha}(\omega_s), \omega_s) = 0, \quad \forall s \in [0,1].
\]
Such branch is indeed unique in view of Theorem \ref{thm:continuation}. This implementation requires explicit expressions for the
gradient and hessian matrix of $F$ which are easily determined except perhaps for the terms coming from the derivatives of the reduced
Hamiltonian $h$ which are given in Appendix \ref{App:C}.

Using interval arithmetic, we then obtain existence and rigorous bounds of the RE  $(\tilde{a}(\omega_s), \omega_s)\in M\times \R$ with $\tilde{a}(\omega_s):=\rho(\tilde{b}(\omega_s))$. As one might expect, such rigorous bounds are better as the interval $[\omega_0, \omega_1]$ is shrunk.

\subsection{Stability}
\label{ss:CAP-stability}

%
%
%
%
%

Our strategy to prove that the matrix representation,  $\mathcal{M}(\omega_s)$, of $\left . d^2H_\omega(\tilde{a}(\omega_s)) \right |_{\mathcal{N}_a}$ is positive definite for all $s\in [0,1]$ is to first determine 
 positive enclosures for the eigenvalues of $\mathcal{M}(\omega_0)$. Once this is 
 done, we prove that $\mathcal{M}(\omega_s)$ is invertible for all $s\in [0,1]$, which is easily done by 
 applying interval arithmetic and the function {\tt inv} in INTLab
to all of the blocks described in section \ref{ss:Summary}. 
This proves  that $\mathcal{M}(\omega_s)$ is  positive definite for all $s\in [0,1]$ as a simple continuity argument shows.  
This  approach  avoids the computationally expensive validation of the eigenvalues of $\mathcal{M}(\omega_s)$ for values of $s\neq 0$.


The framework summarised in section \ref{ss:Summary} gives a convenient block diagonalisation of $\mathcal{M}(\omega_0)$ with blocks 
$\mathcal{M}_l(\omega_0)$, $l=0,\dots, \left [ \frac{m}{2} \right ]$, constructed in terms of $\tilde a(\omega_0)$ and given
 in Table \ref{eq:QlPl-summary}.
The calculation of rigorous enclosures of the eigenvalues of $\mathcal{M}(\omega_0)$ is then performed on each of the blocks $\mathcal{M}_l(\omega_0)$ as we explain below. We emphasise that the
block diagonalisation is useful to avoid clustering of eigenvalues. Moreover, the complexification of certain blocks is essential to deal with  eigenvalues which
would appear with double multiplicity in their real form. Such clusterings or multiplicities of eigenvalues would otherwise present significant
 computational challenges. 

\subsubsection{Validation of simple eigenvalues}

To begin, we repeat the existence procedure outlined  above to obtain a stronger 
rigorous error bound on  $\tilde{a}(\omega_0)$. More precisely, we use Theorem~\ref{thm:Uniform Bounds} letting $s = 0$ instead of $s \in [0,1]$. This makes the bound $\hat{Y}$ in such theorem disappear and provides a tighter enclosure on a point instead of a branch. Using this tighter bound, we compute an enclosure $\tilde{\mathcal{M}}_l(\omega_0)$
 of the block ${\mathcal{M}}_l(\omega_0)$ according to the prescription of section \ref{ss:Summary}.
We then compute approximate eigenvalues $(\bar{\lambda}_1, \bar{\lambda}_2, ..., \bar{\lambda}_{\D_l})$ and associated eigenvectors $(\bar{v}_1, \bar{v}_2, ..., \bar{v}_{\D_l})$ of $\tilde{\mathcal{M}}_l(\omega_0)$, where $\D_l$ is the
size of  $\tilde{\mathcal{M}}_l(\omega_0)$ (given explicitly in Table \ref{eq:QlPl-summary} in terms of $m$, $n$, $p$).

 For each   $k \in \{1,\dots,\D_l\}$, we  enclose rigorously the value of the true eigenpair $(\lambda_k,v_k)$ close to the numerical approximation $(\bar \lambda_k, \bar{v}_k)$ as follows. Let $j_k$ be such that $| \bar{v}_k \cdot e_{j_k}|= \max_{j=1,\dots,\D_l} \left( | \bar{v}_k \cdot e_j | \right)$ (note that $j_k$ need not be unique) and consider the map 
\begin{equation} \label{eq:G_eigenvalue_problem}
  \mathcal{G}_l: \R^{\D_l+1}\to \R^{\D_l+1}, \quad   \mathcal{G}_l(\lambda, v) \bydef 
    \begin{pmatrix}
        (\tilde{\mathcal{M}}_l(\omega_0) - \lambda I)v \\
        (v - \bar{v}_k) \cdot e_{j_k}
    \end{pmatrix},
\end{equation}
where $e_{j_k}$ is the $j_k^{th}$ canonical basis vector of $\R^{\D_l}$. The last component of \eqref{eq:G_eigenvalue_problem} is included to remove the scaling invariance of eigenvectors. With this definition, 
an isolated zero of $\mathcal{G}_l$ corresponds
to  a simple eigenpair $(\lambda, v)$ with the prescribed $j_k$ component of the eigenvector $v$ (which by 
construction is guaranteed to be non zero).  The requirement that the zero of $\mathcal{G}_l$ is isolated 
is only satisfied if $\lambda$ is a simple eigenvalue. Under this condition, we can apply a (non uniform) version of Theorem ~\ref{thm:Uniform Bounds} with $F=\mathcal{G}_l$  and $d=\D_l+1$, to obtain  rigorous bounds on the eigenpair of $(\lambda, v)$. The implementation of the method will fail in the presence of other eigenvalues close to $\lambda$ 
(clustering).

\subsubsection{Validation  of clustered eigenvalues}
\label{ss:CAP-clustered}

To handle the case of clustered eigenvalues, we follow the {\em argument principle and validated winding number computation} technique presented in Section 3.2 of \cite{MR4337868}. 
 Suppose that $\lambda_0$ is a clustered
eigenvalue of $\tilde {\mathcal{M}}_l(\omega_0)$ 
and consider the following small rectangle centred around $\lambda_0$ in the complex plane \[
	R_\epsilon := \{ z \in \C : | Re(z) - \lambda_0 | \leq \epsilon, \quad | Im(z) | \leq 1 \},
\]
where $\epsilon>0$ will be chosen  small enough to obtain tight bounds and 
such that no other eigenvalue of  $\tilde{\mathcal{M}}_l(\omega_0)$ lies on the boundary $\Gamma$ of $R_\epsilon$. Let ${\mathcal {G}}$ be the complex polynomial $\mathcal {G}(z) \bydef \det(\mathcal{M}_l(\omega_0) - z I)$. By Cauchy's argument principle,
\[
	\#zeros  = \frac{1}{2 \pi i} \oint_{\Gamma} \frac{{\mathcal {G}}'(z)}  {{\mathcal {G}}(z)} dz,
\]
where $\#zeros$ denotes the number of zeros of ${\mathcal {G}}$ inside the contour $\Gamma$. Considering that these
zeros are the  eigenvalues of $\mathcal{M}_l(\omega_0)$, applying  the substitution $u = {\mathcal {G}}(z)$, we obtain
\[
	\#eigenvalues = \frac{1}{2 \pi i} \oint_{{\mathcal {G}}(\Gamma)} \frac{du} {u} .
\]
Therefore, the number of eigenvalues inside $R_\epsilon$ is exactly the winding number of the curve ${\mathcal {G}}(\Gamma)$ about 0. To rigorously compute  this, we first naturally separate $\Gamma$ into its 4 linear segments $\Gamma_1, \Gamma_2, \Gamma_3, \Gamma_4$, each parameterized by a parameter $s$. Then we create a fine mesh of each $\Gamma_i$. For every segment of these meshes, denoted by $\Gamma^*$, we use interval arithmetic to compute a rigorous enclosure $ {\tilde {\mathcal {G}}(\Gamma^*)}$ of ${\mathcal {G}}(\Gamma^*)$. With all of these enclosures, we then count the number of times they intersect the positive real axis of the complex plane. To determine if the count  goes up or down by 1 at
an intersection, we use interval arithmetic to compute an enclosure of the imaginary part of 
$ \frac{d}{ds} \mathcal {G}(\gamma(s))$ where $\gamma(s)$ is the parametrisation of $\Gamma^*$.

\section{Examples of CAPs of existence and stability of RE.}
\label{s:CAPs-application}

Below we exemplify our framework by presenting proofs of existence 
and stability  of some specific RE.  In our discussion stability always means (nonlinear) $SO(3)_\mu$ stability as introduced in subsection \ref{ss:stability-prelim}.
For RE which are not equilibria ($\omega \neq 0$) this  coincides with orbital stability in 
virtue of Proposition \ref{prop:stability}. We always assume that all vortices have equal strengths. 

The RE will, in general, be $\Z_m$-symmetric and consist of $n$ rings, each consisting of a regular polygon of $m$ 
vertices, and  $p$ poles, so the total number of vortices $N=mn+p$.
%
Throughout this section we say that a RE $(a,\omega)\in M\times \R$ is of 
type  $(n_{k_1},m_{k_2},p_{k_3})$ if $a$ consists of $n=k_1$ rings, each consisting of a regular polygon of $m=k_2$ 
vertices, and $p=k_3$ poles. Nonsymmetric RE correspond to having $m=1$. 
We also simplify the notation and denote  $\mu=\Phi_3(a)\in \R$ (instead of $\mu=\Phi(a)=(0,0,\Phi_3(a))\in \R^3$ as before).
It is easy to see from the expression of $\Phi$ in \eqref{eq:Phi} that $|\mu|<N$ and that  $\mu\to N$ implies that the vortices are approaching a total 
collision at the North pole. 

We begin with subsection \ref{ss:CAP-minimisers}  which presents a precise description 
of the equilibrium configurations in Table \ref{ground-states} and shows that they are stable for $N=8,9,10,11$
(the case $N=7$ is degenerate and the stability in other cases was already known).
We continue our discussion in subsection \ref{ss:CAP-RE-groundstates} which
focuses on the determination of stable branches of RE emanating
from the equilibrium  configurations of Table \ref{ground-states}.
Finally, subsection \ref{ss:RE-totalcollision} focuses on stable RE in which the vortices are close to  total collision. 
Note that our discussion considers  opposite 
realms of the dynamics in terms of the center of vorticity:  all equilibria in  Table \ref{ground-states} have vanishing
momentum and instead, as mentioned above, the momentum is  maximised at total collision.

\subsection{Minimising and stability properties of the equilibria in Table \ref{ground-states}}
\label{ss:CAP-minimisers} 

Here we discuss the stability, non-degeneracy and global minimising properties of the equilibria in 
Table  \ref{ground-states} for the different values of $N$.

For $N=2,3,4,5,6,12$, it is known that these are ground states (see e.g.  section 11 of \cite{Beltran20} and the references therein).
These are $SO(3)$-stable equilibria as was proved by Kurakin \cite{K} for $N=4,6,12$. Moreover, with the 
help of a computer algebraic program, one can easily check the non-degeneracy, and hence also the $SO(3)$-stability
in the case $N=2,3,5$.

The case $N=7$ is special. On the one hand it is widely conjectured to be the ground state (see e.g. Conjecture 11.1 of
 \cite{Beltran20}). On the other hand, according to the computation in section \ref{ss:CAP-RE-groundstates},
  it is a degenerate critical point of $H$ so we cannot even prove 
that  it is a local minimiser or conclude stability by the study of the Hessian, see Remark \ref{rmk:N7deg} for details. 

Before proving the properties for $N=8,9,10,11$ we give a precise description of these configurations. We emphasise that the ground state for these values of $N$ is unknown and we only 
present a conjecture of their
form  which was inspired by elementary numerical exploration. For $N=8,9,10$,  symmetry considerations allowed
us to determine analytical expressions for these positions in terms of zeros of polynomials  as explained below.

For $N=8$ our conjectured ground state is an antiprism whose bases are squares and whose height is $2\sqrt{\frac{1}{7}(2\sqrt{58}-13)}$. This is consistent with the equilibrium given in  \cite[Table (4.3)]{Luis}.

Our conjectured ground state for $N=9$ is a polyhedron whose  vertices are located on three horizontal equilateral triangles. The middle triangle lies on the equator whereas the heights of the other 
two are $\pm z_0\approx \pm 0.703111$ where  $z_0^2\approx 0.494365$ is the smallest positive root of the polynomial $p(x)=64x^4+105x^3-87x^2-45x+27$. The top and bottom
triangles are aligned and the middle one is staggered with respect to them (i.e. the vertical projection of the top and bottom triangles coincide in an equilateral triangle that
makes an angle of $\frac{2\pi}{6}$ with the vertices of the middle triangle).

For $N=10$ our conjectured ground state is an antiprism whose bases are squares parallel to the equator, together with the North and South poles, and whose height equals $\frac{2}{3}\sqrt{2\sqrt{106} -19}$. This 
matches the equilibrium given in  \cite[Table (5.2)]{Luis}.

We were unable to determine analytic coordinates for our conjectured ground state for $N=11$. Instead, we obtained a CAP of the existence of a  $\mathbb{Z}_2$-symmetric
equilibrium with $5$ rings, each made up of $2$ vortices, together with  the North pole. The coordinates of the ring generators are within $10^{-13}$ of
\begin{equation*}
\begin{split}
 u_1&=  (0.414622789752781,   0.748445554893721,   0.517607180763029), \\
u_2 & =  (-0.984889687565531,  -0.009599383086507,   0.172916613347094), \\
u_3&=   (0.514196162925374,  -0.840060801883643,   0.172916613347109), \\
u_4&=   (0.402032242619801,   0.725718055903693, -0.558304020431036), \\
u_5&=   (0.518800630696144,  -0.287404425636917,  -0.805136387026196).
   \end{split}
\end{equation*}
The following theorem states the properties of these equilibria.
%
%
%

%
%
\begin{theorem}
\label{th:gs}
Let $N\in \{8,9,10,11\}$. The equilibrium configurations in Table \ref{ground-states} with positions given 
above are non-degenerate local minima of $H$. In particular,
these configurations are $SO(3)$-stable equilibria of the equations of motion \eqref{eq:motion}.
\end{theorem}
\begin{proof}
Denote by $\tilde a$ an enclosure of the equilibrium configuration.
We did a CAP 
following the prescription described in section \ref{ss:CAP-stability} to  obtain
 enclosures of the eigenvalues of the
blocks $\mathcal{M}_l$ of $\left . d^2H(\tilde{a}) \right |_{\mathcal{N}_a}$  using the following values of $m$, $n$ and $p$
according to the value of $N$:
\begin{equation*}
\begin{array}
[c]{|c|c|c|c|c|c|}%
\hline
N   & 8 &  9 &10 & 11  \\ \hline
m   &  4 &  3 & 4 & 2 \\ \hline
n   &  2 &  3 & 2 & 5 \\ \hline
p   &  0 &  0 & 2 & 1 \\ \hline
\end{array}.
\end{equation*}

For $l\neq 1$ all eigenvalues of $\mathcal{M}_l$  are shown to be  positive. On the other hand,
since for these configurations $\mu=0$, in accordance with the discussion in section 
\ref{ss:Summary}, we know that 
the block $ \mathcal{M}_1$ has a two-dimensional null space, and we validated that 
all other eigenvalues are positive. We mention that for $N=8,10$ the block $\mathcal{M}_2$, 
which is $4\times 4$, has two pairs of clustered eigenvalues that require  
the application of the technique described in section \ref{ss:CAP-clustered} for their validation.
\end{proof}

\subsection{Stable RE arising from the equilibria in Table \ref{ground-states}}
\label{ss:CAP-RE-groundstates}

In this section  we focus on finding (enclosures  of)  stable branches of RE 
emanating from the equilibrium configurations of Table \ref{ground-states}. 
This can be done analytically for the values of $N=4,7$ since such stable branches consist of 
one ring (see treatment below). For all other values of $N\in  \{5,\dots, 12\}$ we rely on CAPs as we explain below.  

Fix $N\in \{5,\dots, 12\}$, $N\neq 7$. The local 
existence of symmetric branches of RE emanating from the equilibrium in  Table \ref{ground-states} is guaranteed by Corollary \ref{cor:existenceREsymmetric} for the $\mathbb{Z}_m$-symmetries indicated in Table \ref{ground-states} and illustrated in Figures \ref{Fig:GS-Known} and \ref{Fig:GS-Conjecture}
(the non-degeneracy 
condition in the corollary holds in view of the discussion in section  \ref{ss:CAP-minimisers}).
However, nothing can be said a priori about their stability. In order to prove existence 
of stable branches, that we conjecture to
be minimisers of $H$ for fixed $\mu$,   we proceeded as follows:
\begin{enumerate}
\item[1.] We obtained numerical   approximations of the $\Z_m$-symmetric branches of RE  for all symmetries 
indicated in Table  \ref{ground-states}. For each of them we ran  numerical stability tests aiming
to identify stable candidates. We found exactly one stable branch for all values of $N$ except $N=10$ for
which we found none. 
\item[2.] For the stable candidate identified above, we implemented the CAP of existence
and stability described in Section \ref{sec:CAPs}. The main difficulty that we encountered  is that two of the eigenvalues of the  $\mathcal{M}_1$ block approach zero
as $\omega\to 0$  which is due to the change of dimension of the symplectic slice at zero momentum (compare Lemma \ref{lemma:symp-slice-def} with Lemma 
\ref{lemma:basis-Bl-mu0}). For this reason,  the stability of the branch could not
be determined using CAPs for values of $\omega$ arbitrarily close to $0$.
\end{enumerate}

%
%
%
The results of the CAPs are illustrated with figures that indicate the position of the vortices of the (branch of) RE. We also present the energy-momentum bifurcation
diagram giving a plot of the values of  $H$ and $\mu$ along these RE. 
The figures employ the following colour code: 
\begin{enumerate}
\item[$\bullet$] green: both existence and stability  have been proved;
\item[$\bullet$] yellow: existence is proved but stability is not (due to eigenvalues that are either too close to zero or clustered);
\item[$\bullet$] black: we were unable to prove existence due to the proximity to a bifurcation point at which the Jacobian $D_{x}F$
appearing in the  Newton-Kantorovich theorem \ref{thm:Uniform Bounds} is non-invertible.
\end{enumerate}
%
%


\subsubsection*{$N=4$}

In this case the computations can be done analytically with the help of a symbolic software. Therefore we do not present a computed assisted proof. 
A comparison of the branches of RE emanating from the tetrahedron  (obtained numerically) was given before in \cite{MR1811389} (see Figure 7)
showing numerical evidence that the branch $(n_1,m_3,p_1)$ (corresponding to the $\mathbb{Z}_3$-symmetry and illustrated at equilibrium in Figure \ref{Fig:GSN4Z3}) 
is the global minimiser arising from the tetrahedron. More precisely,  if three vortices are placed in a ring forming 
an equilateral triangle at height $z\in (-1,1)$ and another vortex is placed at the North pole, then this is a RE provided that $\omega =\frac{1+3z}{2(1-z^2)}$.
The tetrahedron configuration is attained at $z=-\frac{1}{3}$. As follows from \cite[Theorem 6.2]{Mo11} these RE are  stable if $z>-\frac{1}{3}$.

On the other hand, the branch $(n_2,m_2,p_0)$ corresponding to the $\mathbb{Z}_2$-symmetry is unstable near the tetrahedron. This branch gains stability for larger
values of $\mu$ and in fact appears to minimise the Hamiltonian for fixed momenta for a certain range of positive $\mu$ as predicted by Figure 7 in \cite{MR1811389}.

\subsubsection*{$N=5$}

Our CAPs establish that the branch $(n_2,m_2,p_1)$ is stable for    $0.1\leq \omega \leq 0.49$ which corresponds
to $0.18\leq \mu \leq 0.98$. As may be appreciated from  the leftmost panel in Figure \ref{Fig:CAP5}, the rings in this branch are ``staggered". We conjecture that this branch
(illustrated at equilibrium in Figure \ref{Fig:GSN5Z2}) is born
stable close to the triangular bipyramid and is the global minimiser of $H$ 
for fixed values of $\mu>0$ close to zero. 
 
The two rings of the branch $(n_2,m_2,p_1)$ are positioned at the equator when $\mu=1$ where it meets the  branch $(n_1,m_4,p_1)$ which,
as shown in  \cite[Theorem 6.2]{Mo11}, is stable as long as  $z>0$ where $z$ is the height of the ring of 4 vortices and the other vortex is at the North pole. Note however 
that the branch $(n_1,m_4,p_1)$
cannot  converge to the triangular bipyramid since this figure does not have a $\mathbb{Z}_4$-symmetry.
This bifurcation is also illustrated  in Figure \ref{Fig:CAP5}.

\begin{figure}[hh]
\centering
\includegraphics[width=15cm]{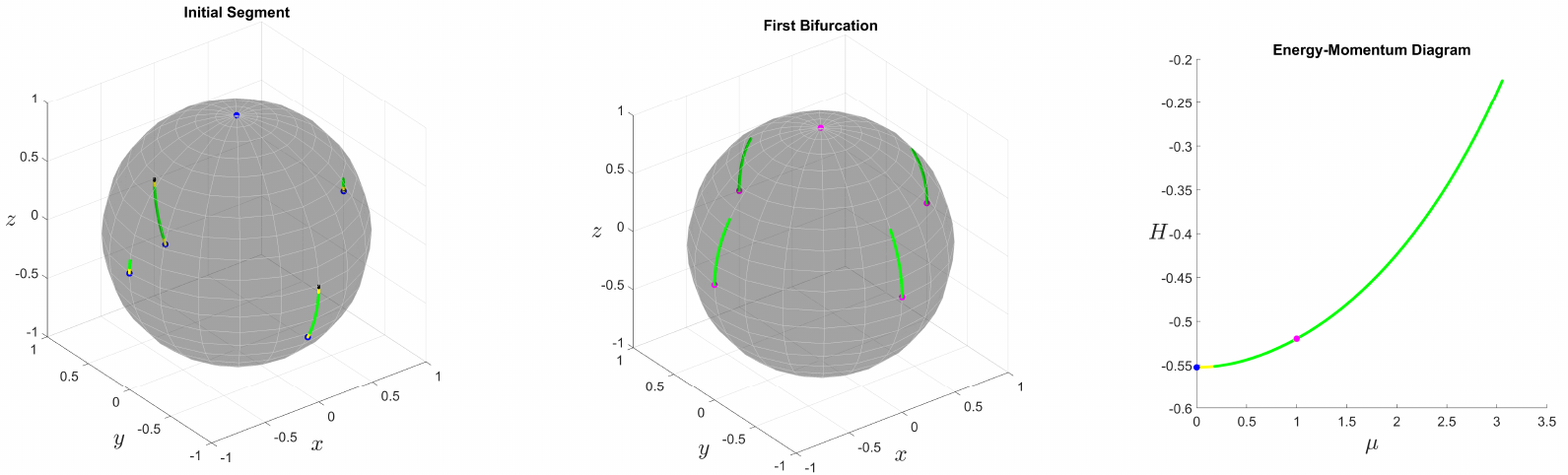}
\caption{\small{CAP of existence and stability of RE for $N=5$ near the ground state. The branch on the leftmost panel is of type $(n_2,m_2,p_1)$ with staggered
rings and the branch in the middle is of type $(n_1,m_4,p_0)$. }\label{Fig:CAP5} }
\end{figure}

\subsubsection*{$N=6$}

We establish stability via CAPs of  the branch $(n_2,m_3,p_0)$ where the two rings of three vortices are staggered for the values of   $0.2\leq \omega \leq 1.4$ which correspond
to $0.32 \leq \mu \leq 2.36$. We conjecture that this branch (illustrated at equilibrium in Figure \ref{Fig:GSN6Z3}) is born
stable close to the octahedron and is the global minimiser of $H$ 
for fixed values of $\mu>0$ close to zero. Our analysis shows that this branch loses stability around $\omega=1.4150$ and $\mu=2.38$. We claim that an asymmetric
RE gains stability at this point since we do not observe a bifurcation in the map $F$ \eqref{eq:augmap} with $m=3$ and no other apparent symmetries are visible.
The results are illustrated in Figure  \ref{Fig:CAP6}.
\begin{figure}[hh]
\centering
\includegraphics[width=12cm]{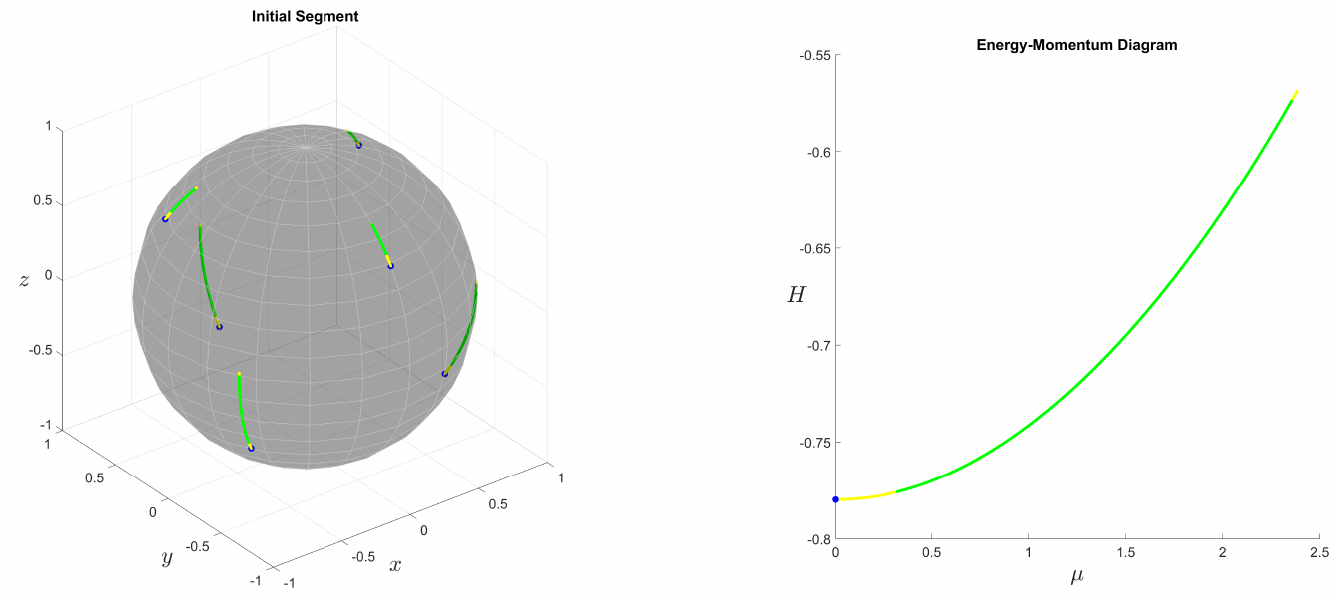}
\caption{\small{CAP of existence and stability of the RE of type $(n_2,m_3,p_0)$ for $N=6$ near the ground state. }\label{Fig:CAP6} }
\end{figure}

\subsubsection*{$N=7$}

In this case  Corollary \ref{cor:existenceREsymmetric} may not be used to prove local existence of branches
of RE since the non-degeneracy condition is not satisfied. However, 
a direct calculation shows that the configuration $(n_1,m_5,p_2)$ 
consisting of 1 ring of 5 vortices arranged at the vertices of a horizontal pentagon at height $z\in (-1,1)$ together
with the North and South poles is a RE provided that $\omega=\frac{3z}{1-z^2}$. The corresponding value
of the momentum is $\mu=5z$ since the contribution from the vortices at the North and South pole vanishes.
Note that both the angular velocity $\omega$ and the momentum $\mu$ vanish at the pentagonal bipyramid
which is attained when  $z=0$.

The stability of these RE can be examined analytically with the approach of 
section \ref{sec:Stability}. As seen from Table \ref{eq:QlPl-summary}, the stability is guaranteed by the
positive definiteness of the Hermitian 
matrices $\mathcal{Q}_1$ and $\mathcal{Q}_2$ which are respectively $3\times 3$
and $2\times 2$. With the help of a symbolic software one finds that they are both real and given by
\begin{equation}
\label{eq:QforN7}
\mathcal{Q}_1=\left(
\begin{array}{ccc}
 25 z^2+5 & -\frac{5}{2} (z (5 z+2)+5) & \frac{5}{2} (z (5 z-2)+5) \\
 -\frac{5}{2} (z (5 z+2)+5) & 25 (4-3 z) z+65 & \frac{5}{2} \left(1-5 z^2\right) \\
 \frac{5}{2} (z (5 z-2)+5) & \frac{5}{2} \left(1-5 z^2\right) & 65-25 z (3 z+4) \\
\end{array}
\right), \qquad \mathcal{Q}_2=\left(
\begin{array}{cc}
 15 & 0 \\
 0 & 15 z^2 \\
\end{array}
\right).
\end{equation}

It is clear that $\mathcal{Q}_2$ is positive definite for any value of $z$ different from $0$. On the other
hand, we have the following.
\begin{proposition}
The matrix $\mathcal{Q}_1$ in \eqref{eq:QforN7} is positive definite for $0<|z|<\sqrt{\frac{1}{35} \left(43-4 \sqrt{109}\right)}\approx 0.188132$.
\end{proposition}
\begin{proof}
A direct calculation gives $\det \mathcal{Q}_1=\frac{9375}{2} z^2 \left(35 z^4-86 z^2+3\right)$, which shows that
$\det \mathcal{Q}_1$ can only change sign when $z^2=0$ or $z^2=\frac{1}{35} \left(43\pm 4 \sqrt{109}\right)$.
In particular, none of its eigenvalues can vanish for $0<|z|<\sqrt{\frac{1}{35} \left(43-4 \sqrt{109}\right)}$.
It is then elementary to verify that all 3 eigenvalues are positive for a particular value of $z$ in this interval. 
\end{proof}

Therefore, we have proved that the branch $(n_1,m_5,p_2)$ (illustrated at equilibrium in Figure \ref{Fig:GSN7Z5}) arises a stable branch
of RE from the pentagonal dypiramid and loses stability at $|z|=\sqrt{\frac{1}{35} \left(43-4 \sqrt{109}\right)}$ which corresponds to $|\omega|=\frac{1}{4} \sqrt{89-8 \sqrt{109}}$. We conjecture that this branch minimises
$H$ for fixed and small values of the momentum. (We mention that these analytic results are in perfect correspondence with
the ones predicted by our numerics and CAPs).

\begin{remark}
\label{rmk:N7deg}
Denote by $a$  the equilibrium configuration  attained when $z=0$ (the pentagonal 
bipyramid of Table \ref{ground-states}).  The zero eigenvalue of the matrix $\mathcal{Q}_2$ at this point corresponds to a 2-dimensional null
space of $d^2H(a)$ spanned by the real and imaginary parts of the vector $\hat C_{1,2}$. These vectors
lie on the space $U_2$ of the symplectic slice defined in Lemma \ref{lemma:symp-slice-def} and are therefore transversal to  $\so(3).a$. This leads to the surprising conclusion  that, according to Definition \ref{def:Bott}, the $SO(3)$-orbit of the pentagonal 
bipyramid is a degenerate critical manifold of $H$ when $N=7$. In particular, it is not possible 
to determine that this configuration is a local minimum of $H$, nor that it is a stable equilibrium of \eqref{eq:motion}, using only 
information from  the second derivatives of $H$. This degenerate situation is reminiscent of the well-known {\em Thomson heptagon problem} of
stability of a ring of $7$ vortices on the plane (see e.g. \cite{Ku2,Schmidt}). 
\end{remark}

\subsubsection*{$N=8$}

Using CAPs we establish stability of the branch $(n_2,m_4,p_0)$ emerging from the cubic antiprism for 
$0.2\leq \omega \leq 1.6$ which corresponds to $0.31\leq \mu \leq 2.6$. We conjecture that this branch (illustrated at equilibrium in Figure \ref{Fig:GSN8Z4})
is born stable and minimises $H$ for small fixed values of $\mu$. The branch is illustrated in the leftmost panel
of Figure \ref{Fig:CAP8} where one can appreciate that the two rings are staggered. 

Our investigations showed that this branch bifurcates into $(n_4,m_2,p_0)$ around  $\omega = 1.61$ and $\mu = 2.61$. We give a CAP of 
the stability of this new branch on the interval  $1.62 \leq \omega \leq 1.9$ which corresponds to $2.62 \leq \mu \leq 3.05$, and is illustrated in 
the middle panel of Figure \ref{Fig:CAP8}. This branch loses stability at around  $\omega = 1.94$ and $\mu = 3.12$ and we conjecture 
that an asymmetric
RE gains stability around these values since  we do not observe a bifurcation in the corresponding map $F$ \eqref{eq:augmap} with $m=2$ 
and no other apparent symmetries are visible.

\begin{figure}[hh]
\centering
\includegraphics[width=15cm]{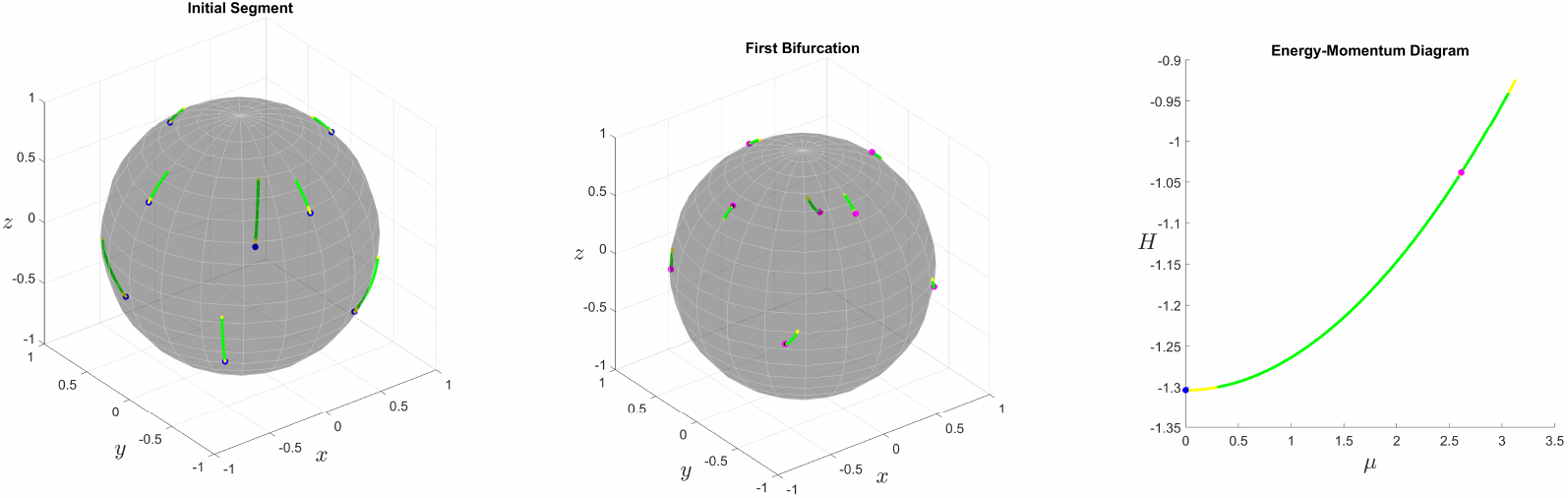}
\caption{\small{CAP of existence and stability of RE for $N=8$ near the ground state. The branch on the leftmost panel is of type $(n_2,m_4,p_0)$ with staggered
rings and the branch in the middle is of type $(n_4,m_2,p_0)$. }\label{Fig:CAP8} }
\end{figure}

\subsubsection*{$N=9$}

Our CAP establishes stability of the branch  $(n_3,m_3, p_0)$ on the interval $0.3 \leq \omega \leq 4.98$ corresponding
to $0.45 \leq \mu \leq 5.93$ and we conjecture that this branch (illustrated at equilibrium in Figure \ref{Fig:GSN9Z3}) is born stable and 
minimises $H$ for fixed small positive values of $\mu$. We found
2 bifurcations that maintain the $(n_3,m_3, p_0)$ symmetry around $\omega = 4.99$ and $\omega = 5.22$ 
(which correspond to $\mu = 5.94$ and $\mu = 6.05$). For these new branches we give CAPs of stability on the intervals
$5.03 \leq \omega \leq 4.98$, and $5.225 \leq \omega \leq 7.53$, which correspond to 
 $5.96 \leq \mu \leq 6.03$, and $6.06 \leq \mu \leq 6.85$. These results are illustrated in Figure  \ref{Fig:CAP9}.

\begin{figure}[hh]
\centering
\includegraphics[width=15cm]{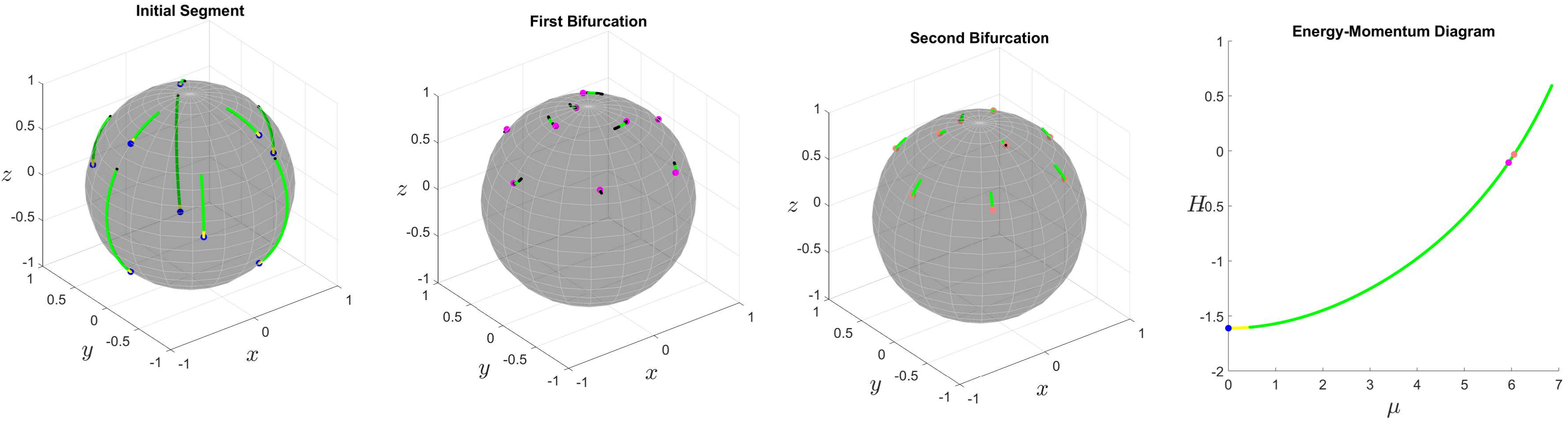}
\caption{\small{CAP of existence and stability of RE for $N=9$ near the ground state. All branches are of type $(n_3,m_3, p_0)$. }\label{Fig:CAP9} }
\end{figure}

\subsubsection*{$N=10$}

Our investigations indicated that the symmetric branches emerging from the equilibrium  (i.e. $(n_2,m_4, p_0)$ and
$(n_5,m_2, p_0)$ illustrated in Figures \ref{Fig:GSN10Z4} and \ref{Fig:GSN10Z2}) fail the stability test of Section \ref{sec:Stability}. Therefore, we conjecture that the stable branch that
minimises $H$ for small fixed values of $\mu$ is asymmetric. 

We  looked for this asymmetric minimising branch 
numerically and, in view of Remark \ref{rmk:continuation},  were forced to implement our continuations starting from 
 $\omega$ away from zero. We were
only able to produce a CAP of existence of this branch for $1\leq \omega \leq 2.27$ which corresponds to $1.5\leq \mu \leq 3.5$. On the other
hand, we give a CAP of stability for the interval $1.56 \leq \omega \leq 2.2$, which corresponds to $2.39 \leq \mu \leq 3.40$. 
The branch remains stable until a bifurcation around $\omega = 2.29$ and $\mu = 3.53$, after which we could not find a stable branch.
These results are illustrated in Figure \ref{Fig:CAP10}.

\begin{figure}[hh]
\centering
\includegraphics[width=15cm]{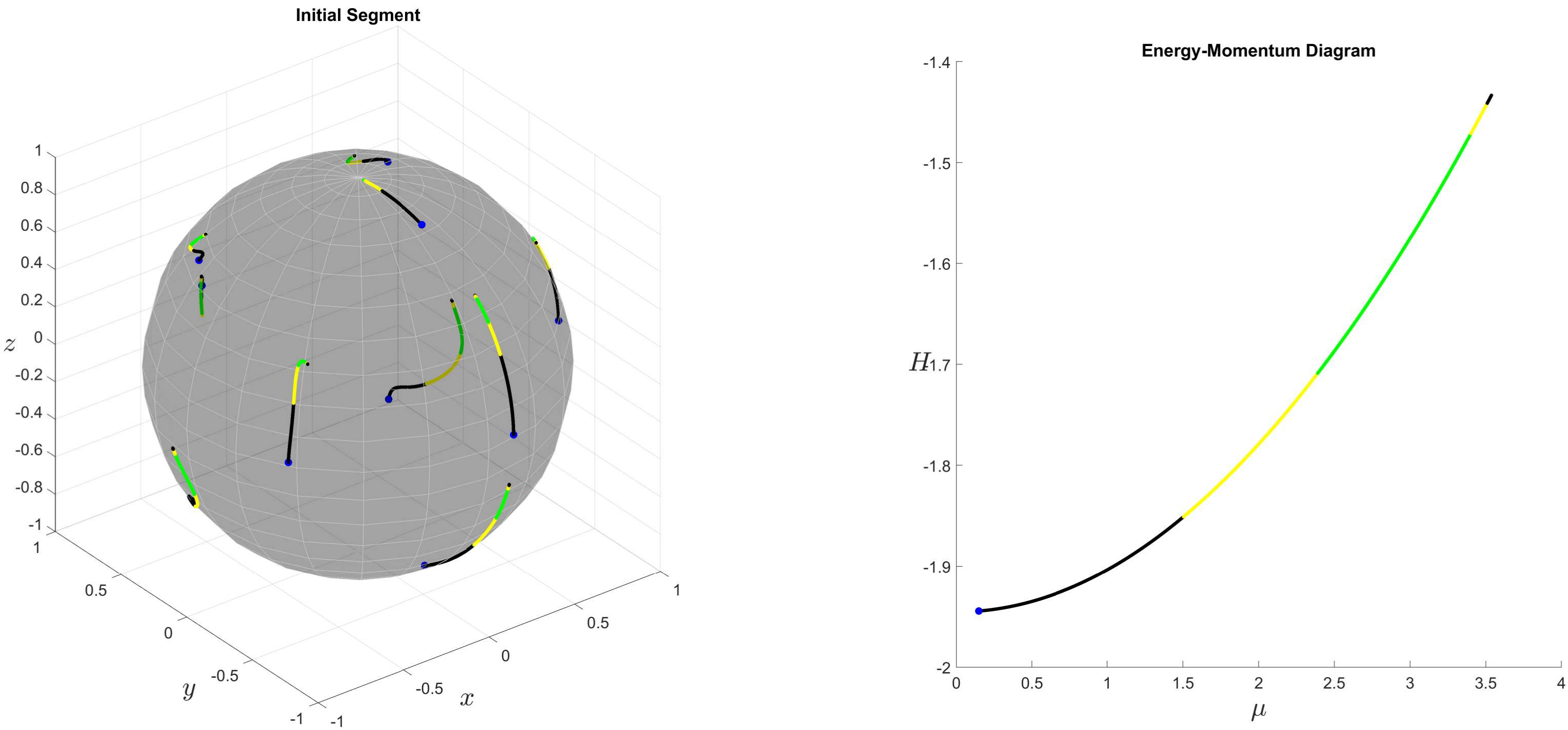}
\caption{\small{CAP of existence and stability of RE for $N=10$ near the ground state. The branch has no symmetry. }\label{Fig:CAP10} }
\end{figure}

\subsubsection*{$N=11$}

We give a CAP of stability of the branch $(n_5,m_2, p_1)$ on the interval $0.5 \leq \omega \leq 1.26$ which corresponds to $0.74 \leq \mu \leq 1.88$, as illustrated in the
left panel of Figure  \ref{Fig:CAP11}. We conjecture that this branch (illustrated in Figure \ref{Fig:GSN11Z2}) is born stable and minimises $H$ for small fixed values of $\mu$.
This branch bifurcates and loses stability around $\omega = 1.28$ and $\mu = 1.90$. We were able to give a CAP of stability 
of an asymmetric branch in the interval $1.35 \leq \omega \leq 1.46$ which corresponds to $2.02 \leq \mu \leq 2.19$ (illustrated in the middle panel of  Figure  \ref{Fig:CAP11}). 
 This branch undergoes another bifurcation around $\omega = 1.47$ and $\mu = 2.21$ but we were not able to find a stable branch afterwards.

\begin{figure}[hh]
\centering
\includegraphics[width=15cm]{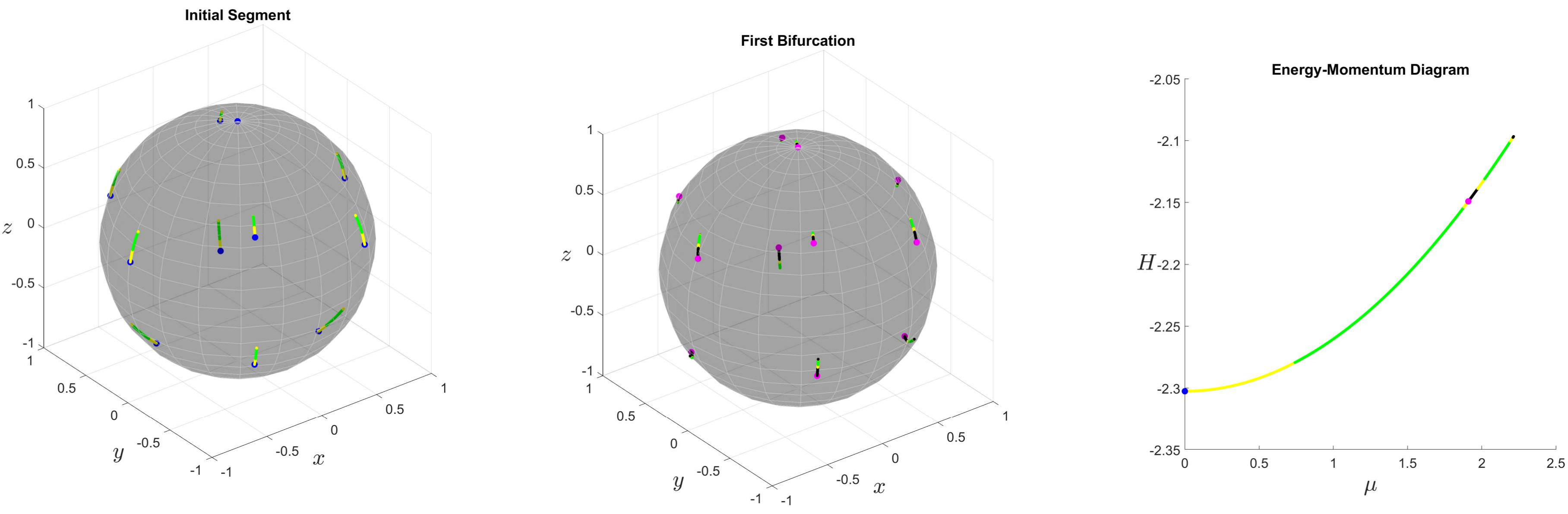}
\caption{\small{CAP of existence and stability of RE for $N=11$ near the ground state. The first branch is $(n_5,m_2, p_1)$ and the second has no
symmetries. }\label{Fig:CAP11} }
\end{figure}

\subsubsection*{$N=12$}

We established stability of the branch  $(n_4,m_3, p_0)$ with the CAP illustrated in Figure  \ref{Fig:CAP12} on the interval  $0.9 \leq \omega \leq 2.84$ which corresponds to $1.31 \leq \mu \leq 4.26$. We conjecture that this branch (illustrated at equilibrium in Figure \ref{Fig:GSN12Z3}) is born stable and minimises $H$ for fixed small positive values of $\mu$. This branch loses stability around 
 $\omega = 2.85$ and $\mu = 4.28$. We claim 
that an asymmetric
RE gains stability around these values since we do not observe a bifurcation in the corresponding map $F$ \eqref{eq:augmap} with $m=3$ 
and no other apparent symmetries are visible.

\begin{figure}[hh]
\centering
\includegraphics[width=15cm]{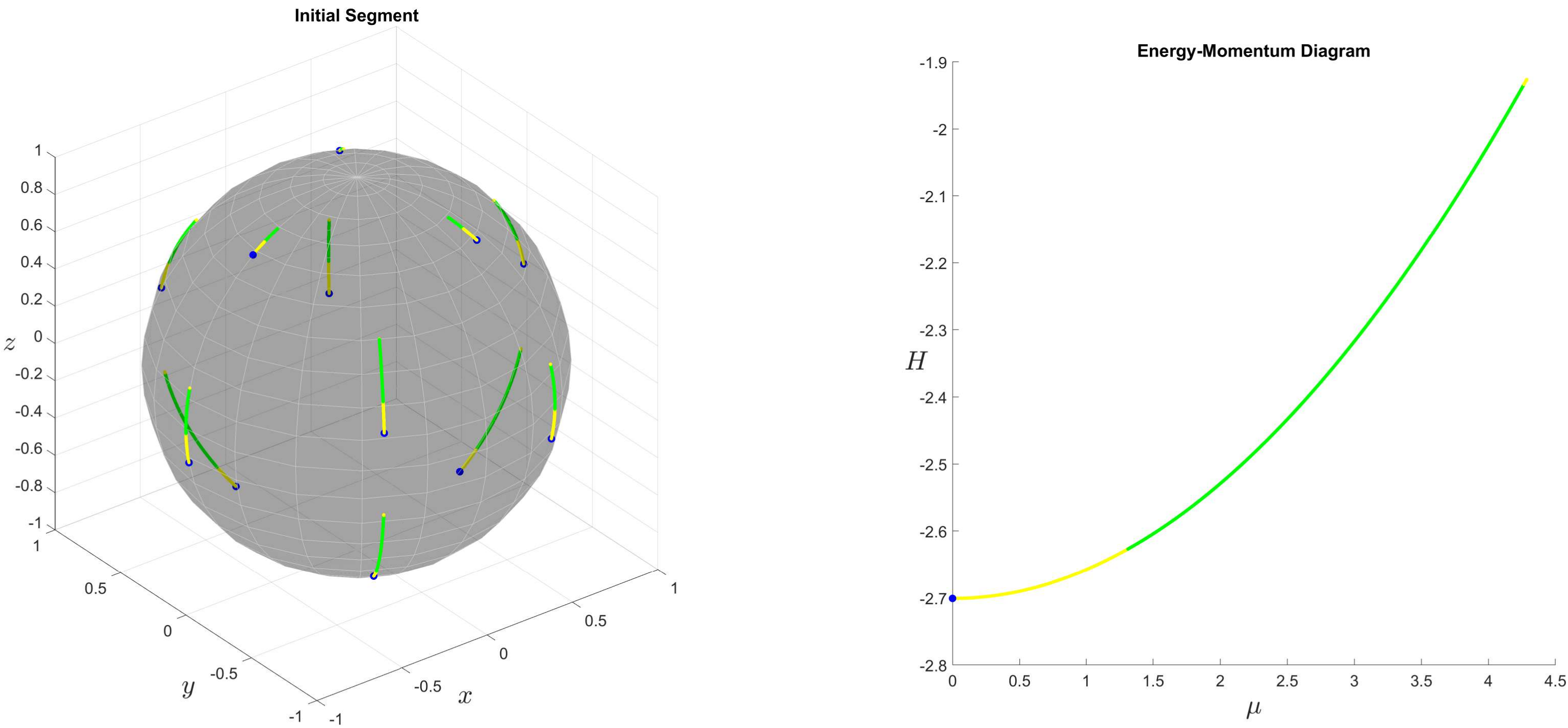}
\caption{\small{CAP of existence and stability of RE for $N=12$ near the ground state. The  branch is $(n_4,m_3, p_0)$. }\label{Fig:CAP12} }
\end{figure}

\subsection{Stable RE near  total collision}
\label{ss:RE-totalcollision}

We now focus on investigating stable RE near total collision at the North pole. It is easy to see, from the expression of $\Phi$ in \eqref{eq:Phi}, that this corresponds to  $\mu$ approaching the 
upper bound $N$ of $\Phi$.

The simplest RE of the equations of motion \eqref{eq:motion} is $(n_1,m_N,p_0)$ (consisting of a single horizontal ring where  the 
 vortices are located at the  vertices of a regular $N$-gon). Such RE is known to be (linearly) unstable if $N\geq 7$ and instead (nonlinearly) stable if $N\leq 6$ and the height $z$ of the ring is sufficiently
close to one of the poles (the threshold value of $z$ is given in Table \ref{eq:table-threshold1ring} below). These 
conclusions follow from \cite[Theorem 5.2]{Mo11} and  appear in earlier references \cite{PolvDritschel,BoattoCabral}.
 It is reasonable to conjecture that for $N\leq 6$ these RE minimise $H$
for fixed values of $\mu$ sufficiently close to $N$ (i.e. near total collision).

The  presence of an additional vortex at the North pole has a stabilising effect. Indeed, as follows from   \cite[Theorem 6.2]{Mo11} (see also 
\cite{CabralMeyerSchmidt}) the RE  $(n_1,m_{N-1},p_1)$ (consisting of 
 one vortex at the North pole and   a  regular horizontal $(N-1)$-gon) is stable for $N\leq 9$ 
 provided again that the height  $z$ of the ring containing the polygon is sufficiently
close to the North pole (the threshold value of $z$ is given in Table \ref{eq:table-threshold1ring} below). Therefore, it is reasonable to conjecture that for $N=7,8,9$ these RE minimise $H$
for fixed values of $\mu$ sufficiently close to $N$ and this conjecture is supported by numerical experiments.

\begin{table}[h]
\begin{center}
\begin{tabular}[[c]{|c|c|c|}%
\hline
$N$  &$ \begin{array}{cc} \mbox{Stability region of} \\  \mbox{the RE $(n_1,m_N,p_0)$} \end{array} $ & $ \begin{array}{cc} \mbox{Stability region of} \\  \mbox{the RE $(n_1,m_{N-1},p_1)$} \end{array} $  \\ \hline
$4 $ & $|z|> \frac{1}{\sqrt{3}}  $&$z>-\frac{1}{3}$ \\ \hline
$5$  &$ |z|> \frac{1}{\sqrt{2}}$ &$z>0$  \\ \hline
$6$ &$ |z|> \frac{2}{\sqrt{5}} $&$z> \frac{1}{5} (\sqrt{6}-1)$ \\ \hline
$7$ &   - &$z>\frac{1}{6} (\sqrt{19}-1)   $\\ \hline
   $8$ & -  & $z>\frac{5}{7}$\\ \hline
 $  9$  & - & z>$\frac{1}{8} (\sqrt{65}-1)$ \\ \hline
\end{tabular}
\end{center}
\caption{Stability regions for RE with 1 ring with (right column) and without (left column) a vortex at the North pole. As usual, $N$ is the total
number of vortices and all of them have identical strengths. }
\label{eq:table-threshold1ring}
\end{table}

%
%
%

For $N\geq 10$ both branches $(n_1,m_N,p_0)$ and $(n_1,m_{N-1},p_1)$ are always unstable and 
there is no obvious candidate for the branch of (stable) RE minimising the Hamiltonian $H$ as the momentum
$\mu\to N$ (i.e. as the vortices approach total collision). Below we attempt to give a partial answer to this question for $N=10,11,12$. Starting from a numerical approximation
of such minimiser for a specific $\mu$ close to $N$, we applied the procedure
of section \ref{sec:CAPs} to prove existence and stability of a RE. However,
%
%
in contrast with the approach followed in section \ref{ss:CAP-RE-groundstates}, 
due to the large fluctuations of the energy and angular velocity 
obtained for small configuration variations near total collision, we did not attempt to validate 
branches of RE but had to settle with the analysis of an isolated one (that we chose to have $\omega=50$). The RE that we found
are illustrated in Figure \ref{Fig:LargeRE} and approximated numerically  in Appendix \ref{App:NumCollision}. They all appear to have a central triangle with the remaining vortices organised more or less uniformly 
in an outer ring. We were only able to make precise symmetry statements for $N=12$.

\subsubsection*{$N=10$}

We proved existence and stability of the RE given in \eqref{eq:RELarge10} for  $\omega =50$. The 
tolerance bound on  each coordinate is $4 \times 10^{-13}$. This RE is illustrated in Figure \ref{Fig:CAP10Large}. Although
it appears to have some symmetry, we are unable to make any precise statements about it. Note however that there  are
 four pairs of  vortices  whose heights differ by at most $8 \times 10^{-13}$ (i.e. the pairs $(a_1, a_6), (a_2, a_4), (a_5, a_{10})$ and $(a_8, a_9)$ in \eqref{eq:RELarge10}).

\subsubsection*{$N=11$}

For  $\omega =50$, we proved the existence of the RE illustrated in Figure \ref{Fig:CAP11Large} and whose coordinates are given by
\eqref{eq:RELarge11} (with  
tolerance bound on  each coordinate of $6 \times 10^{-11}$). As above, we are unable to make any precise statements about the symmetries of this RE. We can only establish that the height of 
 4  pairs of  vortices   differs by at most $2\times 10^{-10}$ (i.e. the pairs $(a_1,a_{10}), (a_2,a_8), (a_5, a_7)$ and $(a_6,a_{11})$ in \eqref{eq:RELarge11}) while the heights of the others are shown to be distinct.  For
  this RE we can only report numerical evidence of stability since our CAP for stability failed due to 
 the clustering of eigenvalues.

\subsubsection*{$N=12$}

This time, for $\omega=50$, we obtained CAPs of the existence and stability  of a $\Z_3$-symmetric 
RE of type $(n_4,m_3,p_0)$ whose generators are given by 
\eqref{eq:RELarge12} (with  
tolerance bound on each coordinate of $3 \times 10^{-13}$). Our proof also shows that the height of two of the four rings (those corresponding to $u_3$ and
$u_4$ in \eqref{eq:RELarge12}) differ by at most $5\times 10^{-13}$. This RE is illustrated in Figure \ref{Fig:CAP12Large}.

\begin{figure}[ht]
\begin{subfigure}{.31\textwidth}
  \centering
  \includegraphics[width=.9\linewidth]{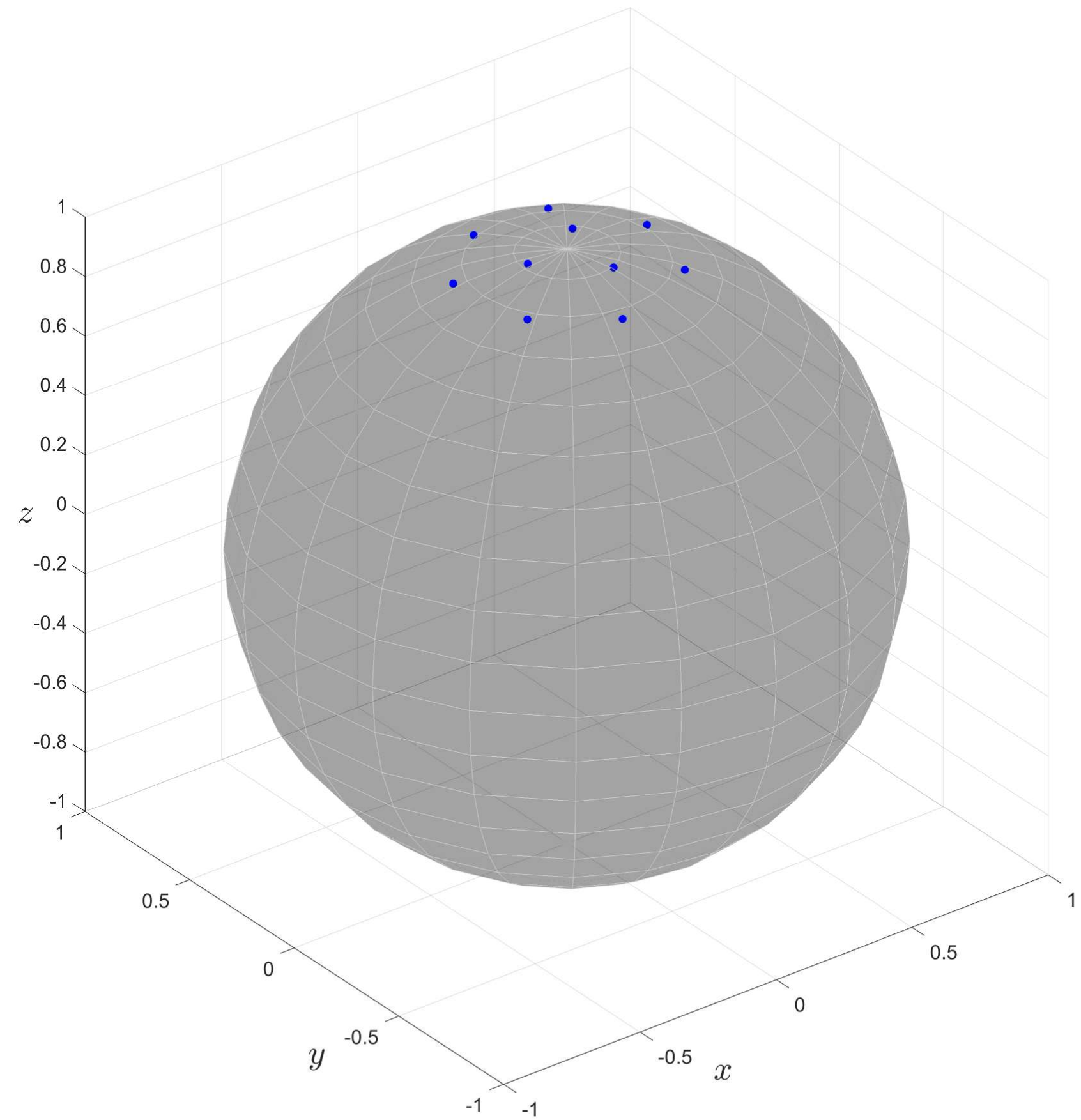}  
  \caption{$N=10$}
  \label{Fig:CAP10Large}
\end{subfigure}
\quad
\begin{subfigure}{.31\textwidth}
  \centering
  \includegraphics[width=.9\linewidth]{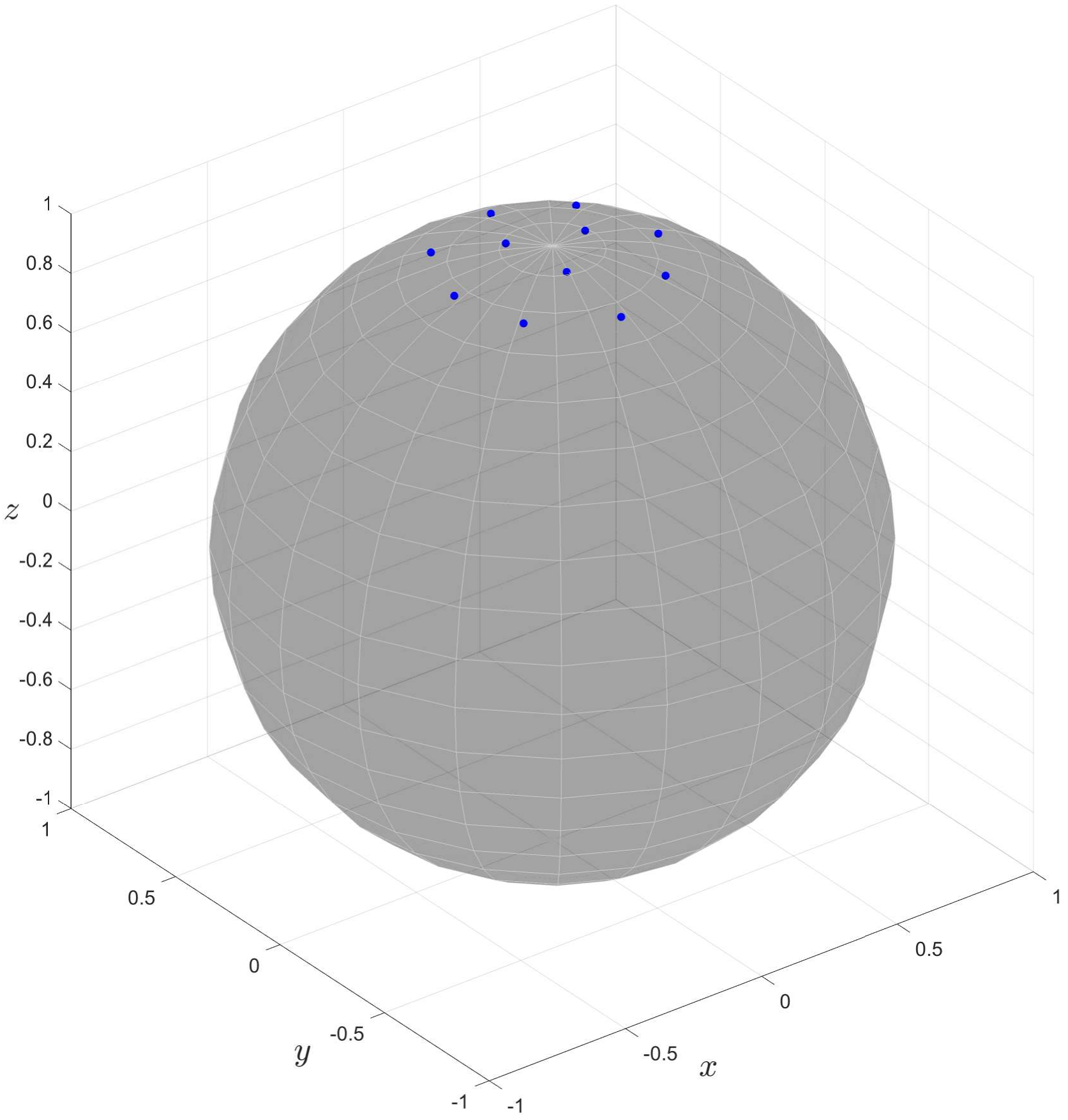}  
  \caption{$N=11$}
  \label{Fig:CAP11Large}
\end{subfigure}
\quad
\begin{subfigure}{.31\textwidth}
  \centering
  \includegraphics[width=.9\linewidth]{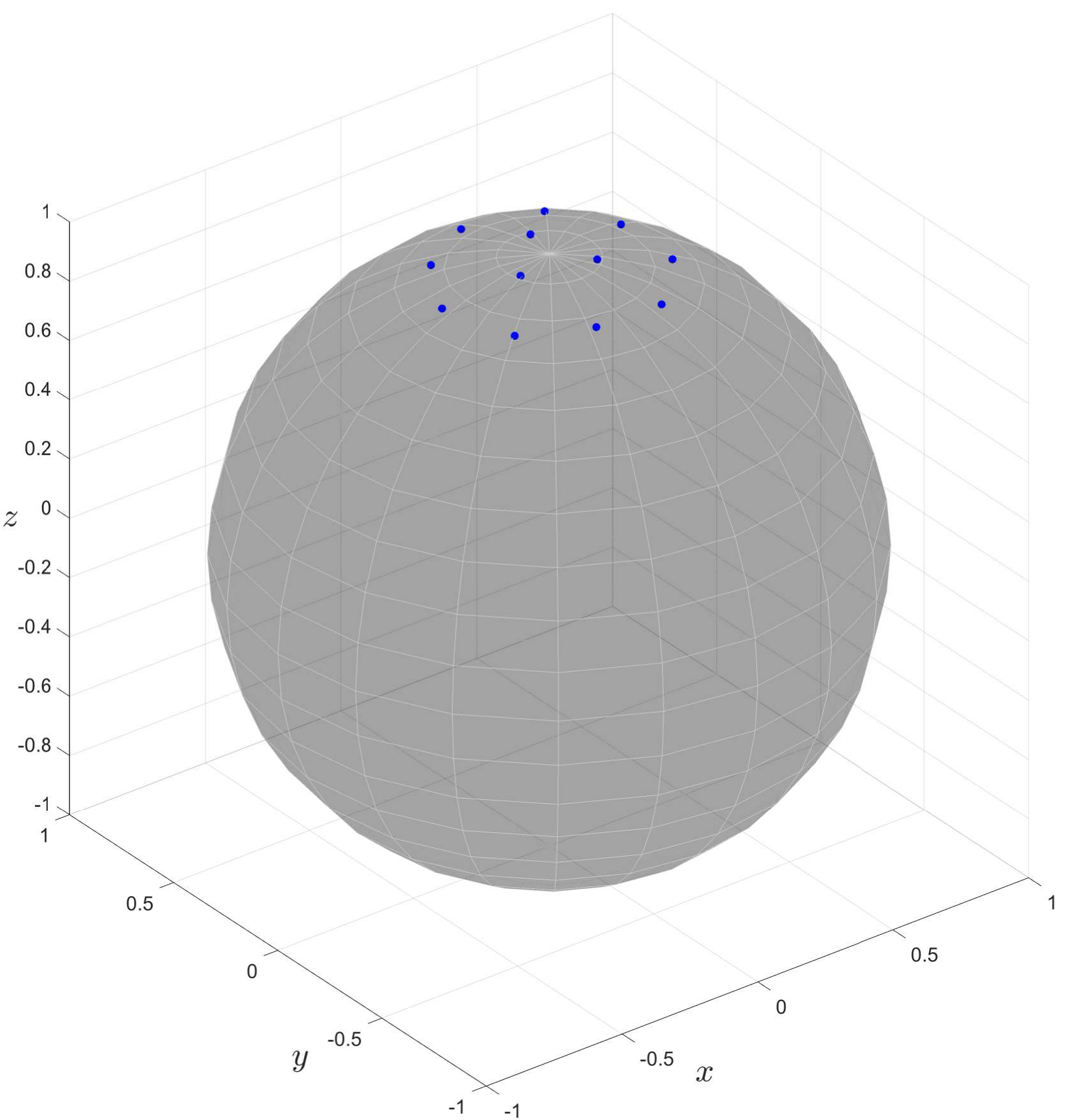}  
  \caption{$N=12$}
  \label{Fig:CAP12Large}
\end{subfigure}
\caption{Relative equilibria near total 
collision having  $\omega=50$. Their existence is established via CAPs (with tolerance bounds given in the text). On the other hand, their stability has also been demonstrated for  $N=10$ and $N=12$ with a CAP whereas for $N=11$ we only have numerical evidence of it.}
\label{Fig:LargeRE}
\end{figure}

\section*{Acknowledgements}
We are grateful to C. Beltr\'an for indicating some references to us on Smale's problem \#7. KC an JPL acknowledge support from NSERC. CGA 
 acknowledges support for his research from the Program UNAM-PAPIIT-IA100423.
LGN acknowledges support from the projects MIUR-PRIN 20178CJA2B {\em New Frontiers of Celestial Mechanics: theory and applications} and
MIUR-PRIN 2022FPZEES {\em Stability in Hamiltonian dynamics and beyond}.

\section*{Statements and declarations} 

The authors have no competing interests to declare that are relevant to the content of this article.

\section*{Data availability statement} 

The code to reproduce the CAPs presented in our paper is available in \cite{KevinCode}.

\appendix

\section{Geometric preliminaries}
\label{s:geometry}

The phase space $M$ of the system is the open subset of  $(S^2)^N=S^2\times \dots \times S^2$ ($N$ copies) obtained by
removing the collision configurations.
The component $v_j$ of the element $v=(v_1,\dots, v_N)\in M$ specifies the position of the 
$j^{th}$ vortex.  Throughout the paper we represent $(S^2)^N$ as the embedded submanifold
in $(S^2)^N$ specified by  $(S^2)^N:=R^{-1}(0)$ where $R:=(R_1,\dots, R_N):(\R^3)^N\to \R^N$,
\begin{equation*}
R_j: (\R^{3})^N \to \R, \qquad (x_1,\dots, x_N)\mapsto \frac{1}{2}(\|x_j\|^2-1).
\end{equation*}
This embedding induces a convenient  identification of tangent vectors to $M$ with vectors in $\R^{3N}$.
Specifically, if  $v=(v_1,\dots, v_N)\in M$ we identify
\begin{equation*}
T_vM=V_v:=\{ \, {\bf w}=(w_1,\dots, w_N) \in (\R^{3})^N \, : \, v_j\cdot w_j=0 \, \},
\end{equation*}
where $\cdot$ is the euclidean scalar product in $\R^3$.

Let  $f:M \to \R$. The differential of $f$ at $v\in M$ is the linear map $df(v):V_v\to \R$ 
defined intrinsically by 
\begin{equation*}
df(v)({\bf w})= L_Xf  (v), \qquad {\bf w}\in V_v,
\end{equation*}
where $X$ is any locally defined vector field on $M$ satisfying 
$X(v)={\bf w}$ and $L_X$ denotes the Lie derivative along $X$. Note that $df(v)\in V_v^*$. 
If $\tilde f: (\R^{3})^N \to \R$ is any smooth extension of $f$, then
\begin{equation*}
df(v)({\bf w})= \langle \nabla_x \tilde f (v) ,   {\bf w} \rangle,
\end{equation*}
where $\langle\cdot  ,  \cdot \rangle$ is the euclidean inner product in $(\R^3)^N$ and $\nabla_x \tilde f (v)$ is 
the standard gradient of $\tilde f$ evaluated at $v\in (\R^3)^N$.

A point $a\in M$ is a critical point of $f$ if $df(a)=0$. Hence, 
 $a$ is a critical point of 
$f$ if and only if for any smooth extension $\tilde f: (\R^{3})^N \to \R$  of $f$ one has
\begin{equation*}
\langle \nabla_x \tilde f (a) ,   {\bf w} \rangle =0  \qquad \forall  {\bf w} \in V_a.
\end{equation*}
In other words, $ \nabla_x \tilde f (a)\in V_a^\perp$.
Considering that $V_a^\perp=\mathrm{span}\{ \nabla_xR_i(a) \, :\, i=1,\dots, N\}$, we conclude that 
if $a\in M$ is a critical point of $f$ and $\tilde f$ is an extension then:
\begin{enumerate}
\item[(a)]  There  
exist Lagrange multipliers $c_1,\dots, c_N\in \R$ such that
\begin{equation*}
 \nabla_x \tilde f (a)+ \sum_{i=1}^N c_i \nabla_x R_i(a)=0.
\end{equation*}
\item[(b)]  Define the function $\tilde f^*: (\R^{3})^N\times \R^N\to \R$  by
\begin{equation*}
\tilde f^*  \left ( x, \lambda \right ) := \tilde f(x)+
\sum_{i=1}^N\lambda_i R_i(x), \qquad x=(x_1,\dots, x_N)\in  (\R^{3})^N, \; \lambda=(\lambda_1,\dots,
\lambda_N)\in \R^N.
\end{equation*}
Then $\tilde f^*$ has a critical point at $(a,c)$, with $c=(c_1,\dots, c_N)$.
\end{enumerate}

Continue to assume that $a\in M$ is a critical point of $f:M\to \R$. The second derivative of
$f$ at $a$ is the symmetric bilinear form $d^2f(a):V_a\times V_a \to \R$ defined intrinsically by 
\begin{equation*}
d^2f(a)({\bf w}_1, {\bf w}_2)=L_{X_1} L_{X_2} f (a)=L_{X_2} L_{X_1} f (a),
\end{equation*}
where $X_1$ and $X_2$ are any locally defined vector fields on $M$ satisfying 
$X_1(a)={\bf w}_1$, $X_2(a)={\bf w}_2$.
With the definitions introduced above we have 
\begin{equation*}
 d^2f(a)({\bf w}_1, {\bf w}_2)= {\bf w}_1^T   \nabla^2_x \tilde f^* (a,c)
 {\bf w}_2 , \qquad \forall {\bf w}_1, {\bf w}_2 \in V_a,
\end{equation*}
where $ \nabla^2_x \tilde f^* (a,c)$  denotes the standard Hessian matrix
of $x\mapsto \tilde f^*(x,c)$ evaluated at $x=a$.

\section{Computation of the gradient and Hessian matrix of the reduced
Hamiltonian}

\label{App:C}

Below we give explicit expressions for the gradient and Hessian matrix of
the reduced Hamiltonian $h$ given by \eqref{eq:h(u)} interpreted as function 
$h:\R^{3n}\to \R$. Throughout this appendix we simplify the notation and
write  $g:=g_{m}\in SO(3)$ and denote by $I$ the $3\times 3$ identity matrix.

\subsection{Gradient $\protect\nabla_u h(u)$.}

The gradient of the reduced Hamiltonian $h$ has components 
\begin{equation}  \label{eq:gradh}
\nabla _{u_{j}}h(u)=-m\sum_{i=1}^{m-1}\frac{u_{j}-g^{i}u_{j}}{\left\|
u_{j}-g^{i}u_{j}\right\| ^{2}}-m\sum_{j^{\prime }=1(j^{\prime }\neq
j)}^{n}\sum_{i=1}^{m}\frac{u_{j}-g^{i}u_{j^{\prime }}}{\left\|
u_{j}-g^{i}u_{j^{\prime }}\right\| ^{2}}-m\sum_{f\in F}\frac{u_{j}-f}{%
\left\| u_{j}-f\right\| ^{2}}.
\end{equation}

To see this, write $M:=I-g^{i}$ and compute: 
\[
\nabla _{u}\ln \left\| Mu\right\| ^{2}=\frac{1}{\left\| Mu\right\| ^{2}}%
\nabla _{u}\left\langle Mu,Mu\right\rangle =\frac{1}{\left\| Mu\right\| ^{2}}%
\nabla _{u}\left\langle M^{T}Mu,u\right\rangle =\frac{1}{\left\| Mu\right\|
^{2}}\left( 2M^{T}M\right) u. 
\]%
Since $M^{T}M=(I-g^{-i})+(I-g^{i})$, we conclude%
\[
\nabla _{u_{j}}\sum_{i=1}^{m-1}\ln \left\| u_{j}-g^{i}u_{j}\right\|
^{2}=2\sum_{i=1}^{m-1}\frac{(I-g^{-i})u+(I-g^{i})u}{\left\| \left(
I-g^{i}\right) u\right\| ^{2}}=4\sum_{i=1}^{m-1}\frac{(I-g^{i})u}{\left\|
\left( I-g^{i}\right) u\right\| ^{2}}. 
\]%
Combining these identities with \eqref{eq:h(u)} proves \eqref{eq:gradh}.

\subsection{Hessian matrix $\protect\nabla^2h(u)$.}

We now describe the Hessian matrix $\nabla _{u}^{2}h(u)$ in terms of the $%
3\times 3$ blocks $D_{u_{j^{\prime }}}\nabla _{u_{j}}h(u)$, for the indices $%
j,j^{\prime }\in \{1,\dots, n\}$.

\subsubsection*{Case $j^{\prime }\neq j$.}

In this case we have:

\begin{equation}  \label{eq:hessh1}
D_{u_{j^{\prime }}}\nabla _{u_{j}}h(u)=m\sum_{i=1}^{m}\frac{1}{\left\|
u_{j}-g^{i}u_{j^{\prime }}\right\| ^{4}}\left( g^{i}\left\|
u_{j}-g^{i}u_{j^{\prime }}\right\| ^{2}-2\left( u_{j}-g^{i}u_{j^{\prime
}}\right) \left( u_{j}-g^{i}u_{j^{\prime }}\right) ^{T}g^{i}\right).
\end{equation}

To prove \eqref{eq:hessh1}, start with \eqref{eq:gradh} and use the
condition $j^{\prime }\neq j$ to obtain 
\[
\begin{split}
D_{u_{j^{\prime }}}\nabla _{u_{j}}h(u)& =m\sum_{i=1}^{m}g^{i}D_{u_{j^{\prime
}}}\frac{u_{j^{\prime }}-g^{-i}u_{j}}{\left\Vert u_{j^{\prime
}}-g^{-i}u_{j}\right\Vert ^{2}}=m\sum_{i=1}^{m}g^{-i}D_{u_{j^{\prime }}}%
\frac{u_{j^{\prime }}-g^{i}u_{j}}{\left\Vert u_{j^{\prime
}}-g^{i}u_{j}\right\Vert ^{2}} \\
& =m\sum_{i=1}^{m}g^{-i}\left( \frac{I}{\left\Vert u_{j^{\prime
}}-g^{i}u_{j}\right\Vert ^{2}}-\left\Vert u_{j^{\prime
}}-g^{i}u_{j}\right\Vert ^{-4}\left( u_{j^{\prime }}-g^{i}u_{j}\right)
\left( D_{u_{j^{\prime }}}\left\Vert u_{j^{\prime }}-g^{i}u_{j}\right\Vert
^{2}\right) \right) .
\end{split}%
\]%
Since $D_{u_{j^{\prime }}}\left\Vert u_{j^{\prime }}-g^{i}u_{j}\right\Vert
^{2}=2\left( u_{j^{\prime }}-g^{i}u_{j}\right) ^{T}$, we may rewrite the
above expression as%
\begin{eqnarray*}
D_{u_{j^{\prime }}}\nabla _{u_{j}}h(u) &=&m\sum_{i=1}^{m}g^{-i}\frac{1}{%
\left\Vert u_{j^{\prime }}-g^{i}u_{j}\right\Vert ^{4}}\left( I\left\Vert
u_{j^{\prime }}-g^{i}u_{j}\right\Vert ^{2}-2\left( u_{j^{\prime
}}-g^{i}u_{j}\right) \left( u_{j^{\prime }}-g^{i}u_{j}\right) ^{T}\right) ,
\\
&=&m\sum_{i=1}^{m}g^{i}\frac{1}{\left\Vert u_{j}-g^{i}u_{j^{\prime
}}\right\Vert ^{4}}\left( I\left\Vert u_{j}-g^{i}u_{j^{\prime }}\right\Vert
^{2}-2\left( u_{j^{\prime }}-g^{-i}u_{j}\right) \left( u_{j^{\prime
}}-g^{-i}u_{j}\right) ^{T}\right) , \\
&=&m\sum_{i=1}^{m}g^{i}\frac{1}{\left\Vert u_{j}-g^{i}u_{j^{\prime
}}\right\Vert ^{4}}\left( I\left\Vert u_{j}-g^{i}u_{j^{\prime }}\right\Vert
^{2}-2g^{-i}\left( g^{i}u_{j^{\prime }}-u_{j}\right) \left(
g^{i}u_{j^{\prime }}-u_{j}\right) ^{T}g^{i}\right) ,
\end{eqnarray*}%
which simplifies to \eqref{eq:hessh1}.

\subsubsection*{Case $j^{\prime }= j$.}

In this case we have:

\[
D_{u_{j}}\nabla _{u_{j}}h(u)=\mathcal{A}_{1}(u)+\mathcal{A}_{2}(u)+\mathcal{A%
}_{3}(u),
\]%
where the $3\times 3$ matrices $\mathcal{A}_{j}(u)$ arise from the
differentiation of the three terms in \eqref{eq:gradh}. We prove below that
these are given by 
\begin{equation}
\mathcal{A}_{1}(u)=-m\sum_{i=1}^{m-1}\left( \frac{(I-g^{i})}{\left\Vert
u_{j}-g^{i}u_{j}\right\Vert ^{2}}+\frac{2}{\left\Vert
u_{j}-g^{i}u_{j}\right\Vert ^{4}}\left( u_{j}-g^{i}u_{j}\right) \left(
u_{j}-g^{i}u_{j}\right) ^{T}(I-g^{i})\right) ,  \label{eq:A1}
\end{equation}%
\begin{equation}
\mathcal{A}_{2}(u)=-m\sum_{j^{\prime }=1(j^{\prime }\neq
j)}^{n}\sum_{i=1}^{m}\left( \frac{I}{\left\Vert u_{j}-g^{i}u_{j^{\prime
}}\right\Vert ^{2}}-\frac{2}{\left\Vert u_{j}-g^{i}u_{j^{\prime
}}\right\Vert ^{4}}\left( u_{j}-g^{i}u_{j^{\prime }}\right) \left(
u_{j}-g^{i}u_{j^{\prime }}\right) ^{T}\right) ,  \label{eq:A2}
\end{equation}%
\begin{equation}
\mathcal{A}_{3}(u)=-m\sum_{f\in F}\left( \frac{I}{\left\Vert
u_{j}-f\right\Vert ^{2}}-\frac{2}{\left\Vert u_{j}-f\right\Vert ^{4}}\left(
u_{j}-f\right) \left( u_{j}-f\right) ^{T}\right) .  \label{eq:A3}
\end{equation}

The validity of \eqref{eq:A2} and \eqref{eq:A3} follows from the identities 
\[
\begin{split}
D_{u_{j}}\frac{u_{j}-g^{i}u_{j^{\prime }}}{\left\Vert
u_{j}-g^{i}u_{j^{\prime }}\right\Vert ^{2}}& =\frac{I}{\left\Vert
u_{j}-g^{i}u_{j^{\prime }}\right\Vert ^{2}}-\frac{2}{\left\Vert
u_{j}-g^{i}u_{j^{\prime }}\right\Vert ^{4}}\left( u_{j}-g^{i}u_{j^{\prime
}}\right) \left( u_{j}-g^{i}u_{j^{\prime }}\right) ^{T}, \\
D_{u_{j}}\frac{u_{j}-f}{\left\Vert u_{j}-f\right\Vert ^{2}}& =\frac{I}{%
\left\Vert u_{j}-f\right\Vert ^{2}}-\frac{2}{\left\Vert u_{j}-f\right\Vert
^{4}}\left( u_{j}-f\right) \left( u_{j}-f\right) ^{T},
\end{split}%
\]%
which are deduced proceeding as above. For \eqref{eq:A1}, we write $M=I-g^{i}
$ as before and write 
\[
D_{u_{j}}\frac{u_{j}-g^{i}u_{j}}{\left\Vert u_{j}-g^{i}u_{j}\right\Vert ^{2}}%
=D_{u_{j}}\frac{Mu_{j}}{\left\Vert Mu_{j}\right\Vert ^{2}}=\frac{M}{%
\left\Vert Mu_{j}\right\Vert ^{2}}+Mu_{j}D_{u_{j}}\left( \left\Vert
Mu_{j}\right\Vert ^{2}\right) ^{-1}.
\]%
We have that $D_{u_{j}}\left( \left\Vert Mu_{j}\right\Vert ^{2}\right)
^{-1}=\left( \left\Vert Mu_{j}\right\Vert ^{2}\right)
^{-2}D_{u_{j}}\left\Vert Mu_{j}\right\Vert ^{2}$ and using the computation
of $D_{u_{j}}\left\Vert Mu_{j}\right\Vert ^{2}$ as before we obtain 
\[
D_{u_{j}}\frac{u_{j}-g^{i}u_{j}}{\left\Vert u_{j}-g^{i}u_{j}\right\Vert ^{2}}%
=\frac{M}{\left\Vert Mu_{j}\right\Vert ^{2}}+\frac{2}{\left\Vert
Mu_{j}\right\Vert ^{4}}Mu_{j}\left( Mu_{j}\right) ^{T}M.
\]%
Therefore, 
\[
D_{u_{j}}\frac{u_{j}-g^{i}u_{j}}{\left\Vert u_{j}-g^{i}u_{j}\right\Vert ^{2}}%
=\frac{I-g^{i}}{\left\Vert u_{j}-g^{i}u_{j}\right\Vert ^{2}}+\frac{2}{%
\left\Vert u_{j}-g^{i}u_{j}\right\Vert ^{4}}\left( u_{j}-g^{i}u_{j}\right)
\left( u_{j}-g^{i}u_{j}\right) ^{T}(I-g^{i}),
\]%
proving the validity of \eqref{eq:A1}.

\section{General non-symmetric symplectic slice}
\label{ss:NonSymmSympSlice}

Here we  construct a symplectic slice for a RE $(a,\omega)\in M\times \R$ which does not assume any symmetry.
We assume that both $\mu=\Phi(a)$ and $\omega$ are non-zero and that all vortices have equal strengths and $N\geq 3$.
Write  $a=(a_1,\dots, a_N)\in M$ and suppose that the arrangement of the vortices
 is such that $a_1$ and $a_2$ do not have equal nor opposite
latitudes, and are not poles (otherwise do a suitable permutation). We work with 
the identification $T_aM=V_a$ given by  \eqref{eq:tangent-space}.

Define the vectors $b_j, c_j\in \R^3$, $j=1,\dots, N$, as follows:
\begin{equation*}
\begin{split}
b_j:=\begin{cases} e_3\times a_j \quad \mbox{if $a_j$ is not a pole,} \\ e_1 \qquad \mbox{if $a_j$ is a pole,}\end{cases}\qquad 
c_j:=a_j\times b_j.
\end{split}
\end{equation*}

\begin{proposition}
\label{prop:slice-no-symmetry}
Under the condition that $a_1$ and $a_2$ do not have equal nor opposite
latitudes, and are not poles, the set $\{{\bf u}^{(j)}, {\bf v}^{(j)}\}_{j=1}^{N-1}$, whose elements are defined below, forms
a basis of a symplectic slice.
\begin{equation}
\label{eq:asymmetricSlice}
\begin{split}
{\bf u}^{(1)}&=\left ( a_1\times (b_2\times c_2), - a_1\times (b_2\times c_2), 0, \dots, 0\right ),\\
{\bf u}^{(j)}&=\left ( a_1\times (b_2\times b_{j+2}), - ( a_1\cdot  b_{j+2})  b_2,  0, \dots, 0, 
\stackrel{(j+2)^{th}}{ ( a_1\cdot b_2 ) b_{j+2}}, 0, \dots, 0 \right ), \quad j=2, \dots, N-2,\\
{\bf v}^{(j)}&=\left ( a_1\times (b_2\times c_{j+2}), - ( a_1\cdot  c_{j+2})  b_2,  0, \dots, 0,  \stackrel{(j+2)^{th}}{( a_1\cdot b_2)  c_{j+2}}, 0, \dots, 0 \right ), \quad j=1, \dots, N-2.
\end{split}
\end{equation}
\end{proposition}
\begin{proof}
First note that the condition that $a_1$ and $a_2$ do not have opposite
latitudes, and are not poles guarantees that the scalar product $a_1\cdot b_2\neq 0$.
Consider the following $2(N-2)+1$ vectors in $(\R^3)^N$ 
\begin{equation*}
\begin{split}
\boldsymbol{ \gamma}&=\left ( a_1\times (b_2\times c_2), - a_1\times (b_2\times c_2), 0, \dots, 0\right ),\\
\alpha^{(j)}&=\left ( a_1\times (b_2\times b_{j+2}), - ( a_1\cdot  b_{j+2} ) b_2,  0, \dots, 0, 
\stackrel{(j+2)^{th}}{ ( a_1\cdot b_2 ) b_{j+2}}, 0, \dots, 0 \right ), \quad j=1, \dots, N-2,\\
\beta^{(j)}&=\left ( a_1\times (b_2\times c_{j+2}), - ( a_1\cdot   c_{j+2} ) b_2,  0, \dots, 0,  \stackrel{(j+2)^{th}}{( a_1\cdot b_2 ) c_{j+2}}, 0, \dots, 0 \right ), \quad j=1, \dots, N-2.
\end{split}
\end{equation*}
Using that $\{ a_j, b_j, c_j\}$ is an orthogonal basis of $\R^3$ for all $j=1,\dots, N$, one proves that the
above vectors belong to $V_a$ defined by  \eqref{eq:tangent-space}. Moreover, 
  a straightforward calculation using \eqref{eq:kerdPhi}, shows that these vectors actually belong to
$\ker d\Phi (a)$.  We wish to show that they  form a basis for $\ker d\Phi (a)$. Considering that the dimension of $\ker d\Phi (a)$ is $2N-3$
then it suffices to prove their linear independence.
Take a 
linear combination of them equal to zero:
\begin{equation*}
\lambda_0\gamma+ \sum_{k=1}^{N-2}( \lambda_k \alpha^{(k)}+\mu_k \beta^{(k)})=0.
\end{equation*}
For  $1\leq j \leq N-2$ the $(j+2)^{th}$ entry of the above vector is 
$( a_1\cdot b_2)(\lambda_j b_{j+2}+\mu_j c_{j+2})$. Given that  $a_1\cdot b_2\neq 0$ and considering that $b_j$ and $c_j$ are perpendicular,
this vector can only equal zero if $\lambda_j=\mu_j=0$. Therefore we are left with $\lambda_0\gamma=0$
which leads to $\lambda_0=0$, proving our claim.

Now we claim that  the infinitesimal generator of the orbit, $s_a$, is a linear combination
of the vectors $ \alpha^{(k)}$, $k=1, \dots, N-2$. To see this, first assume that there are no poles so that 
\begin{equation*}
b_j=J_3a_j=e_3\times a_j, \quad \mbox{for all} \quad 1\leq j\leq N, \qquad s_a=(b_1, \dots, b_N). 
\end{equation*}
On the 
other hand we have
  \begin{equation*}
\frac{1}{ a_1\cdot b_2 }\sum_{j=1}^N\alpha^{(j)}=\left ( \frac{1}{a_1\cdot b_2} 
a_1\times \left (b_2 \times\sum_{j=3}^Nb_j \right  ),
-\frac{a_1\cdot \left (\sum_{j=3}^Nb_j \right )}{a_1\cdot b_2}\, b_2, b_3, \dots, b_N \right ).
\end{equation*}
Now note that by Proposition \ref{prop:RE-as-crit-pts}(ii) we have  $\sum_{j=1}^Nb_j=e_3\times \sum_{j=1}^Na_j=e_3\times \Phi(a)=0$, so 
we may write 
\begin{equation*}
\begin{split}
 \frac{1}{a_1\cdot b_2} 
a_1\times \left (b_2 \times\sum_{j=3}^Nb_j \right  )&=  \frac{1}{a_1\cdot b_2}  a_1\times \left (b_2 \times(-b_1-b_2) \right  )\\
&= \frac{-1}{a_1\cdot b_2} a_1\times \left (b_2 \times b_1 \right  )=b_1,
\end{split}
\end{equation*}
and also
\begin{equation*}
-\frac{a_1\cdot \left (\sum_{j=3}^Nb_j \right )}{a_1\cdot b_2}=
\frac{ a_1\cdot b_1+ a_1\cdot b_2 }{a_1\cdot b_2}=1.
\end{equation*}
Therefore 
 \begin{equation*}
\frac{1}{a_1\cdot b_2}\sum_{j=1}^N\alpha^{(j)}=s_a,
\end{equation*}
and $s_a$ is indeed a linear combination of  the vectors $ \alpha^{(k)}$, $k=1, \dots, N-2$.

In the presence of $p$ poles ($p=1,2$), the above calculation gets modified
and one instead shows that 
 \begin{equation*}
\frac{1}{ a_1\cdot b_2 }\sum_{j=1}^{N-p}\alpha^{(j)}=s_a,
\end{equation*}
using that  $\sum_{j=1}^{N-p}b_j=0$. (We use the convention that the poles are at the
end of the array $(a_1, \dots, a_N)$).

Considering that $s_a$ is a linear combination of $\alpha^{(j)}$, $j=1, \dots, N-2$
and that the vectors $\{\gamma, \alpha^{(j)}, \beta^{(j)}\}$ are a basis of $\ker d\Phi (a)$ we conclude that the
subspace $\mathcal{N}_a$ generated by:
\begin{equation*}
\gamma, \quad \alpha^{(j)}, \, j=2, \dots, N-2, \quad \beta^{(j)}, \, j=1, \dots, N-2,
\end{equation*}
is a $2(N-2)$-dimensional subspace of $V_a$ which contains $\ker d\Phi (a)$ and is transversal to $s_a$ and may therefore be
taken as a symplectic slice. The vectors ${\bf u}^{(j)}$, ${\bf v}^{(j)}$ in the statement of the proposition are precisely taken as:
\begin{equation*}
\begin{split}
{\bf u}^{(1)}=\gamma, \qquad  {\bf u}^{(j)}=\alpha^{(j)}, \, j=2, \dots, N-2, \qquad {\bf v}^{(j)}=\beta^{(j)}, \, 
j=1, \dots, N-2.
\end{split}
\end{equation*}

\end{proof}

\section{Proof of Lemma \ref{eq:lemma-complexgeneral}}
\label{app:proof-complex-lemma}

\begin{proof} It is obvious that $\mathcal{Q}$ is Hermitian from its definition and in view of
\eqref{eq:complex-structure-ids}. To prove 
the statement about the indices of inertia and the dimension of the null spaces
let $\{e_1,\dots, e_d\}$ be the canonical basis of $\C^d$ and consider $\psi:V\to \C^d$ 
defined by $\psi(\alpha_j):=e_j$, $\psi(\beta_j):=-ie_j$, $j=1\dots, d$. Then $\psi$ is a
vector space isomorphism between $V$ and $ \C^d_\R$, where $ \C^d_\R$ denotes
the set  $ \C^d$ endowed with the  vector space structure over  $\R$. Let  $\tilde f:\C^d\times \C^d\to \C$
be the Hermitian form whose matrix wrt the canonical basis of $\C^d$ is $\mathcal{Q}$, namely
\begin{equation*}
\tilde f(z_1, z_2)=\bar z_1^T \mathcal{Q}z_2, \qquad z_1,z_2 \in \C^d.
\end{equation*}
The
key point of the proof is to notice that 
  \begin{equation}
\label{eq:aux-complex-lemma1}
f(w_1,w_2)=\frac{1}{2}\mbox{Re}\left ( \tilde f (\psi(w_1), \psi(w_2))\right ), \quad \mbox{for all $w_1,w_2\in V$.}
\end{equation}
This  is a consequence of linearity of $f$, sesquilinearity\footnote{note that by definition of $\tilde f$ we
have  
$\tilde f( \lambda_1  z_1, \lambda_2 z_2)=\bar \lambda_1 \lambda_2 \tilde f(  z_1,  z_2)$ for
$\lambda_1,\lambda_2\in \C$.} of $\tilde f$, and  the identities
\begin{equation*}
f(\alpha_j,\alpha_k)=f(\beta_j,\beta_k)=\frac{1}{2}\mbox{Re}\left (\mathcal{Q}_{jk} \right ), \qquad 
f(\alpha_j,\beta_k)=-f(\alpha_k,\beta_j)=\frac{1}{2}\mbox{Im}\left (\mathcal{Q}_{jk} \right ), 
\end{equation*}
which hold $\forall j,k=1,\dots, d$, because of the hypothesis \eqref{eq:complex-structure-ids}. In particular
we have
\begin{equation}
\label{eq:aux-complex-lemma0}
f(w,w)= \frac{1}{2}\tilde f (\psi(w), \psi(w)), \quad \mbox{for all $w\in V$.}
\end{equation}
%
%

Now let $\{s_1,\dots, s_d \}$   be a basis of $\C^d$ which diagonalises $\tilde f$, i.e. $\tilde f(s_i,s_j)=0$ if $i\neq j$, 
and suppose that
\begin{equation}
\label{eq:aux-complex-lemma2}
\begin{split}
&\tilde f(s_j,s_j)>0, \quad 1\leq j\leq d_+, \qquad \tilde f(s_j,s_j)<0,\quad d_++1\leq j\leq d_++d_-, \\
& \tilde f(s_j,s_j)=0,\quad d_++d_-+1\leq j\leq d_++d_-+d_0,
\end{split}
\end{equation}
 so that $i_+(\tilde f)=d_+$,  $i_-(\tilde f)=d_-$, $\dim \ker \tilde f=d_0$, and $d=d_++d_-+d_0$.
 Then $\{s_1,,\dots, s_d,is_1,\dots, is_d \}$ is a basis of $ \C^d_\R$ and, since $\psi$ is an 
 isomorphism, $\{y_1,\dots, y_{2d}\}$ defined by 
 \begin{equation*}
y_j:=\psi^{-1} (s_j),  \quad y_{m+j}:=\psi^{-1} (is_j),\qquad 1\leq j\leq d,
\end{equation*}
is a basis of $V$.  Moreover, since 
$\{s_1,\dots, s_m \}$ diagonalises $\tilde f$, it is easy to show, using  
 \eqref{eq:aux-complex-lemma1}, that  $\{y_1,\dots, y_{2m}\}$ diagonalises $f$, i.e. $f(y_i,y_j)=0$
 for all $i\neq j$. Using \eqref{eq:aux-complex-lemma0} and sesquilinearity of $\tilde f$ it follows
 that 
 \begin{equation*}
 f(y_j,y_j)= f(y_{m+j},y_{m+j})= \frac{1}{2}\tilde f(s_j,s_j), \qquad j=1,\dots, m,
\end{equation*}
which by \eqref{eq:aux-complex-lemma2} imply  $i_+( f)=2d_+$,  $i_-( f)=2d_-$
and  $\dim \ker  f=2d_0$.
\end{proof}

\section{Approximate coordinates of  RE near total collision}
\label{App:NumCollision}

An approximation of the coordinates of the RE near total collision reported in section \ref{ss:RE-totalcollision} 
is as follows. For all of them $\omega=50$. 

For $N=10$ one has
\begin{equation}
\label{eq:RELarge10}
\begin{split}
a_1&=(-0.321250364476975,   0.125503906002515,   0.938641024514443), \\
a_2&=(-0.281614324121647,  -0.177060196674640,   0.943049871005264), \\
a_3&=(-0.110832315744048,  0.301948550025117,   0.946859689143297), \\
a_4&=(0.329289157466230,   0.047176175895631,   0.943049871005264), \\
a_5&=(-0.056029765308738,  -0.338564275556233,   0.939273600564037), \\
a_6&=(0.163769131776852,   0.303533686063892,   0.938641024514443), \\
a_7&=(0.055171327848398,  -0.150307266747506,   0.987098703345485), \\
a_8&=(0.093915265535320,   0.096705006907256,   0.990872375504786), \\
a_9&=(-0.134176736522985,   0.012982250865822,   0.990872375504786), \\
a_{10}&=(0.261758623547594,  -0.221917836781855,   0.939273600564037).
\end{split}
\end{equation}
For $N=11$,
\begin{equation}
\label{eq:RELarge11}
\begin{split}
a_1&=(0.139326894549961,   0.025868279023347,   0.989908505163702), \\
a_2&=(  -0.023823734155396,  -0.359048230308640,   0.933014896988857), \\
a_3&=(  -0.233002228449082,  0.278282639870440,   0.931809387098295), \\
a_4&=(   0.216459904805164,  -0.258525569202590,   0.941440194425656), \\
a_5&=(   0.034791923532369,   0.340414859454908,   0.939631441321124), \\
a_6&=(  -0.275029387518199,  -0.219648813034752,   0.936009206650121), \\
a_7&=(  -0.341231039929540,   0.025576001826382,   0.939631441321107), \\
a_8&=(   0.357646782971700,  -0.039648210890842,   0.933014896988869), \\
a_9&=(  -0.089752690493977,   0.107194750077486,   0.990178640501257), \\
a_{10}&=(  -0.049951742661942,  -0.132612121654087,   0.989908505163703), \\
a_{11}&=(   0.264565317348943,   0.232146414838348,   0.936009206650103).
\end{split}
\end{equation}
Finally, for $N=12$, the generators of the $\Z_3$-symmetric RE of type $(n_4,m_3,p_0)$ are
\begin{equation}
\label{eq:RELarge12}
\begin{split}
u_1&=(0.034887632581048,   0.136341626351998,   0.990047379682701),  \\
u_2&=(-0.249324115175042,   0.243756694911171,   0.937240715759919),  \\
u_3&=(0.214399606524508,   0.302490526779084,   0.928726165202127),  \\
u_4&=(-0.042756936922558,   0.368292756396255,   0.928726165202127).
\end{split}
\end{equation}

\end{document}